\documentclass[11pt]{article}
\usepackage{etex}

\usepackage[margin=1in]{geometry}
\usepackage{stackengine}

\usepackage{color}
\usepackage[latin1]{inputenc}
\usepackage[T1]{fontenc}
\usepackage[normalem]{ulem}
\usepackage[english]{babel}
\usepackage{verbatim}
\usepackage{graphicx}
\usepackage{enumerate}
\usepackage{amsmath,amssymb,amsfonts,amsthm,amscd,mathrsfs}
\usepackage{array}

\usepackage{dsfont}

\usepackage[T1]{fontenc}
\usepackage{babel}
\usepackage[bookmarksopen, bookmarksnumbered]{hyperref}

\usepackage{url}
\usepackage{pgfplots}

\usepackage{caption}

\newtheorem{theorem}{Theorem}[section]

\newtheorem{lemma}[theorem]{Lemma}
\newtheorem{prop}[theorem]{Proposition}
\newtheorem{cor}[theorem]{Corollary}

\theoremstyle{definition}
\newtheorem{defn}[theorem]{Definition}
\newtheorem{rem}[theorem]{Remark}

\input xy
\xyoption{all}

\DeclareMathOperator{\Res}{Res}

\newcommand{\im}{\mathbf{i}}
\newcommand{\pb}{\mathrm{Prob}}

\newcommand{\E}{\mathbb E}

\makeatletter
\renewcommand\tableofcontents{%
  \null\hfill\textbf{\Large\contentsname}\hfill\null\par
  \@mkboth{\MakeUppercase\contentsname}{\MakeUppercase\contentsname}%
  \@starttoc{toc}%
}

\g@addto@macro\normalsize{%
  \setlength\abovedisplayskip{5pt}
  \setlength\belowdisplayskip{5pt}
  \setlength\abovedisplayshortskip{3pt}
  \setlength\belowdisplayshortskip{3pt}
}

\makeatother

\numberwithin{equation}{section}

\begin{document}
\title{Interlacing adjacent levels of $\beta$--Jacobi corners processes}

\author{Vadim Gorin
\thanks{Department of Mathematics, Massachusetts Institute of Technology, Cambridge, MA, USA and Institute for Information Transmission Problems of Russian Academy of Sciences, Moscow, Russia.
e-mail: vadicgor@gmail.com} \and Lingfu Zhang
\thanks{Department of Mathematics, Princeton University, Princeton, NJ, USA. e-mail: lingfuz@princeton.edu}
}
\date{}

\maketitle

\begin{abstract}
 We study the asymptotics of the global fluctuations for the difference between two
 adjacent levels in the $\beta$--Jacobi corners process (multilevel and general
 $\beta$ extension of the classical Jacobi ensemble of random matrices). The limit
 is identified with the derivative of the $2d$ Gaussian Free Field. Our main tools
 are integral forms for the (Macdonald-type) difference operators originating from
 the shuffle algebra.
\end{abstract}

\tableofcontents

\section{Introduction}

A typical setup in random matrix theory is to take an $N\times N$ Hermitian
matrix $X$ and to study its eigenvalues $x_1<\dots<x_N$ as $N\to\infty$. One
observable of interest is the linear statistic
\begin{equation}
\mathfrak{S}_{f}:=\sum_{i=1}^N f(x_i),
\end{equation}
for suitable
(usually smooth) functions $f$. In many cases $\frac{1}{N} \mathfrak{S}_{f}$
converges to a constant as $N\to\infty$, while $\mathfrak{S}_f-\E \mathfrak{S}_f$ is
asymptotically Gaussian, see e.g.\ the textbooks \cite{anderson2010introduction},
\cite{Pastur_Shcherbina}, \cite{Forrester}.  Such limit results are usually referred
to as the Law of Large Numbers and the Central Limit Theorem  for the \emph{global
fluctuations} of $X$.

The asymptotic covariance for $\mathfrak{S}_f$ is best understood through the
\emph{corners processes} --- a 2d extension obtained by looking at the joint
distribution of the eigenvalues of all principal corners of $X$. In more details,
let $x^k_1<x^k_2<\dots<x^k_k$ be the eigenvalues of the $k\times k$ top--left corner of
$X$, $k=1,2,\dots,N$. The global fluctuations of the array $\{x_i^j\}_{1\le i \le j
\le N}$ as $N\to\infty$ can be then described by a pullback of the 2--dimensional
Gaussian Free Field, as was proven in \cite{Borodin_CLT}, \cite{borodin2015general},
\cite{Dum_Paq}, \cite{Johnson_Pal}, \cite{Gan_Pal}
 for numerous ensembles of random matrix theory: Wigner,
Wishart, $\beta$--Jacobi, and adjacency matrices of random graphs.

The corners processes also pave a way for a sequential construction of $\mathfrak S_f$. Define
\begin{equation}
 \mathfrak{S}'_f(k)=\left(\sum_{i=1}^k f ( x^k_i) \right)-\left(\sum_{i=1}^{k-1} f ( x^{k-1}_i) \right),
\end{equation}
then clearly $\mathfrak{S}_f=\sum_{k=1}^N \mathfrak{S}'_f(k)$. The aforementioned
Central Limit Theorems for the corners processes imply that for any
$0<\alpha_1<\dots<\alpha_m<1$, and smooth functions $f_1,\dots,f_m$, the
$m$--dimensional vector
\begin{equation} \label{eq_der_sums}
 \sum_{k=1}^{\lfloor L \alpha_i \rfloor} (\mathfrak{S}'_{f_i}(k)- \E \mathfrak{S}'_{f_i} (k)),\quad i=1,\cdots,m,
\end{equation}
is asymptotically Gaussian as $L \rightarrow \infty$, and its covariance can be identified with the joint
covariance of certain integrals of the Gaussian Free Field.

A different approach of understanding \eqref{eq_der_sums} is to analyze the asymptotics of the joint law of its
individual terms, $\mathfrak{S}'_f(k)-\E \mathfrak{S}'_f(k)$. As far as the authors
know, such analysis in the setting of the Central Limit Theorem escaped the
attention up until recently and it is the main topic of the present article. Let us
however note that in the Law of Large Numbers context, $\mathfrak{S}'_f(k)$ was
previously considered for Wigner matrices by Kerov
\cite{Kerov93transitionprobabilities}, \cite{kerov1993root},
\cite{kerov1998interlacing} and Bufetov \cite{Bufetov_2013}, leading to interesting
connections with orthogonal polynomials and random partitions.
\bigskip

The stochastic system that we work with is the $\beta$--Jacobi corners process,
first introduced in \cite{borodin2015general}. This is a random array of particles
split into levels, and such that the distribution of particles at level $N$ can be
identified with the classical $\beta$--Jacobi ensemble of random $N \times N$ matrices. The
special values of the parameter $\beta=1,2,4$ arise when considering real, complex
or quaternion matrices, and for such values of the parameter the Jacobi corners
process can be identified with the eigenvalues of the MANOVA ensemble
$A^*A(A^*A+B^*B)^{-1}$ with rectangular Gaussian matrices $A$ and $B$ that vary with
$N$, see \cite[Section 1.5]{borodin2015general} and \cite{sun2015matrix} for the
details. More generally, extrapolating from the $\beta=1,2,4$ cases, the definition also
makes sense for any value of $\beta>0$, see Section \ref{Section_def} for the
details.

In the $\beta$--Jacobi corners process, suppose that the particles at level $k$ are $x_1^k < x_2^k < \cdots x_k^k$, then we define $\mathfrak{S}_f(k) = \sum_{i=1}^k f(x_i^k)$, and $\mathfrak{S}'_f(k) = \mathfrak{S}_f(k) - \mathfrak{S}_f(k-1)$.
We study $\mathfrak{S}'_f(k)$ in two separate asymptotic regimes: for individual
$k=\lfloor y L \rfloor$ as $L\to\infty$, and in the integrated form by averaging
$\mathfrak{S}'_f(\lfloor y L \rfloor)$ with a smooth weight function on $y$.
These two regimes have very different behaviors.
For the first one, $\mathfrak{S}'_f(\lfloor y L \rfloor)$ converges as $L\to\infty$
to a constant (depending on $f$, see Theorem \ref{thm:mea}), while $\mathfrak{S}'_f(\lfloor y L \rfloor)-\E \mathfrak{S}'_f(\lfloor y L \rfloor)$ decays as
$L^{-1/2}$ and becomes asymptotically Gaussian upon rescaling. Somewhat
surprisingly, for $y_1 \neq y_2$ the random variables $L^{1/2}\left(\mathfrak{S}'_f(\lfloor y_1 L \rfloor)-\E \mathfrak{S}'_f(\lfloor y_1 L \rfloor)\right)$ and
$L^{1/2}\left(\mathfrak{S}'_f(\lfloor y_2 L \rfloor)-\E \mathfrak{S}'_f(\lfloor y_2 L
\rfloor)\right)$ are asymptotically independent, see Theorem \ref{thm:clt1} for the
exact statement and details.
The scaling is also different for the second limit regime, as the weighted averages of
the form
\begin{equation}
 \int_{0}^1 g(y) \left( \mathfrak{S}'_f(\lfloor y L \rfloor)-\E \mathfrak{S}'_f(\lfloor y L
\rfloor) \right) dy
\end{equation}
decay as $L^{-1}$ and become Gaussian upon rescaling, see Theorem \ref{thm:clt2} for
the exact statement and details.

The results in both scalings are best understood if we recall the main theorem of
\cite{borodin2015general}: $\mathfrak{S}_f(\lfloor y L \rfloor)-\E \mathfrak{S}_f(\lfloor y L \rfloor)$ is asymptotically Gaussian (jointly in several $y$'s and
$f$'s), and the limit can be identified with the integral of a generalized Gaussian
field, which in turn is a pullback of the Gaussian Free Field. Then our results
yield that $\mathfrak{S}'_f(\lfloor y L \rfloor)-\E \mathfrak{S}'_f(\lfloor y L
\rfloor)$ converges to the $y$--derivative of this generalized Gaussian field.
Rigorously speaking, the field is not differentiable in $y$--direction, and this is
what leads to the appearance of two scalings; we make this connection more precise
in Theorems \ref{thm:gff:la}, \ref{thm:gff:sm}.

From this perspective, our main results strengthen the convergence of $\mathfrak{S}_f(\lfloor y L \rfloor)-\E \mathfrak{S}_f(\lfloor y L \rfloor)$ to the pullback of
the Gaussian Free Field up to the convergence of the derivatives in the $y$--direction.
We emphasize that there is no a priori reason why such an upgrade for the CLT should
hold. Indeed, in a parallel work \cite{Erdos_Schroder} Erd\H{o}s and Schr\"oder show
that this is not the case for general Wigner matrices; in that article the limit
might even fail to be Gaussian.

On a more technical side, our result can be linked to an observation that the
Gaussian convergence of $\mathfrak{S}_f(\lfloor y L \rfloor)-\E \mathfrak{S}_f(\lfloor
y L \rfloor)$ is faster than it might have been: the cumulants (of order $3$ and
greater) decay much faster than just $o(1)$, see Proposition \ref{prop:gtp2} for the details. In several $2d$ stochastic systems,
which have no direct connection with our setup, but also lead to the asymptotic
appearance of the Gaussian Free Field, somewhat similar fast decay of cumulants were observed e.g.\ in \cite{Conlon_Spencer}, \cite{Yau_CLT}. 

For the proofs, we adopt parts of the methodology of \cite{borodin2015general} and
exploit the fact that the $\beta$--Jacobi process is a limit of Macdonald processes
of \cite{borodin2014macdonald}, \cite{borodin2016observables}. A critical new
ingredient is the use of a family of difference operators arising from the work of
Negut \cite{Negut}, \cite{Negut_Shuffle} (see also \cite{feigin2009commutative}) on
Macdonald operators and shuffle algebra. These operators were first introduced by
Borodin and the first author in the appendix to \cite{Fyodorov2016} (without
detailed proofs), and here we further develop them to their full power. Let us
emphasize that although asymptotic information on $\beta$--Jacobi process was
previously accessible through other operators used in \cite{borodin2015general},
their combinatorics was too complicated for the delicate asymptotic regimes
addressed in this text. Therefore, our use of a new family of difference operators
is crucial for the proofs.  These operators output answers in the form of contour
integrals in large dimensions, and additional efforts and ideas are required to
convert them into a trackable one or two--dimensional form. This last step is done
by employing new integral identities for dimension reductions, which are
presented in Appendix \ref{sec:dr}.

\bigskip

Independently of the present article, the Central Limit Theorem for $\mathfrak{S}'_f(k)$ in the context of Wigner matrices was also considered recently by
Erd\H{o}s--Schr\"oder \cite{Erdos_Schroder} and by Sodin \cite{Sodin}. Since these
authors consider only the asymptotics for a single $k$, the link to the Gaussian Free Field is less visible there, although we believe that it should be also present (at
least in the case when the Wigner matrices have Gaussian entries, i.e.\ for GOE,
GUE, GSE). Despite the connections, our setup is quite different from \cite{Erdos_Schroder},
\cite{Sodin}. In particular, these papers rely on matrix models and independence of
matrix elements; we do not know how to extend such an approach to our settings of
$\beta$--Jacobi corners process.
In the opposite direction: although there is a
simple and well-known limit from $\beta=1,2,4$ Jacobi corners process to
GOE/GUE/GSE, some technical difficulties prevent us from
performing such a limit transition in the exact
formulas that we use for the asymptotic analysis.

\section*{Acknowledgments}

The authors would like to thank Alexei Borodin and Alexey Bufetov for helpful
discussions. 
We are very grateful to the two anonymous referees, whose feedback led to many important improvements in the text.
V.G.~was partially supported by the NSF grants DMS-1407562, DMS-1664619, and by the
Sloan Research Fellowship. L.Z.~ was partially supported by Undergraduate Research
Opportunity Program (UROP) in Massachusetts Institute of Technology.

\section{Background and setup}

\subsection{$\beta$--Jacobi corners process}

\label{Section_def}

\begin{defn}
The \emph{$K$-particle Jacobi ensemble} is a probability distribution on
$K$-tuples of real numbers $0\leq x_1 < \cdots < x_K \leq 1$ with density (with
respect to Lebesgue measure) proportional to

\begin{equation}
\prod_{1\leq i<j\leq K} (x_j - x_i)^{\beta}\prod_{i=1}^K x_i^p (1-x_i)^q,
\end{equation}
where $\beta > 0, p, q > -1$ are real parameters.
\end{defn}

This is the distribution of the eigenvalues of the MANOVA (Jacobi) random matrix ensemble.
Specifically, consider two infinite matrices $X_{ij}$ and $Y_{ij}$, $i, j = 1, 2,
\cdots$ where entries are i.i.d.\ real, complex, or quaternion Gaussian,
corresponding to $\beta=1,2,4$, respectively. For integers $A\geq M>0$ and $N > 0$,
let $X^{AM}$ be the $A\times M$ top--left corner of $X$, and $Y^{N M }$ the $N\times
M$ top--left corner of $Y$. For the $M\times M$ matrix,
\begin{equation}
\mathcal{M}^{ANM} = (X^{AM})^{*} X^{AM} \left( (X^{AM})^{*} X^{AM} + (Y^{NM})^{*}
Y^{NM} \right)^{-1} ,
\end{equation}
almost surely $K=\min(N,M)$ of its $M$ eigenvalues are different from $0$ and $1$;
they are distributed as $K$-particle Jacobi ensembles, for $\beta = 1, 2, 4$, and $p
= \frac{\beta}{2}(A-M+1) - 1$, $q = \frac{\beta}{2}(|M-N|+1) - 1$, see e.g.\
\cite[Section 3.6]{Forrester}.


\smallskip

Following \cite{borodin2015general}, we further introduce the $\beta$--Jacobi
corners process by coupling a sequence of Jacobi ensembles. Let $\chi^M$ be the set
of infinite families of sequences $x^1, x^2, \cdots$, where for each $N\geq 1$,
$x^N$ is an increasing sequence with length $\min(N, M)$:
\begin{equation}
0\leq x_1^N < \cdots < x_{\min(N, M)}^N \leq 1
\end{equation}
and for each $N>1$, $x^N$ and $x^{N-1}$ interlace:
\begin{equation}
x_1^N < x_1^{N-1} < x_2^N < \cdots .
\end{equation}

\begin{defn}
The \emph{$\beta$--Jacobi corners process} is a random element of $\chi^M$ with
distribution $\mathbb{P}^{\alpha, M, \theta}$, where $\alpha, M, \theta = \frac{\beta}{2} > 0$ are parameters, such that the sequence $x^N$, for
$N=1,2,\dots$, is a Markov chain satisfying the following conditions.
First, the marginal distribution of a single $x^N$ has
density (with respect to Lebesgue measure) proportional to
\begin{equation}  \label{eq:cp:defn1}
\prod_{1\leq i<j\leq \min(N, M)} (x_j^N - x_i^N)^{2\theta}\prod_{i=1}^{\min(N, M)}
(x_i^N)^{\theta\alpha-1} (1-x_i^N)^{\theta(|M-N|+1)-1} .
\end{equation}
Second, the conditional density of $x^{N-1}$ given $x^{N}$ is
\begin{multline}  \label{eq:cp:defn2}
\frac{\Gamma(N\theta)}{\Gamma(\theta)^N}
\prod_{i=1}^{N}(x_i^N)^{(N-1)\theta}\prod_{1\leq i<j \leq N - 1}(x_j^{N-1} - x_i^{N-1})
\prod_{1\leq i<j\leq N}(x_j^N - x_i^N)^{1-2\theta} \\
\times \prod_{i=1}^{N-1}\prod_{j=1}^{N}|x_j^N -
x_i^{N-1}|^{\theta-1}\prod_{i=1}^{N-1}\frac{1}{(x_i^{N-1})^{N\theta}} 
\end{multline}
when $N\leq M$, and
\begin{multline}  \label{eq:cp:defn3}
\frac{\Gamma(N\theta)}{\Gamma(\theta)^M\Gamma(N\theta - M\theta)}
\prod_{1\leq i<j\leq M}(x_j^{N-1} - x_i^{N-1})(x_j^N - x_i^N)^{1-2\theta} \\
\times \prod_{j=1}^{M} (x_i^N)^{(N-1)\theta}(1-x_i^N)^{\theta(M-N-1)+1}
\prod_{i,j=1}^{M}|x_j^N - x_i^{N-1}|^{\theta-1}
\prod_{i=1}^{M}\frac{(1-x_i^{N-1})^{\theta(N-M)-1}}{(x_i^{N-1})^{N\theta}} 
\end{multline}
when $N > M$.
\end{defn}

The proof that the distribution $\mathbb{P}^{\alpha, M, \theta}$ is well-defined (i.e., that the formulas (\ref{eq:cp:defn1}), (\ref{eq:cp:defn2}), and (\ref{eq:cp:defn3}) agree with each other) can be found in \cite[Proposition 2.7]{borodin2015general}.
It is based on integral identities due to Dixon \cite{dixon1905generalization} and Anderson \cite{Anderson1991}.

Sun proved in \cite[Section 4]{sun2015matrix} that the joint distribution of the
(different from $0, 1$) eigenvalues in $\mathcal{M}^{AnM}$, $n = 1, \cdots, N$ is
the same as the first $N$ rows of $\beta$--Jacobi corners process with $\alpha = A -
M + 1$, and $\beta = 1, 2$ (corresponding to real and complex entries, respectively).
See \cite[Section 1.5]{borodin2015general} for more discussions about matrix models of the $\beta$--Jacobi corners process.

\subsection{Signed measures and their diagrams}

Our main object of study is a pair of interlacing sequences $x^{N-1}$, $x^N$ drawn from the
$\beta$--Jacobi corners process. We assign to such a pair two closely related objects:
a \emph{signed measure} and a \emph{diagram}.

\begin{defn} \label{defn:sim}
Given an interlacing sequence $x_1 \leq y_1 \leq \cdots \leq y_{n-1} \leq x_n$, the
corresponding \textit{signed interlacing measure} $\mu^{\{x_i\},\{y_i\}}$ is an
atomic signed measure on $\mathbb R$ of total mass $1$ given by
\begin{equation}
\mu^{\{x_i\},\{y_i\}}(A) = \sum_{i=1}^{n} \mathds{1}_{x_i\in A} - \sum_{i=1}^{n-1}
\mathds{1}_{y_i\in A}, \quad \forall A\subset \mathbb{R} .
\end{equation}
\end{defn}

An alternative way to describe interlacing sequences (due to Kerov
\cite{Kerov93transitionprobabilities}, see also \cite{Bufetov_2013}) relies on the
notion of a diagram.

\begin{defn}  \label{defn:dia}
A \textit{diagram} $w: \mathbb{R}\rightarrow \mathbb{R}$ is a function satisfying:
\begin{enumerate}
\item Lipschitz condition: $| w(u_1) - w(u_2) | \leq |u_1 - u_2|$, $\forall u_1 , u_2 \in \mathbb{R}$.
\item There is a $u_0 \in \mathbb{R}$, the \textit{center} of $w$, such that $w(u) = |u - u_0|$ for $|u|$ large enough.
\end{enumerate}
Any diagram $w$ that is piecewise linear and satisfies $\frac{d}{du} w = \pm 1$ (except for
finitely many points) is called \textit{rectangular}.
\end{defn}

We draw a connection between interlacing sequences and diagrams, see Figure
\ref{fig:interseq} for an example.

\begin{defn}  \label{defn:yod}
For any interlacing sequence $x_1 \leq y_1 \leq \cdots \leq y_{n-1} \leq x_n$,
define its diagram $w: \mathbb{R}\rightarrow \mathbb{R}$ as follows:
\begin{enumerate}
\item For $u < x_1$ or $u > x_n$, let $w(u) = |u - u_0|$, where $u_0 = \sum_{i=1}^n x_i - \sum_{i=1}^{n-1}
y_i$.
\item For $i = 1, \cdots, n$, let $w(x_i) = \sum_{1\leq j < i}(y_j - x_j) + \sum_{i<j\leq n}(x_j - y_{j-1})$.
\item For $i = 1, \cdots, n-1$, let $w(y_i) = \sum_{1\leq j < i}(y_j - x_j) - x_i + x_{i+1} + \sum_{i + 1<j\leq n}(x_j - y_{j-1})$.
\item In all the intervals $[x_i, y_i]$ and $[y_i, x_{i+1}]$, $w$ is linear.
\end{enumerate}
\end{defn}

It's easy to verify that the defined $w$ is a rectangular diagram; to be more precise, it satisfies the following conditions:
\begin{enumerate}
\item $\frac{d}{du} w(u) = 1$, for any
 $u \in \left( \bigcup_{i=1}^{n-1}(x_i, y_i) \right) \bigcup (x_n, \infty)$ .
\item $\frac{d}{du} w(u) = -1$, for any
 $u \in \left( \bigcup_{i=1}^{n-1}(y_i, x_{i+1}) \right) \bigcup (- \infty, x_1)$ .
\end{enumerate}

\begin{figure}
    \centering
    \includegraphics[width=0.8\textwidth]{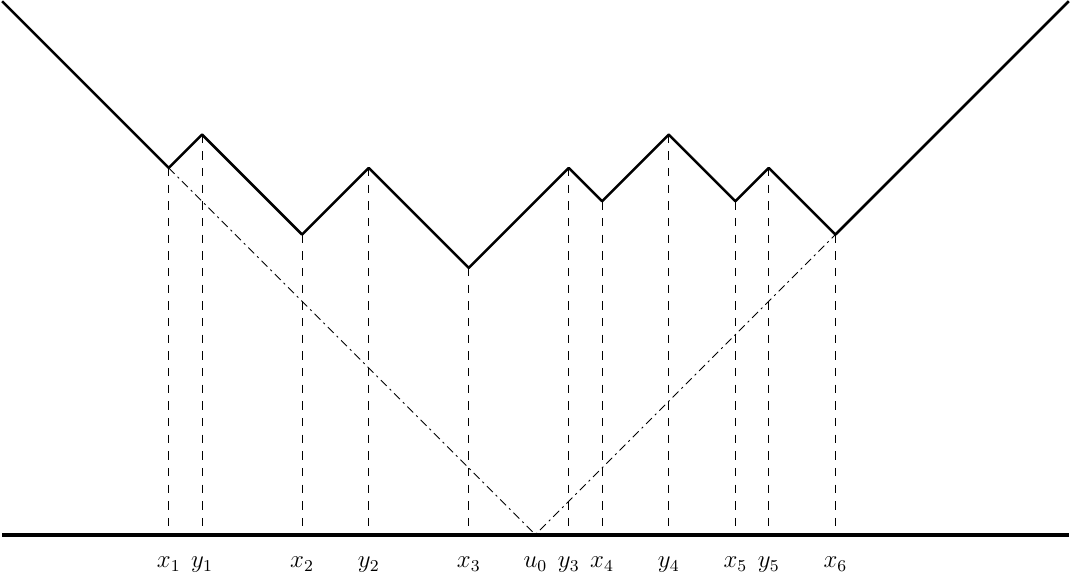}
    \caption{The diagram of an interlacing sequence}
    \label{fig:interseq}
\end{figure}

\begin{rem}  \label{rem:secder}
For an interlacing sequence $x_1 < y_1 < \cdots < y_{n-1} < x_n$, and its diagram
$w$, the second derivative $\frac{d^2}{du^2}w$ can be identified with
$2\mu^{\{x_i\},\{y_i\}}$.
\end{rem}

\subsection{Pullback of the Gaussian Free Field}

In this section we briefly define a pullback of the Gaussian Free Field, and review the results of \cite{borodin2015general} about the appearance of the GFF with
Dirichlet boundary conditions in the asymptotics of the $\beta$--Jacobi corners process.

Detailed surveys of the 2--dimensional Gaussian Free Field are given in
\cite{sheffield2007gaussian}, \cite[Section 4]{dubedat2009sle}, \cite{Werner_GFF},
and here we will omit some details. Informally, the Gaussian Free Field with
Dirichlet boundary conditions in the upper half plane $\mathbb{H}$ is defined as a
mean $0$ (generalized) Gaussian random field $\mathcal{G}$ on $\mathbb{H}$.
It vanishes on the real axis, and the
covariance (for any $z, w \in \mathbb{H}$) is
\begin{equation}  \label{eq:defn:gff}
\mathbb{E}(\mathcal{G}(z) \mathcal{G}(w)) = -\frac{1}{2\pi} \ln \left| \frac{z-w}{z-\bar{w}} \right| .
\end{equation}
Since \eqref{eq:defn:gff} has a singularity at the diagonal $z=w$, the value of the
GFF at a point is not defined. However, the GFF can be well-defined as an element of a
certain functional space. In particular, the integrals of $\mathcal{G}(z)$ against
sufficiently smooth measures are bona fide Gaussian random variables.

\bigskip

The next step is to define a correspondence which maps to the upper half--plane the
space where particles of the $\beta$--Jacobi corners process live.
\begin{defn}   \label{defn:dom}
Let $\hat{M}, \hat{\alpha} > 0$ be parameters, and define $D \subset [0, 1] \times \mathbb{R}_{>0}$ be the set of all $(x, \hat{N})$, satisfying the following:
\begin{equation}   \label{eq:sec:hei}
\left| x - \frac{\hat{M}\hat{N} + (\hat{M} + \hat{\alpha})(\hat{N} + \hat{\alpha})  }{(\hat{N} + \hat{\alpha} + \hat{M})^2} \right| < \frac{2\sqrt{\hat{M}\hat{N}(\hat{M} + \hat{\alpha})(\hat{N} + \hat{\alpha})}}{(\hat{N} + \hat{\alpha} + \hat{M})^2} .
\end{equation}
Let $\Omega : D \rightarrow \mathbb{H} $ be
such that the horizontal section of $D$ at height $\hat{N}$ (for $\hat{N} \neq \hat{M}$) is mapped to the
half-plane part of the circle, centered at
\begin{equation}
\frac{\hat{N}(\hat{\alpha} + \hat{M})}{\hat{N} - \hat{M}}
\end{equation}
with radius
\begin{equation}
\frac{\sqrt{\hat{M}\hat{N}(\hat{M} + \hat{\alpha})(\hat{N} + \hat{\alpha})}}{\left| \hat{N} - \hat{M} \right|}.
\end{equation}
(when $\hat{N} = \hat{M}$ the circle is replaced by the vertical line at $\frac{\hat{\alpha}}{2}$),
in a way that for any point $u$ on the half circle (or half vertical line), it is the image of
\begin{equation}
\left( \frac{u}{u + \hat{N}}\cdot \frac{u - \hat{\alpha}}{u - \hat{\alpha} - \hat{M}} , \hat{N} \right)     .
\end{equation}
\end{defn}
It is known that such $\Omega$ is well defined and injective, and the image is $\mathbb{H}$ without (the upper half of) the ball centered at $\hat{\alpha} + \hat{M}$ with radius $\sqrt{\hat{M}(\hat{M} + \hat{\alpha})}$ (see \cite[Section 4.6]{borodin2015general}).

\begin{defn}
\label{def_GFF_pullback} $\mathcal{K}$ is a generalized Gaussian random field in
$[0,1]\times \mathbb{R}_{> 0}$ which vanishes outside $D$  and is equal to $\mathcal{G} \circ
\Omega$ (i.e.\ the pullback of $\mathcal{G}$ with respect to map $\Omega$) inside
$D$.
\end{defn}

\begin{rem}
Since the image of $\Omega$ is smaller that $\mathbb{H}$, only the
restriction of the field $\mathcal{G}$ to the image $\Omega(D)$ is important.
However, for the definition of $\mathcal{G}$ we need to take the entire
$\mathbb{H}$, i.e.\ we do not impose any new boundary conditions for the
field on the non-real part of the boundary of the image.
\end{rem}

Again, the value of $\mathcal{K}$ at a given point in $D$ is not well-defined, but
it can be integrated with respect to certain types of measures; specifically, we
have the following results, which can essentially be taken as an alternative
definition of $\mathcal K$ (indeed,
either of Lemmas \ref{lemma:pbcov1}, \ref{lemma:pbcov2} can be used to recover the covariance kernel
of $\mathcal{K}(u, y)$).
\begin{lemma}   \label{lemma:pbcov1}
For any $0 < \hat{N}_1 \leq \cdots \leq \hat{N}_h $, and positive integers $k_1, \cdots, k_h$, the following random vector
\begin{equation}   \label{eq:lemma:pbcov1}
\left( \int_0^1 u^{k_i}\mathcal{K}(u, \hat{N}_i)du \right)_{i=1}^h
\end{equation}
is jointly centered Gaussian, and the covariance between the $i$th and $j$th component is the double contour integral
\begin{multline}
\label{eq_1d_cov} \frac{\theta^{-1}}{(2\pi\im)^2 (k_i + 1)(k_j + 1)} \oint \oint
\frac{dv_1 dv_2}{(v_1 -v_2)^2}
\\
\times
\left(\frac{v_1}{v_1 + \hat{N}_i}\cdot \frac{v_1 - \hat{\alpha}}{v_1 - \hat{\alpha} - \hat{M}} \right)^{k_i+1}
\left(\frac{v_2}{v_2 + \hat{N}_j}\cdot \frac{v_2 - \hat{\alpha}}{v_2 - \hat{\alpha} - \hat{M}} \right)^{k_j+1}  ,
\end{multline}
where $|v_1| \ll |v_2|$, and the contours enclose $-\hat{N}_i$, $-\hat{N}_j$,
but not $\hat{\alpha} + \hat{M}$.
\end{lemma}
\begin{rem}
We note that $|v_1| \ll |v_2|$ means that the contour of $v_2$ encloses any pole that depends on $v_1$, for any $v_1$ in its contour. In particular, in \eqref{eq_1d_cov}, this is equivalent to that the contour of $v_1$ is inside the contour of $v_2$.
We will use the notation $\ll$ throughout the following text.
\end{rem}
\begin{lemma}   \label{lemma:pbcov2}
For any integers $k_1, \cdots, k_h$, and $g_1, \cdots, g_h \in C^{\infty}([0, G])$ for some $G \in \mathbb{R}_{>0}$, the joint distribution of the vector
\begin{equation}   \label{eq:lemma:pbcov2}
\left( \int_0^G \int_0^1 u^{k_i} g_i(y) \mathcal{K}(u, y) du dy \right)_{i=1}^h
\end{equation}
is centered Gaussian, and the covariance between the $i$th and $j$th component is
\begin{multline}
\label{eq_2d_cov}
\int_0^G \int_0^G \frac{g_i(y_1)g_j(y_2)\theta^{-1}}{(2\pi \im)^2 (k_i + 1)(k_j + 1)} \oint \oint \frac{1}{(v_1 - v_2)^2} \\
\times \left(\frac{v_1}{v_1 + y_1} \cdot \frac{v_1 - \hat{\alpha}}{v_1 - \hat{\alpha} - \hat{M}}\right)^{k_i + 1} \left(\frac{v_2}{v_2 + y_2} \cdot \frac{v_2 - \hat{\alpha}}{v_2 - \hat{\alpha} - \hat{M}}\right)^{k_j + 1} dv_1 dv_2 dy_1 dy_2 ,
\end{multline}
where the inner contours enclose poles at $-y_1$ and $-y_2$, but not $\hat{\alpha} + \hat{M}$,
and are nested: when $y_1 \leq y_2$, $|v_1| \ll |v_2|$; when $y_1 \geq y_2$, $|v_2| \ll |v_1|$.
\end{lemma}
These above two Lemmas are obtained from the Gaussianity of integrals against the Gaussian Free Field (see e.g. \cite[Lemma 4.5, 4.6]{borodin2015general}), by following the arguments in \cite[Section 4.6]{borodin2015general} to pull the integrals back to $\mathcal{K}$.

Let us emphasize once again that since the values of $\mathcal{K}$ are not defined,
the expressions \eqref{eq:lemma:pbcov1} and \eqref{eq:lemma:pbcov2} are not
conventional integrals, rather they are pairings of a generalized random function
$\mathcal{K}$ with certain measures.

\section{Main results}
We proceed to the statements of our asymptotic theorems.

In our limit regime the parameters $\alpha$, $M$ of the $\beta$--Jacobi corners process and level
$N$ depend on a large auxiliary variable $L\rightarrow \infty$, in such a way that
\begin{equation} \label{eq:lsch}
\lim_{L\rightarrow \infty} \frac{\alpha}{L} = \hat{\alpha},\quad  \lim_{L\rightarrow \infty}
\frac{N}{L} = \hat{N},\quad  \lim_{L\rightarrow \infty} \frac{M}{L} = \hat{M}.
\end{equation}
For a random $\mathbb{P}^{\alpha, M, \theta}$--distributed sequence $(x^1, x^2, \cdots) \in
\chi^M$, we introduce random variables
\begin{equation}
\tilde{x}^N =
\begin{cases}
    x^N ,  & N \leq M \\
    (x^N_1, \cdots, x^N_M, \underbrace{1, \cdots, 1}_{N-M}\,),  & N > M
\end{cases}
\end{equation}
and
\begin{equation} \label{eq:defpk} \mathfrak{P}_k(x^N) = \sum_{i=1}^{N} (\tilde{x}_i^N)^k .
\end{equation}

The following three theorems give the $L\to\infty$ Law of Large Numbers for the pair $(x^{N-1}, x^N)$ in
three different forms.
They are equivalent to each other.
\begin{theorem} \label{thm:mom}
In the limit regime \eqref{eq:lsch}, the random variable $\mathfrak{P}_k(x^N) -
\mathfrak{P}_k(x^{N-1})$ converges to a constant as $L\rightarrow \infty$, in the sense that the
variance
\begin{equation}
\mathbb{E}\left[(\mathfrak{P}_k(x^N) - \mathfrak{P}_k(x^{N-1}))-\mathbb{E}(\mathfrak{P}_k(x^N) -
\mathfrak{P}_k(x^{N-1}))\right]^2 ,
\end{equation}
decays as $O(L^{-1})$. The constant is given by the following contour integral:
\begin{equation} \label{eq:mom:diff}
\lim_{L\rightarrow \infty} \mathbb{E}\left(\mathfrak{P}_k(x^N) - \mathfrak{P}_k(x^{N-1})\right) = \frac{1}{2\pi\im}\oint \left(\frac{v}{v + \hat{N}}\cdot \frac{v - \hat{\alpha}}{v - \hat{\alpha} - \hat{M}}\right)^k \frac{1}{v+\hat{N}} dv ,
\end{equation}
where the integration contour encloses the pole at $-\hat{N}$ but not $\hat{\alpha} + \hat{M}$,
and is positively oriented.
\end{theorem}

\begin{rem}
Using exactly the same approaches as those in the proof of Theorem \ref{thm:mom}, we can show that
\begin{equation}
\lim_{L\rightarrow \infty} L^{-1}\mathbb{E}\left(\mathfrak{P}_k(x^N)\right) = - \frac{1}{k} \cdot \frac{1}{2\pi\im}\oint \left(\frac{v}{v + \hat{N}}\cdot \frac{v - \hat{\alpha}}{v - \hat{\alpha} - \hat{M}}\right)^k dv ,
\end{equation}
whose $\hat{N}$ derivative is precisely \eqref{eq:mom:diff}.
\end{rem}


\begin{theorem} \label{thm:dia}
Let $w^{\tilde{x}^N, \tilde{x}^{N-1}}$ be the interlacing diagram of the sequence $\tilde{x}^N_1 \leq \tilde{x}^{N-1}_{1} \leq \cdots \leq \tilde{x}^{N-1}_{N-1} \leq \tilde{x}^N_N$.
Then it converges to a deterministic diagram $\varphi$ under the limit scheme \eqref{eq:lsch}, in the sense that
\begin{equation}
\lim_{L \rightarrow \infty } \sup_{u \in \mathbb{R} } \left| w^{\tilde{x}^N, \tilde{x}^{N-1}}(u) -
\varphi(u) \right| = 0,
\end{equation}
in probability. Here $\varphi$ is the unique diagram satisfying
\begin{equation}
\varphi''(u) =
\begin{cases}
    \frac{\hat{M} - \hat{N} + (\hat{N}+\hat{M} + \hat{\alpha})(1-u)}{\pi(\hat{N}+\hat{M} + \hat{\alpha})(1-u)}\frac{1}{\sqrt{(\gamma_2 - u)(u - \gamma_1)}} ,  & u \in (\gamma_1, \gamma_2) \\
    2C(\hat{M}, \hat{N}) \delta(u - 1),  & u \in (-\infty, \gamma_1] \bigcup [\gamma_2, \infty) ,
\end{cases}
\end{equation}
where
\begin{equation}
\gamma_1 = \frac{\left(\sqrt{(\hat{\alpha} + \hat{M})(\hat{\alpha} + \hat{N})} -
\sqrt{\hat{M}\hat{N}}\right)^2}{(\hat{N} + \hat{M}+\hat{\alpha})^2}, \quad \quad \gamma_2  =
\frac{\left(\sqrt{(\hat{\alpha} + \hat{M})(\hat{\alpha} + \hat{N})} +
\sqrt{\hat{M}\hat{N}}\right)^2}{(\hat{N} + \hat{M}+\hat{\alpha})^2} ,
\end{equation}
\begin{equation}
C(\hat{M}, \hat{N}) =
\begin{cases}
    0 ,  & \hat{M} > \hat{N} \\
    \frac{1}{2} ,  & \hat{M} = \hat{N} \\
    1 ,  & \hat{M} < \hat{N}
\end{cases}.
\end{equation}
\end{theorem}

Note that $\gamma_1,\gamma_2$ precisely describe the left and right boundaries of the
domain $D$ in Definition \ref{defn:dom}.

\begin{theorem} \label{thm:mea}
Let $\varphi$ be defined as in Theorem \ref{thm:dia}. Take a function
$f: [0, 1]\rightarrow \mathbb{R}$, such that $f'$ exists almost everywhere, and is of finite variation. Then the random variable
\begin{equation} \label{eq:thm:mea}
\int_{0}^1 f d\mu^{\tilde{x}^N, \tilde{x}^{N-1}} = \sum_{i=1}^N f(\tilde{x}_i^N) - \sum_{i=1}^{N-1}
f(\tilde{x}_i^{N-1})
\end{equation}
converges (in probability) as $L\to\infty$ in the limit regime \eqref{eq:lsch} to the constant
$\frac{1}{2}\int_{0}^1 f(u) \varphi''(u) du$.
\end{theorem}

\begin{figure}

\begin{minipage}{.33\textwidth}
\centering
\resizebox{\textwidth}{1.7in}{%
\begin{tikzpicture}
    \begin{axis}[
        domain=0:1,
        xmin=-0.05, xmax=1.1,
        ymin=-0.8, ymax=4.0,
        samples=200,
        axis y line=center,
        axis x line=middle,
        xlabel={$u$},
    ]
        \addplot+[mark=none, color=blue] {(5-4*x)/(4 * pi * (1 - x) * sqrt((0.93301270189 - x) * (x - 0.0669872981)))};
        \addplot+[mark=none, style=dashed, color=black] coordinates {(0.0669872981, -0.8) (0.0669872981, 4.0)};
        \addplot+[mark=none, style=dashed, color=black] coordinates {(0.93301270189, -0.8) (0.93301270189, 4.0)};
    \end{axis}
\end{tikzpicture}
}
\caption*{$\hat{M} = 2$, $\hat{N} = \hat{\alpha} = 1$}
\end{minipage}
\begin{minipage}{.33\textwidth}
\centering
\resizebox{\textwidth}{1.7in}{%
\begin{tikzpicture}
    \begin{axis}[
        domain=0:1,
        xmin=-0.05, xmax=1.1,
        ymin=-0.8, ymax=4.0,
        samples=200,
        axis y line=center,
        axis x line=middle,
        xlabel={$u$},
    ]
        \addplot+[mark=none, color=blue] {1 /(pi * sqrt((1 - x) * (x - 0.1111111111)))};
        \addplot+[mark=none, color=blue] coordinates {(1, 0) (1, 4.5)};
        \addplot+[mark=none, style=dashed, color=black] coordinates {(1, -0.8) (1, 4.0)};
        \addplot+[mark=none, style=dashed, color=black] coordinates {(0.1111111111, -0.8) (0.1111111111, 4.0)};
    \end{axis}
\end{tikzpicture}
}
\caption*{$\hat{M} = \hat{N} = \hat{\alpha} = 1$}
\end{minipage}
\begin{minipage}{.33\textwidth}
\centering
\resizebox{\textwidth}{1.7in}{%
\begin{tikzpicture}
    \begin{axis}[
        domain=0:1,
        xmin=-0.05, xmax=1.1,
        ymin=-2.3, ymax=2.5,
        samples=200,
        axis y line=center,
        axis x line=middle,
        xlabel={$u$},
    ]
        \addplot+[mark=none, color=blue] {(3-4*x)/(4 * pi * (1 - x) * sqrt((0.93301270189 - x) * (x - 0.0669872981)))};
        \addplot+[mark=none, color=blue] coordinates {(1, 0) (1, 2.5)};
        \addplot+[mark=none, style=dashed, color=black] coordinates {(0.93301270189, -2.3) (0.93301270189, 2.5)};
        \addplot+[mark=none, style=dashed, color=black] coordinates {(0.0669872981, -2.3) (0.0669872981, 2.5)};
    \end{axis}
\end{tikzpicture}
}
\caption*{$\hat{N} = 2$, $\hat{M} = \hat{\alpha} = 1$}
\end{minipage}
\caption{Plots of the density $\frac{\varphi''}{2}.$ \label{Figure_plots}}
\end{figure}

\begin{rem}
As each $\mu_{\tilde{x}^N, \tilde{x}^{N-1}}$ is not a positive measure, Theorem \ref{thm:mea} does
not hold for general $f$. For example, take any $A\subset [0, 1]$, and let $f$ be the indicator
function of $A$; then \eqref{eq:thm:mea} takes only integer values, and can not converge to a
non-integer constant. This implies that the measures $d\mu_{\tilde{x}^N, \tilde{x}^{N-1}}$ \emph{do
not} weakly converge to the measure with density $\frac{\varphi''}{2}$.

Also, the measure with density $\frac{\varphi''}{2}$ is not necessarily positive (although it has
total mass $1$): when $\hat{M} < \hat{N}$ the density function can take negative values, cf.\
Figure \ref{Figure_plots}.  This measure is an instance of the \textit{interlacing measures}, which
were introduced and studied by Kerov (see \cite[Section 1.3]{kerov1998interlacing}).
\end{rem}

We note that the asymptotic objects in Theorem \ref{thm:mom}, \ref{thm:dia}, and \ref{thm:mea} are independent of the parameter $\theta$ in the underlying $\beta$--Jacobi corners process.
However, these results have a remarkable limit as $\theta \rightarrow \infty$.
They degenerate to statements about asymptotic separation of the roots of Jacobi orthogonal polynomials.

Let $\mathcal{F}_{n}^{p, q}$ be the Jacobi orthogonal polynomials
of degree $n$ with weight function $x^{p}(1-x)^{q}$ on $[0, 1]$, see e.g.
\cite{book:7334}.
Let $j_{M,N,\alpha,i}$ be the $i$th root (in increasing order) of $\mathcal{F}_{\min(M, N)}^{\alpha - 1, |M-N|}$,
for $1 \leq i \leq \min(M, N)$.
We further denote $j_{M,N,\alpha,i} = 1$,
for any fixed $M, N, \alpha$, and $\min(M, N) < i \leq N$.

\begin{theorem}  \label{thm:jpr}
There is an interlacing relationship for the roots:
\begin{equation}
j_{M,N,\alpha,1} \leq j_{M,N-1,\alpha,1} \leq j_{M,N,\alpha,2} \leq \cdots .
\end{equation}
Let $\iota_{M,N,\alpha}$ be the diagram corresponding to this interlacing sequence, as in
Definition \ref{defn:yod}, and let $\varphi$ be defined as in Theorem \ref{thm:dia}. Under the
limit scheme \eqref{eq:lsch}, the diagrams $\iota_{M,N,\alpha}$ converge to $\varphi$ in the uniform
topology.
\end{theorem}
Kerov in \cite{kerov1993root} proved similar statements about Hermite and Chebyshev polynomials.

\bigskip

Now we switch to the Central Limit Theorems. The first result describes the asymptotic behavior of
fluctuations for the individual $\mathfrak{P}_{k_i}(x^{N}) - \mathfrak{P}_{k_i}(x^{N-1})$.

\begin{theorem} \label{thm:clt1}
For positive integers $k_1, \cdots, k_h$, $k'_1, \cdots, k'_{h'}$, $N_1, \cdots, N_h$, and $N'_1, \cdots, N'_{h'}$,
in addition to the limit scheme \eqref{eq:lsch} we also let
\begin{equation} \label{eq:lsch2}
\lim_{L\rightarrow \infty} \frac{N_i}{L} = \hat{N}_i , \quad 1\leq i \leq h, \quad
\quad \quad \lim_{L\rightarrow \infty} \frac{N'_i}{L} = \hat{N}'_i , \quad 1\leq i
\leq h' .
\end{equation}
The random vectors
\begin{equation} \label{eq:clt1:d:cov}
L^{\frac{1}{2}}\left( \mathfrak{P}_{k_i}(x^{N_i}) - \mathfrak{P}_{k_i}(x^{N_i-1}) - \mathbb{E}\left(\mathfrak{P}_{k_i}(x^{N_i}) - \mathfrak{P}_{k_i}(x^{N_i-1})\right) \right)_{i=1}^h
\end{equation}
and
\begin{equation} \label{eq:clt1:cov}
\left( \mathfrak{P}_{k'_i}(x^{N'_i}) - \mathbb{E}\left(\mathfrak{P}_{k'_i}(x^{N'_i}) \right) \right)_{i=1}^{h'}
\end{equation}
jointly converge (as $L\rightarrow \infty$) to centered Gaussian random vectors in distribution, and the two limiting vectors are independent.
Within the limiting vector of \eqref{eq:clt1:d:cov},
the covariance between the $i$th and $j$th component becomes
\begin{equation}
 - \mathds{1}_{\hat{N}_i = \hat{N}_j} \cdot \frac{k_i k_j}{k_i + k_j} \cdot
\frac{\theta^{-1}}{2\pi\im } \oint \frac{1}{(v + \hat{N}_i)^2} \left(
\frac{v}{v+\hat{N}_i}\cdot \frac{v-\hat{\alpha}}{v - \hat{\alpha} -
\hat{M}}\right)^{k_i + k_j} dv,
\end{equation}
where the contour encloses $-\hat{N}_i$ but not $\hat{\alpha} + \hat{M}$; within the
vector \eqref{eq:clt1:cov}, the covariance between the $i$th and $j$th component becomes
\begin{equation}  \label{eq:clt1:cov:wi}
\frac{\theta^{-1}}{(2\pi\im)^{2}}
\oint \oint
\frac{1}{(v_1 - v_2)^2}
\left( \frac{v_1}{v_1+\hat{N}'_i}\cdot \frac{v_1-\hat{\alpha}}{v_1 - \hat{\alpha} - \hat{M}} \right)^{k_i} 
\left( \frac{v_2}{v_2+\hat{N}'_j}\cdot \frac{v_2-\hat{\alpha}}{v_2 - \hat{\alpha} - \hat{M}} \right)^{k_j} 
dv_1 dv_2  ,
\end{equation}
where the contour of $v_1$ encloses $-\hat{N}'_i$ and the contour of $v_2$ encloses $-\hat{N}'_j$, and neither of them encloses $\hat{\alpha} + \hat{M}$.
We also require that $|v_1| \ll |v_2|$, assuming that $\hat{N}'_i \leq \hat{N}'_j$.
\end{theorem}

\begin{rem}
The Gaussianity of the vector \eqref{eq:clt1:cov} is actually \cite[Theorem
4.1]{borodin2015general}, and here we are more interested in its joint distribution
with (\ref{eq:clt1:d:cov}).
The proof presented in this text also gives an alternative derivation of the results of \cite[Theorem 4.1]{borodin2015general}. In particular, the covariance \eqref{eq:clt1:cov:wi} can
be directly computed using the approach of Section \ref{ssec:cc}.
\end{rem}

The asymptotic behavior is different for the weighted averages (in $N$) of
$\mathfrak{P}_{k_i}(x^{N}) - \mathfrak{P}_{k_i}(x^{N-1})$.
\begin{theorem} \label{thm:clt2}
Let $k_1, \cdots, k_h$ be integers, and $g_1, \cdots, g_h \in L^{\infty}([0, G])$, for some $G \in \mathbb{R}_{>0}$.
Under \eqref{eq:lsch}, the random vector
\begin{equation}  \label{eq:clt2:def}
\left( L \int_0^G  g_i(y)
\left( \mathfrak{P}_{k_i}(x^{\lfloor Ly \rfloor}) - \mathfrak{P}_{k_i}(x^{\lfloor Ly \rfloor - 1}) - \mathbb{E}\left(\mathfrak{P}_{k_i}(x^{\lfloor Ly \rfloor}) - \mathfrak{P}_{k_i}(x^{\lfloor Ly \rfloor - 1})\right) \right)
dy \right)_{i=1}^h
\end{equation}
converges in distribution to a centered Gaussian vector, with covariance
between the $i$th and $j$th component
\begin{multline} \label{eq:clt2:cov}
\iint_{0 \leq y_1 < y_2 \leq G}
\frac{\theta^{-1} }{(2\pi\im)^{2}}
\oint \oint
\frac{k_i k_j}{(v_1 - v_2)^2(v_1 + y_1 )(v_2 + y_2 )}
\\
\times
\left(
g_i(y_1) g_j(y_2)
\left( \frac{v_1}{v_1+y_1}\cdot \frac{v_1-\hat{\alpha}}{v_1- \hat{\alpha} - \hat{M}} \right)^{k_i}
\left( \frac{v_2}{v_2+y_2}\cdot \frac{v_2-\hat{\alpha}}{v_2- \hat{\alpha} - \hat{M}} \right)^{k_j}
\right.
\\
+
\left.
g_j(y_1) g_i(y_2)
\left( \frac{v_1}{v_1+y_1}\cdot \frac{v_1-\hat{\alpha}}{v_1- \hat{\alpha} - \hat{M}} \right)^{k_j}
\left( \frac{v_2}{v_2+y_2}\cdot \frac{v_2-\hat{\alpha}}{v_2- \hat{\alpha} - \hat{M}} \right)^{k_i}
\right)
dv_1 dv_2 dy_1 dy_2
\\
- \int_0^G
\frac{\theta^{-1} }{2\pi\im}
\oint
\frac{g_i(y) g_j(y) k_i k_j}{(k_i + k_j)(v + y)^2} \left( \frac{v}{v+y}\cdot \frac{v-\hat{\alpha}}{v - \hat{\alpha} - \hat{M}}\right)^{k_i + k_j} dv dy ,
\end{multline}
where in the first integral, the contours are nested:
$|v_1| \ll |v_2|$, and enclose $-y_1, -y_2$ but not $\hat{\alpha} + \hat{M}$;
in the second integral, the contour encloses $-y$ but not $\hat{\alpha} + \hat{M}$.
\end{theorem}

\begin{rem}
Let us emphasize that the scalings in \eqref{eq:clt1:d:cov} and \eqref{eq:clt2:def}
are different: for a single difference the scale is $L^{\frac{1}{2}}$, while for the
weighted average the scale is $L$.
The conceptual reason is that the limiting
field $\mathcal{K}$ (of Definition \ref{def_GFF_pullback}) is differentiable only in \emph{generalized sense} in a
spatial direction even after smoothing in another direction; see
Theorems \ref{thm:gff:la}, \ref{thm:gff:sm} for more details on the field $\mathcal K$.
One vague analogy
is with the functional central limit theorem describing the convergence
of the random walk to the Brownian motion: the latter is
non-differentiable, and we need to rescale increments of the random walk
differently (in fact, there will be no rescaling at all) to see a finite
random variable as their limit.
\end{rem}

\begin{rem}
If $g_i = \mathds{1}_{y\leq \hat{N}'_i}$, then \eqref{eq:clt2:def} has the same asymptotic behavior as \eqref{eq:clt1:cov}.
In particular, one can get \eqref{eq:clt1:cov:wi} from \eqref{eq:clt2:cov}, by integrating in $y_1, y_2$ in the first integral in \eqref{eq:clt2:cov}.
In doing this it suffices to consider the anti-derivative of its integrand.
At $\hat{N}'_i$ it is \eqref{eq:clt1:cov:wi}, at zero it vanishes, and along the line $y_1 = y_2$ it is the second integral of \eqref{eq:clt2:cov} (using integral identities from Appendix \ref{sec:dr}).
\end{rem}

Theorems \ref{thm:clt1} and \ref{thm:clt2} have an interpretation in terms of the
Gaussian Free Field. For that we define (random) height functions, similar to
\cite{Borodin_CLT}, \cite{borodin2015general}.

\begin{defn} \label{defn:hf}
Let the sequences $x^1, x^2, \cdots$ be distributed as $\mathbb{P}^{\alpha, M, \theta}$.
For any $L > 0$ and $(u, y) \in [0, 1]\times \mathbb{R}_{>0}$, define $\mathcal{H}_L(u, y)$ to be the number of $i$ such that $x_i^{\lfloor Ly \rfloor}$ is less than $u$.
For $y > 1$, let $\mathcal{W}_L(u, y) = L\left(\mathcal{H}_L(u, y) - \mathcal{H}_L(u, y-L^{-1})\right)$.
\end{defn}

In \cite{borodin2015general},
the authors proved that centered $\mathcal{H}_L$ converges to the random field
$\mathcal{K}$ of Definition \ref{def_GFF_pullback}.
Our Central Limit
Theorems imply the convergence of centered $\mathcal{W}_L$ to a derivative of the random field
$\mathcal{K}$, as $L \rightarrow \infty$. In more details, Theorem \ref{thm:clt1} leads to the weak convergence
to a ``renormalized derivative'' of the random field $\mathcal{K}$, in the following
sense.

\begin{theorem} \label{thm:gff:la}
Under the limit scheme (\ref{eq:lsch}), for any integers $k_1, \cdots, k_h$ and $k'_1, \cdots, k'_{h'}$, real numbers $0 < \hat{N}_1 \leq \cdots \leq \hat{N}_h$ and $0 < \hat{N}'_1 \leq \cdots \leq \hat{N}'_{h'}$, the vector
\begin{equation}  \label{eq:prop:gff:la1}
\left( L^{-\frac{1}{2}} \int_0^1 u^{k_i} \left( \mathcal{W}_L(u, \hat{N}_i) - \mathbb{E} \left( \mathcal{W}_L(u, \hat{N}_i) \right) \right) du \right)_{i=1}^h
\end{equation}
as $L\rightarrow \infty$ converges in distribution to a Gaussian vector, which is the same as the weak limit
\begin{equation}  \label{eq:prop:gff:la2}
\lim_{\delta\rightarrow 0_+} \delta^{-\frac{1}{2}} \left( \int_0^1 u^{k_i}
\mathcal{K}(u, \hat{N}_i + \delta) du - \int_0^1 u^{k_i} \mathcal{K}(u, \hat{N}_i)
du \right)_{i=1}^h.
\end{equation}
In addition,
\begin{equation}  \label{eq:lemma:if:2}
\left(
\int_0^1 u^{k'_i} \left( \mathcal{H}_L(u, \hat{N}'_i) - \mathbb{E} \left( \mathcal{H}_L(u, \hat{N}'_i) \right) \right) du \right)_{i=1}^{h'}
\end{equation}
and \eqref{eq:prop:gff:la1} jointly converge (in distribution) as $L \rightarrow \infty$, while the limit vectors are independent.
\end{theorem}

In words, Theorem \ref{thm:gff:la} means that the limiting field for $1d$--slices of
$L^{-\frac{1}{2}} \mathcal{W}_L$ (in the $u$--direction), is the same as the renormalized
$y$--derivative of the limiting field for $\mathcal{H}_L$; but when letting
$L\rightarrow \infty$ simultaneously for $1d$ slices of $L^{-\frac{1}{2}}\mathcal{W}_L$ and
$\mathcal{H}_L$, one gets independent fields.

\begin{rem}
By \cite[Theorem 4.13]{borodin2015general}, \eqref{eq:lemma:if:2} converges to
\begin{equation}  \label{eq:rem:cit}
\left(\int_{0}^{1} u^{k'_i} \mathcal{K}\left(u, \hat{N}'_i\right) du \right)_{i=1}^{h'} 
\end{equation}
as $L \rightarrow \infty$.
A simple computation involving Lemma \ref{lemma:pbcov1} shows that \eqref{eq:prop:gff:la2} is
independent from \eqref{eq:rem:cit},
in the sense that as $\delta \rightarrow 0_+$, the covariances tend to zero.
Note, however, that
\eqref{eq:prop:gff:la2} is defined as a weak limit, and may not actually exist in the
probability space of $\mathcal{K}$.
\end{rem}

In contrast to Theorem \ref{thm:gff:la}, when we deal with $2d$--integrals of $
\mathcal{W}_L$ and $\mathcal{H}_L$, then the limiting fields turn out to be much more closely
related. Namely, we define the pairings $\mathfrak{Z}_{g, k}$ of the $y$--derivative
of the field $\mathcal K$ with test functions $u^k g(y)$ through the following
procedure based on integration by parts in the $y$--direction.
\begin{defn}   \label{defn:pair}
For any $G \in \mathbb{R}_{>0}$ and $g \in C^{\infty}([0, G])$,
with $g(G) = 0$,
define
\begin{equation}
\mathfrak{Z}_{g, k} = \int_0^G\int_0^1 u^{k} \left(\frac{d}{dy}g(y)\right) \mathcal{K}(u, y) du dy  .
\end{equation}
\end{defn}

\begin{lemma}   \label{lemma:derk:defn}
For any $G \in \mathbb{R}_{>0}$ and $g \in L^2([0, G])$, and positive integer $k$, there exists a sequence of functions $g_1, g_2,
\cdots$, satisfying that
\begin{enumerate}
\item Each $g_n \in C^{\infty}([0, G])$, and $g_n(G) = 0$.
\item $\lim_{n \rightarrow \infty} g_n = g$ in $L^2([0, G])$.
\item The sequence of random variables $\mathfrak{Z}_{g_1, k}, \mathfrak{Z}_{g_2, k}, \cdots$ converges almost surely.
\end{enumerate}
If there is another sequence $\tilde{g}_1, \tilde{g}_2, \cdots$ satisfying the same conditions, then the limits $\lim_{n \rightarrow \infty} \mathfrak{Z}_{g_n, k}$ and $\lim_{n \rightarrow \infty} \mathfrak{Z}_{\tilde{g}_n, k}$ are almost surely the same.
\end{lemma}

With this we can extend the definition of $\mathfrak{Z}_{g, k}$ to any $G \in \mathbb{R}_{>0}$ and $g \in L^2([0, G])$.

\begin{defn}  \label{defn:pair2}
For any $G \in \mathbb{R}_{>0}$ and $g \in L^2([0, G])$ we define $\mathfrak{Z}_{g, k}$ to be the limit in Lemma \ref{lemma:derk:defn}.
\end{defn}

Now we state the convergence of $\mathcal{W}_L$ to the $y$--derivative of $\mathcal{K}$ in the following sense.
\begin{theorem} \label{thm:gff:sm}
Let $k_1, \cdots, k_h$ be positive integers, $G \in \mathbb{R}_{>0}$, and $g_1, \cdots, g_h \in L^{\infty}([0, G])$.
As $L \rightarrow \infty$,
the vector
\begin{equation}  \label{eq:prop:gff:sm1}
\left( \int_0^G \int_0^1 u^{k_i} g_i(y) \left( \mathcal{W}_L(u, y) - \mathbb{E} \left( \mathcal{W}_L(u, y) \right) \right) du dy \right)_{i=1}^h
\end{equation}
converges in distribution to the vector
$\left( \mathfrak{Z}_{g_i, k_i} \right)_{i=1}^h$.
Moreover, take differentiable functions $\tilde{g}_1, \cdots, \tilde{g}_{h'} \in L^{\infty}([0, G])$, such that
$\frac{d}{dy}\tilde{g}_{i} \in L^{\infty}([0, G])$ for each $1 \leq i \leq h'$,
and positive integers $k_1', \cdots, k_{h'}'$.
Then consider the vector
\begin{equation}  \label{eq:prop:gff:sm2}
\left( \int_0^G \int_0^1 - u^{k_i'} \left(\frac{d}{dy}\tilde{g}_{i}(y)\right) \left( \mathcal{H}_L(u, y) - \mathbb{E} \left( \mathcal{H}_L(u, y) \right) \right) du dy \right)_{i=1}^{h'} .
\end{equation}
As $L \rightarrow \infty$, \eqref{eq:prop:gff:sm1} and \eqref{eq:prop:gff:sm2} jointly converge in distribution to the $h + h'$ dimensional vector 
$\left( \mathfrak{Z}_{g_i, k_i} \right)_{i=1}^h \bigcup \left( \mathfrak{Z}_{\tilde{g}_i, k_i'} \right)_{i=1}^{h'}$.
\end{theorem}
\begin{rem}
We point out that for $g_i \in C^{\infty}([0, G])$, Theorem \ref{thm:gff:sm} can be obtained from \cite[Theorem 4.13]{borodin2015general} via integration by parts.
However, we are unaware of any approach that extends to the case of $g_i \in L^{\infty}([0, G])$ without using Theorem \ref{thm:clt2}.
\end{rem}

\subsection*{Organization of remaining text}
The remaining sections are devoted to proofs of the above stated results.

Section \ref{sec:dj} presents the formulas for the expectations of the joint moments of $\beta$--Jacobi corners processes, using Macdonald processes and difference operators.
The proofs of the Law of Large Numbers (Theorem \ref{thm:mom}) and the related convergence of diagrams and measures
(Theorems \ref{thm:dia}, \ref{thm:mea}) can be found in Section \ref{sec:tm}, except
that the decay of variance in Theorem \ref{thm:mom} is left for Section
\ref{sec:clt}, which contains the proofs of the Central Limit Theorems (Theorems \ref{thm:clt1}, \ref{thm:clt2}).
Section \ref{sec:gff} contains the proofs of Theorems
\ref{thm:gff:la} and \ref{thm:gff:sm}, and Lemma \ref{lemma:derk:defn}.
Throughout the proofs we will widely use some contour integral identities, which are given in Appendix \ref{sec:dr}, to simplify the computations.

\section{Discrete joint moments}  \label{sec:dj}

In this section we compute the joint moments in $\beta$--Jacobi corners processes.
The main goal is to prove the following result.

\begin{theorem} \label{thm:var:dis}
Let $(x^1, x^2, \cdots) \in \chi^M$ be distributed as $\mathbb{P}^{\alpha, M, \theta}$,
and let $\mathfrak{P}_k(x^N)$ be defined as (\ref{eq:defpk}).
Let $l$, $N_1 \leq \cdots \leq N_l$, and $k_1, \cdots, k_l$ be positive integers, satisfying $\alpha + M > k_1 + \cdots + k_l$.

For any positive integers $n$, $m$, $\tilde{m}$, and variables $w_1, \cdots, w_m$,
$\tilde{w}_1, \cdots, \tilde{w}_{\tilde{m}}$,
denote
\begin{multline}
\mathfrak{I}(w_1, \cdots, w_m; \alpha, M, \theta, n) = \frac{1}{(w_2-w_1+1-\theta)\cdots (w_{m} - w_{m-1}+1-\theta)} \\
\times \prod_{1\leq i<j\leq m}\frac{(w_j-w_i) (w_j-w_i+1-\theta)}{(w_j-w_i-\theta) (w_j-w_i+1)}
\prod_{i=1}^{m}\frac{w_i-\theta}{w_i+\theta(n-1)}\cdot \frac{w_i-\theta\alpha}{w_i-\theta\alpha - \theta M} ,
\end{multline}
and
\begin{equation}
\mathfrak{L}(w_1, \cdots, w_m; \tilde{w}_1, \cdots, \tilde{w}_{\tilde{m}}; \theta) =
\prod_{1\leq i\leq \tilde{m}, 1\leq j\leq m}\frac{(\tilde{w}_i - w_j) (\tilde{w}_i - w_j+1-\theta)}{(\tilde{w}_i - w_j-\theta) (\tilde{w}_i - w_j+1)} .
\end{equation}
Then the moments of $\mathfrak{P}_k(x^N)$ can be computed via
\begin{multline}   \label{eq:thm:var:dis:st}
\mathbb{E}\left( \mathfrak{P}_{k_1}(x^{N_1}) \cdots \mathfrak{P}_{k_l}(x^{N_l}) \right)
=\frac{(-\theta)^{-l}}{(2\pi\im)^{k_1 + \cdots + k_l}} \oint \cdots \oint
\prod_{i=1}^{l} \mathfrak{I}(u_{i,1}, \cdots, u_{i,k_i}; \alpha, M, \theta, N_i)
\\
\times \prod_{i < j} \mathfrak{L}(u_{i,1}, \cdots, u_{i,k_i}; u_{j,1}, \cdots, u_{j,k_j}; \theta)
\prod_{i=1}^l \prod_{i'=1}^{k_i} du_{i,i'}
,
\end{multline}
where for each $i = 1, \cdots, l$,
the contours of $u_{i,1}, \cdots, u_{i, k_i}$ enclose $-\theta (N_i - 1)$ but not $\theta(\alpha + M)$,
and $|u_{i,1}|\ll \cdots \ll |u_{i,k_i}|$.
For $1\leq i < l$,
we also require that $|u_{i, k_{i}}| \ll |u_{i+1, 1}|$.
\end{theorem}

We remark that a \emph{different} contour integral expression for the left hand side of (\ref{eq:thm:var:dis:st}) is given in \cite[Section 3]{borodin2015general}.
The authors are not aware of a direct way to match the two expressions.

The proof of Theorem \ref{thm:var:dis} relies on the formalism of
\textit{Macdonald processes}.
Under certain limit transition it weakly converges to $\mathbb{P}^{\alpha, M, \theta}$. In turn, we compute the moments of
Macdonald processes by applying a remarkable family of difference operators coming
from the work \cite{Negut} on the symmetric functions. A particular case ($l = 1$) of Theorem \ref{thm:var:dis} was proven by one of the authors and
Borodin in the appendix to \cite{Fyodorov2016}.

\medskip

We recall the definition and some basic asymptotic relations of Macdonald processes in Section \ref{ssec:mpa}.
Then in Section \ref{ssec:do}, we introduce the differential operators, which help to extract moments of Macdonald processes; we also give another expression of these operators on a special class of functions, in terms of nested contour integrals (Proposition \ref{prop:act}).
In Section \ref{ssec:hom}, we first show that applying the operators repeatedly can get moments of Macdonald processes (Proposition \ref{prop:ev:lhs}).
Then by using Proposition \ref{prop:act} repeatedly, we evaluate the result of applying the operators repeatedly, as nested contour integrals (Proposition \ref{prop:act:mul}). 
In Section \ref{ssec:lt}, we do a limit transition in Macdonald processes, and get the desired expression for moments of $\beta$--Jacobi corners processes.

\subsection{Macdonald processes and asymptotic relations}  \label{ssec:mpa}

Let $\Lambda_N$ denote the \emph{ring of symmetric polynomials in $N$ variables},
and $\Lambda$ denote the \emph{ring of symmetric polynomials in countably many variables}
(see \cite[Chapter I, Section 2]{book:491124}).
Let $\mathbb{Y}$ be the set of \emph{partitions}, i.e. infinite non-increasing sequence of non-negative integers,
which are eventually zero:
$$\mathbb{Y} = \left\{ \lambda = (\lambda_1, \lambda_2, \cdots) \in \mathbb{Z}^{\infty} \left|  \lambda_1 \geq  \lambda_2 \geq \cdots \geq 0, \; \exists N \in \mathbb{Z}_+ , \lambda_N = 0 \right. \right\}, $$
and $\mathbb{Y}_N \subset \mathbb{Y}$ consists of sequences $\lambda$ such that $\lambda_{N+1} = 0$.
Let $|\lambda| = \sum_{i=1}^{\infty} \lambda_i$ be the \emph{size} of partition $\lambda$.

We can make $\mathbb{Y}$ a partially ordered set, by using dominance order:
\begin{equation}
\lambda < \mu \iff |\lambda| = |\mu|,\quad \lambda \neq \mu, \quad \lambda_1 + \cdots + \lambda_i \leq \mu_1 + \cdots + \mu_i, \; \forall i = 1, 2, \cdots .
\end{equation}

For any $\lambda \in \mathbb{Y}$,
denote $P_{\lambda}(\cdot; q, t)\in \Lambda$ to be the \emph{normalized Macdonald polynomial},
\begin{equation}
P_{\lambda}(\cdot; q, t) = m_{\lambda} + \sum_{\mu < \lambda} u_{\lambda \mu} m_{\mu}
\end{equation}
where $m_{\mu}$ are the monomial symmetric polynomials, and $u_{\lambda \mu}$ are certain real coefficients depending on $q, t$, see \cite[Section VI.4]{book:491124}.
Here $q$ and $t$ are real parameters, and we assume that $0<q<1$ and $0<t<1$.
From this definition, each $P_{\lambda}(\cdot; q, t)$ is homogeneous with degree $|\lambda|$,
and the collection
\begin{equation}  \label{eq:coll}
\left\{ P_{\lambda}(\cdot; q, t) \; | \; \lambda \in \mathbb{Y} \right\}
\end{equation}
is a basis of $\Lambda$.
We also denote $Q_{\lambda}(\cdot; q, t) = b_{\lambda}(q,t)P_{\lambda}(\cdot; q, t)$,
where $b_{\lambda}(q,t)$ is a constant uniquely defined by the identity \eqref{eq:cay} below and with an explicit expression given by \cite[Chapter VI (4.11)]{book:491124}.
We further define the \emph{skew Macdonald polynomials} $P_{\lambda/\mu}$ and $Q_{\lambda/\mu}$, where $\lambda, \mu \in \mathbb{Y}$, to be the coefficients of the following expansions (see \cite[Chapter VI, (7.9)]{book:491124}):
\begin{equation}  \label{eq:mpexp}
\begin{split}
P_{\lambda}(a_1, \cdots, a_N, b_1, \cdots, b_N; q, t) = \sum_{\mu\in \mathbb{Y}_N} P_{\lambda/\mu}(a_1, \cdots, a_N; q, t)P_{\mu}(b_1, \cdots, b_N; q, t) \\
Q_{\lambda}(a_1, \cdots, a_N, b_1, \cdots, b_N; q, t) = \sum_{\mu\in \mathbb{Y}_N} Q_{\lambda/\mu}(a_1, \cdots, a_N; q, t)Q_{\mu}(b_1, \cdots, b_N; q, t) .
\end{split}
\end{equation}

\begin{prop}[see \protect{\cite[Chapter VI]{book:491124}}]
For any finite sequences $a_1, \cdots, a_{M_1}$ and $b_1, \cdots, b_{M_2} \in \mathbb{C}$,
with $|a_ib_j| < 1$, $\forall 1\leq i \leq M_1$, $1\leq j \leq M_2$,
we have
\begin{equation} \label{eq:cay}
\sum_{\lambda\in \mathbb{Y}} P_{\lambda}(a_1, \cdots, a_{M_1}; q, t) Q_{\lambda}(b_1, \cdots, b_{M_2}; q, t) = \prod_{1\leq i\leq M_1, 1\leq j\leq M_2} \frac{\prod_{k=1}^{\infty}(1 - ta_ib_jq^{k-1}) }{\prod_{k=1}^{\infty}(1 - a_ib_jq^{k-1}) },
\end{equation}
\begin{equation} \label{eq:tran}
\sum_{\lambda\in \mathbb{Y}} P_{\mu/\lambda}(a_1, \cdots, a_{M_1}; q, t) P_{\lambda/\nu}(b_1, \cdots, b_{M_2}; q, t)
\\ = P_{\mu/\nu}(a_1, \cdots, a_{M_1}, b_1, \cdots, b_{M_2}; q, t) .
\end{equation}
\end{prop}

Let $\Psi^M$ be the set of all infinite families of sequences $\{\lambda^i\}_{i=1}^{\infty}$,
which satisfy
\begin{enumerate}
\item For $N\geq 1$, $\lambda^N \in \mathbb{Y}_{\min\{M, N\}}$.
\item For $N\geq 2$, the sequences $\lambda^N$ and $\lambda^{N-1}$ interlace: $\lambda_1^N \geq \lambda_1^{N-1}\geq \lambda_2^N \geq \cdots$.
\end{enumerate}

\begin{defn}
The infinite ascending \emph{Macdonald process} with positive parameters $M \in \mathbb{Z}_{>0}$, $\{a_i\}_{i=1}^{\infty}$, $\{b_i\}_{i=1}^M$,
$0< a_i < 1$, $0 < b_i < 1$,
is the distribution on $\Psi^M$, such that the marginal distribution for $\lambda^N$ is
\begin{equation} \label{eq:mpmar}
\pb(\lambda^N = \mu) =
\prod_{1\leq i\leq N, 1\leq j\leq M} \frac{\prod_{k=1}^{\infty}(1 - a_ib_jq^{k-1}) }{\prod_{k=1}^{\infty}(1 - ta_ib_jq^{k-1}) }
P_{\mu}(a_1, \cdots, a_N; q, t) Q_{\mu}(b_1, \cdots, b_M; q, t),
\end{equation}
and $\{\lambda^N\}_{N\geq 1}$ is a trajectory of a Markov chain with (backward) transition probabilities
\begin{equation} \label{eq:mptra}
\pb(\lambda^{N-1} = \mu \mid \lambda^N = \nu) = P_{\nu/\mu}(a_N;q,t) \frac{P_{\mu}(a_1, \cdots, a_{N-1}; q, t)}{P_{\nu}(a_1, \cdots, a_{N}; q, t)} .
\end{equation}
\end{defn}

\begin{rem}
The consistency of formulas (\ref{eq:mpmar}) and (\ref{eq:mptra}) follows from properties of Macdonald polynomials.
See \cite[Section 2.2.2]{borodin2014macdonald}, \cite[Section 3.1]{borodin2016observables} for more details.
\end{rem}

From this definition and (\ref{eq:tran}), the following Proposition follows by induction, cf.\ \cite[Definition 2.2.7]{borodin2014macdonald} \cite[Definition 3.2]{borodin2016observables} \cite[Definition 7.8]{borodin2012lectures}.
\begin{prop} \label{prop:jotdis}
Let $\{\lambda^N\}_{N\geq 1}$ be distributed as a Macdonald process with parameters $M \in \mathbb{Z}_{>0}$, $\{a_i\}_{i=1}^{\infty}$, $\{b_i\}_{i=1}^M$, where
each $0< a_i, b_i < 1$.
For integers $0< N_1 < \cdots < N_l$, and $\mu^1 \in \mathbb{Y}_{N_1}, \cdots, \mu^l \in \mathbb{Y}_{N_l}$,
their joint distribution is
\begin{multline} \label{eq:jotdis}
\pb(\lambda^{N_1} = \mu^1, \cdots, \lambda^{N_l} = \mu^{l}) =
\prod_{1\leq i\leq N_l, 1\leq j\leq M} \frac{\prod_{k=1}^{\infty}(1 - a_ib_jq^{k-1}) }{\prod_{k=1}^{\infty}(1 - ta_ib_jq^{k-1}) } \\
\times P_{\mu^1}(a_1, \cdots, a_{N_1}; q, t) \left( \prod_{i=1}^{l-1} P_{\mu^{i+1}/\mu^{i}}(a_{N_i + 1}, \cdots, a_{N_{i+1}}; q, t) \right) Q_{\mu^l}(b_1, \cdots, b_M; q, t) .
\end{multline}
\end{prop}

There is a limit transition which links Macdonald processes with $\mathbb{P}^{\alpha, M, \theta}$.

\begin{theorem}\protect{\cite[Theorem 2.8]{borodin2015general}} \label{thm:lim}
Given positive parameters $M\in \mathbb{Z}$, $\alpha$, $\theta \in \mathbb{R}_{>0}$.
Let the random family of sequences $\{\lambda^N\}_{N \geq 1}$, which takes values in $\Psi^M$,
be distributed as a Macdonald process with parameters $M$, $\{a_i\}_{i=1}^{\infty}$, $\{b_i\}_{i=1}^M$.
For $\epsilon > 0$, set
\begin{equation} \label{eq:lim:trans}
\begin{split}
q = & \exp(-\epsilon), \quad t = \exp(-\theta \epsilon) \\
a_i = t^{i - 1} = & \exp(-\theta \epsilon(i-1)) , \quad i = 1, 2, \cdots, \\
b_i = t^{\alpha + i - 1} = & \exp(-\theta \epsilon(\alpha + i-1)) , \quad 1 \leq i \leq M , \\
x_i^N(\epsilon) = \exp(-\epsilon \lambda_i^N), & \quad N = 1, 2, \cdots,  1 \leq i \leq \min\{M, N\} ,
\end{split}
\end{equation}
then as $\epsilon \rightarrow 0_+$, the finite dimensional distributions of $x^1, x^2, \cdots$ weakly converge to $\mathbb{P}^{\alpha, M, \theta}$.
\end{theorem}

\subsection{Differential operators}   \label{ssec:do}

We introduce operators acting on analytic symmetric functions. Such operators were
originally defined to act on $\Lambda$, and more algebraic discussions of them can
be found in \cite{feigin2009commutative} or \cite{Negut}. We will use them to
extract moments of $\mathbb{P}^{\alpha, M, \theta}$.



\begin{defn}   \label{def:opt:egn}
Let $r > 0$ and $q, t \in (0, 1)$ be parameters.
Let $\mathcal{D}_{-n}^N$ be an operator acting on symmetric analytic functions defined on $B_r^N$, where $B_r = \{ x \in \mathbb{C} : |x| < r\}$.
For any analytic symmetric $F : B_r^N \rightarrow \mathbb{C}$, since \eqref{eq:coll} is a basis of $\Lambda$, we can expand
\begin{equation}  \label{eq:op:defn1}
F(x_1, \cdots, x_N) = \sum_{\lambda \in \mathbb{Y}_N} c_{\lambda} P_{\lambda}(x_1, \cdots, x_N; q, t),
\end{equation}
where $c_{\lambda}$ are complex coefficients.
We set $\mathcal{D}_{-n}^N F : B_r^N \rightarrow \mathbb{C}$ to be the sum of the series
\begin{equation}  \label{eq:op:defn2}
\mathcal{D}_{-n}^N F(x_1, \cdots, x_N) = \sum_{\lambda \in \mathbb{Y}_N} c_{\lambda} \left( ( 1-t^{-n}) \sum_{i=1}^{N} (q^{\lambda_i}t^{-i+1})^n + t^{-Nn} \right) P_{\lambda}(x_1, \cdots, x_N; q, t) .
\end{equation}
\end{defn}

\begin{prop}   \label{prop:dowd}
The series \eqref{eq:op:defn1} and \eqref{eq:op:defn2} converges uniformly on compact subsets of $B_r^N$,
thus the operator $\mathcal{D}_{-n}^N$ is well-defined and linear.
Further, it is continuous in the following sense: for a sequence $\{F_i\}_{i=1}^{\infty}$ of symmetric analytic functions converging to $0$ uniformly on every compact subset of $B_r^N$, then so is the sequence $\{\mathcal{D}_{-n}^N F_i\}_{i=1}^{\infty}$.
\end{prop}



We need the following Lemma in the proof of Proposition \ref{prop:dowd}.
\begin{lemma}  \label{lemma:cout}
For any $r, \delta > 0$, $q, t \in (0, 1)$, and $N \in \mathbb{Z}_{>0}$, there is a constant $C_N > 0$ satisfying the following: for any symmetric analytic function $F : B_r^N \rightarrow \mathbb{C}$ given by (\ref{eq:op:defn1}), if $|F(x_1, \cdots, x_N)| \leq 1$ for every $x_1, \cdots, x_N \in B_r$,
then for every $x_1, \cdots, x_N \in B_{r(1 - \delta)}$,
and $\lambda \in \mathbb{Y}_N$,
one has $|c_{\lambda}P_{\lambda}(x_1, \cdots, x_N ; q, t)| < (1 - \delta^3)^{|\lambda|}C$.
\end{lemma}

\begin{proof}
Since each $P_\lambda$ is homogeneous, by rescaling $x_1, \cdots, x_N$ the decomposition \eqref{eq:op:defn1} is unchanged, and then it suffices to consider the case where $r = 1 + \delta$.

We define a scalar product for any two symmetric analytic functions $f, g$ on $B_{1 + \delta}^N$:
\begin{equation}
\langle f, g \rangle = \int_T f(z_1, \cdots, z_N) \overline{g(z_1, \cdots, z_N)} \Delta(z_1, \cdots, z_N; q ,t) d\vec{z} ,
\end{equation}
where $T$ is the torus $T = \left\{ (z_1, \cdots, z_N) \in \mathbb{C}^N : |z_i| = 1 \right\}$, $d \vec{z}$ is the uniform measure on $T$, and
\begin{equation}
\Delta(z_1, \cdots, z_N) = \prod_{i\neq j} \left(  \prod_{r=0}^{\infty} \frac{1 - z_i z_j^{-1}q^r}{1 - tz_i z_j^{-1}q^r}  \right) .
\end{equation}
This definition follows \cite[Section VI.9]{book:491124}, where one can find more discussions.
We immediately see that in $T$, $\Delta(z_1, \cdots, z_N)$ is always real and takes values in an interval $(\tau, 1)$, where $\tau > 0$ depends on $t$ and $q$.

By \cite[Chapter VI (9.5)]{book:491124}, the Macdonald polynomials $P_{\lambda}(\cdot; q, t)$ are pairwise orthogonal with respect to this scalar product.
Thus we have
\begin{equation}
\langle F, P_{\lambda}(\cdot; q, t) \rangle = c_{\lambda} \langle P_{\lambda}(\cdot; q, t), P_{\lambda}(\cdot; q, t) \rangle .
\end{equation}
By the Cauchy-Schwarz inequality, we have
\begin{equation}
\left| \langle F, P_{\lambda}(\cdot; q, t) \rangle \right|^2
\leq \langle F, F \rangle \langle P_{\lambda}(\cdot; q, t), P_{\lambda}(\cdot; q, t)\rangle ,
\end{equation}
then
\begin{equation}   \label{eq:put1}
|c_{\lambda}| \leq \sqrt{\frac{\langle F, F \rangle}{\langle P_{\lambda}(\cdot; q, t), P_{\lambda}(\cdot; q, t) \rangle }} .
\end{equation}
For $\langle F, F \rangle$, since $|F|$ is bounded by $1$ in $B_{1 + \delta}^N$, one has $\langle F, F \rangle \leq \langle 1, 1 \rangle$.

To lower bound $\langle P_{\lambda}(\cdot; q, t), P_{\lambda}(\cdot; q, t) \rangle$,
recall that $P_{\lambda}(z_1, \cdots, z_N; q, t) = \sum_{\substack{\mu \in \mathbb{Y}_N, \\ \mu \leq \lambda}} u_{\lambda\mu} m_{\mu}$, with $u_{\lambda \lambda} = 1$.
Denote $\mathcal{N}_{\mu}$ to be the number of different permutations of $(\mu_1, \cdots, \mu_N)$,
then we have
\begin{equation}   \label{eq:put2}
\langle P_{\lambda}(\cdot; q, t), P_{\lambda}(\cdot; q, t) \rangle
\geq
\int_T \tau \left| P_{\lambda}(z_1, \cdots, z_N; q, t) \right|^2 d\vec{z} =
(2\pi)^N \tau \sum_{\substack{\mu \in \mathbb{Y}_N, \\ \mu \leq \lambda}} |u_{\lambda \mu}|^2 \mathcal{N}_{\mu},
\end{equation}
where the last equality is due to the orthogonality of the monomials $m_{\mu}$, with respect to integrating against $d\vec{z}$ on $T$.

Also note that for any $x_1, \cdots, x_N \in B_{1 - \delta^2}$, we have
\begin{equation}   \label{eq:put3}
|P_{\lambda}(x_1, \cdots, x_N ; q, t)| \leq \sum_{\mu \leq \lambda} |u_{\lambda \mu}| \mathcal{N}_{\mu} (1 - \delta^2)^{|\lambda|} ,
\end{equation}
since for each $\mu \leq \lambda$, $m_{\mu}$ has degree $|\mu| = |\lambda|$.

Then by putting \eqref{eq:put1}, \eqref{eq:put2}, and \eqref{eq:put3} together we get
\begin{multline}
|c_{\lambda}P_{\lambda}(x_1, \cdots, x_N ; q, t)| \leq
\sqrt{\frac{\langle 1, 1 \rangle }{\tau(2\pi)^N }}  (1 - \delta^2)^{|\lambda|}
\frac{\sum_{\mu \leq \lambda} |u_{\lambda \mu}| \mathcal{N}_{\mu} }{\sqrt{ \sum_{\mu \leq \lambda} |u_{\lambda \mu}|^2 \mathcal{N}_{\mu}}}
\\ \leq
\sqrt{\frac{\langle 1, 1\rangle }{\tau(2\pi)^N }}  (1 - \delta^2)^{|\lambda|}
\sqrt{\sum_{\mu \leq \lambda} \mathcal{N}_{\mu} } ,
\end{multline}
where the last inequality is Cauchy-Schwarz.

Note that each $\mathcal{N}_{\mu}$ is bounded by $N!$, and the number of $\mu \in \mathbb{Y}_N, \mu \leq \lambda$ grows polynomially in $|\lambda|$.
Then we conclude that there is a constant $C_N$ such that
$
|c_{\lambda}P_{\lambda}(x_1, \cdots, x_N ; q, t)| \leq (1 - \delta^3)^{|\lambda|} C_N
$.
\end{proof}

\begin{proof}[Proof of Proposition \ref{prop:dowd}]
The uniform convergence of \eqref{eq:op:defn1} and \eqref{eq:op:defn2} on any compact subset of $B_r^N$ follows from Lemma \ref{lemma:cout}, and that
\begin{equation}
\left( ( 1-t^{-n}) \sum_{i=1}^{N} (q^{\lambda_i}t^{-i+1})^n + t^{-Nn} \right)
\end{equation}
is uniformly bounded.

For the continuity, expand
\begin{equation}
F_i(x_1, \cdots, x_N) = \sum_{\lambda \in \mathbb{Y}_N} c_{i, \lambda} P_{\lambda}(x_1, \cdots, x_N; q, t), \quad i = 1, 2, \cdots .
\end{equation}
By Lemma \ref{lemma:cout}, for any small $\delta > 0$ and $x_1, \cdots, x_N \in B_{r(1 - \delta)^2}^N$, we have
\begin{equation}
\left| \mathcal{D}_{-n}^N F_i(x_1, \cdots, x_N) \right|
\leq
\sum_{\lambda \in \mathbb{Y}_N }
\left|
( 1-t^{-n}) \sum_{j=1}^{N} (q^{\lambda_j}t^{-j+1})^n + t^{-Nn}
\right|
(1 - \delta^3)^{|\lambda|}C_N
\sup_{B_{r(1 - \delta)}^N} |F_i|  ,
\end{equation}
and this converges to $0$ as $i \rightarrow \infty$.
\end{proof}

So far we've defined the operator $\mathcal{D}_{-n}^N$ via its eigenvectors and eigenvalues, in Definition \ref{def:opt:egn}, and proved continuity.
Next, we evaluate the action of $\mathcal{D}_{-n}^N$ on a special class of functions, and give an expression as a nested contour integral.
\begin{prop}  \label{prop:act}
For any positive integer $m$, $\tilde{m}$, variables $w_1, \cdots, w_m$,
$\tilde{w}_1, \cdots, \tilde{w}_{\tilde{m}}$, and parameters $q, t$,
denote
\begin{equation}
\mathfrak{B}(w_1, \cdots, w_m; q, t) =
\frac{\sum_{i=1}^m \frac{w_m t^{m-i}}{w_i q^{m-i}} }{\left(1 - \frac{tw_2}{qw_1}\right) \cdots \left(1 - \frac{tw_m}{qw_{m-1}}\right)}
\prod_{i<j} \frac{\left(1 - \frac{w_i}{w_j}\right)\left(1 - \frac{qw_i}{tw_j}\right)}{\left(1 - \frac{w_i}{tw_j}\right)\left(1 - \frac{qw_i}{w_j}\right)} ,
\end{equation}
\begin{equation}
\mathfrak{F}(w_1, \cdots, w_m; \tilde{w}_1, \cdots, \tilde{w}_{\tilde{m}}; q, t) =
\prod_{i=1}^m \prod_{i'=1}^{\tilde{m}} \frac{w_i - t^{-1}q \tilde{w}_{i'}}{w_i - q \tilde{w}_{i'}} ,
\end{equation}

Let $f: B_r \rightarrow \mathbb{C}$ be analytic, such that $f(0) \neq 0$;
and $g: B_{r'} \rightarrow \mathbb{C}$ be analytic, with $r' > r$, such that $g(z)f(q^{-1}z) = f(z)$ for any $z \in B_r$.
The action of $\mathcal{D}_{-n}^N$ can be identified with a integral:
\begin{multline}  \label{eq:prop:act:int}
\mathcal{D}_{-n}^N \left( \prod_{i=1}^N f(a_i) \right)
= \left( \prod_{i=1}^N f(a_i) \right) \frac{(-1)^{n-1}}{(2 \pi \im)^n} 
\\
\times
\oint \cdots \oint
\mathfrak{B}(z_{1}, \cdots, z_{n}; q, t)
\mathfrak{F}(z_{1}, \cdots, z_{n}; a_1, \cdots, a_{N}; q, t)
\prod_{i=1}^n \frac{g(z_i) dz_i}{z_i}  ,
\end{multline}
where the contours are in $B_{r'}$ and nested: all enclose $0$ and $qa_{1}, \cdots, qa_N$, and $|z_i| < |tz_{i+1}|$ for each $1\leq i \leq n-1$.
\end{prop}


\begin{proof}
We prove Proposition \ref{prop:act} by introducing an algebraic version of the operator $\mathcal{D}_{-n}^N$, which acts on formal power series, and using \cite[Theorem 1.2]{Negut}, a formula for that algebraic operator.

Define $\mathbf{Z}$ to be a ring which
contains all elements of the form $\sum_{m \in \mathbb{Z}^n}d_{m} \prod_{i=1}^{n-1} \left(\frac{z_{i}}{z_{i+1}}\right)^{m_i} z_n^{m_n}$, where $d_m$ are complex coefficients, and satisfy that for some $m' \in \mathbb{Z}$, all coefficients $d_m$ with $\min_{i} m_i < m'$ vanish.
In other words, $\mathbf{Z}$ is the ring of Laurent power series in $z_n, \frac{z_i}{z_{i+1}}$, $i = 1, \cdots, n-1$.

Let $\Res: \mathbf{Z} \rightarrow \mathbb{C}$ be the ($\mathbb{C}$-linear) map, sending every such element to its coefficient of $\prod_{i=1}^n z_i^{-1}$.
This is an analogue of computing contour integrals around $0$.

Define $\tilde{\Lambda}$ to be the ring of symmetric formal power series with complex coefficients in countably many variables $a_1, a_2, \cdots $,
and $\tilde{\Lambda}[\mathbf{Z}]$ to be the ring of symmetric formal power series in countably many variables $a_1, a_2, \cdots$, with coefficients in $\mathbf{Z}$.

For any $\mathbf{F} \in \tilde{\Lambda}[\mathbf{Z}]$, it can be uniquely written as
\begin{equation}
\mathbf{F} = \sum_{\lambda \in \mathbb{Y}} \mathbf{c}_{\lambda} P_{\lambda}(\cdot; q, t),
\end{equation}
where each $\mathbf{c}_{\lambda} \in \mathbf{Z}$,
so we can also define $\Res$ as an operator $\tilde{\Lambda}[\mathbf{Z}] \rightarrow \tilde{\Lambda}$, by acting on each coefficient.

Define $\mathbf{D}_{-n} : \tilde{\Lambda} \rightarrow \tilde{\Lambda}$, through
\begin{equation}
\mathbf{D}_{-n} \left( \sum_{\substack{N \in \mathbb{Z}_+ \\ \lambda \in \mathbb{Y}_N}} c_{\lambda} P_{\lambda}(\cdot; q, t) \right) :=
\sum_{\substack{N \in \mathbb{Z}_+ \\ \lambda \in \mathbb{Y}_N}} c_{\lambda}
\left( ( 1-t^{-n}) \sum_{i=1}^{N} (q^{\lambda_i}t^{-i+1})^n + t^{-Nn}\right)
P_{\lambda}(\cdot; q, t) ,
\end{equation}
where each $c_{\lambda} \in \mathbb{C}$.

For each $k \in \mathbb{Z}_{+}$, denote $p_k = \sum_{i=1}^{\infty} a_i^k \in
\tilde{\Lambda}$. Note that any element in $\tilde{\Lambda}$ can be uniquely written
as a formal power series in $p_1, p_2, \cdots$; then $\frac{\partial}{\partial p_k}$
defines an operator from $\tilde{\Lambda}$ to itself. Now we present the ``integral form'' of the operator $\mathbf{D}_{-n}$, which is a reformulation of
\cite[Theorem 1.2]{Negut}:
\begin{multline}  \label{eq:prop:act:pf}
\mathbf{D}_{-n}
= (-1)^{n-1}
\Res \left[
\mathfrak{B}(z_{1}, \cdots, z_{n}; q, t)
\exp\left( \sum_{k=1}^{\infty}q^k (1 - t^{-k}) \frac{z_1^{-k} + \cdots + z_n^{-k}}{k} p_k \right)
\right.
\\
\left.
\times
\exp\left(
\sum_{k=1}^{\infty}(z_1^k + \cdots z_n^k)(1 - q^{-k}) \frac{\partial}{\partial p_k}
\right)
\prod_{i=1}^n z_i^{-1}
\right],
\end{multline}
where, inside $\mathfrak{B}(z_{1}, \cdots, z_{n}; q, t)$, the factors $\left(1 - \frac{tz_{i+1}}{qz_i}\right)^{-1}$ and $\left(\left(1 - \frac{z_i}{tz_j}\right)\left(1 - \frac{qz_i}{z_j}\right)\right)^{-1}$ 
are elements in $\mathbf{Z}$, by expanding in $z_i / z_{i+1}$ for $1 \leq i \leq n-1$, and in $z_i / z_j$ for $1 \leq i < j \leq n$ (that's where the condition $|z_i| < |t z_{i + 1}|$ arrives from).
The exponential expressions are operators from $\tilde{\Lambda}$ to $\tilde{\Lambda}[\mathbf{Z}]$, by expanding them into power series in the usual way (where $p_k$ is the operator of multiplying by $p_k$).

We need to emphasize that under the setting of Negut, \eqref{eq:prop:act:pf} 
is an identity of operators acting on \emph{polynomials}.
Since $\mathbf{D}_{-n}$ preserves degree, and the vector space of fixed degree polynomials is finitely generated, we can extend it to formal power series of $\tilde{\Lambda}$.
In the rest of the proof, we translate this algebraic statement into analytic formulation of Proposition \ref{prop:act}.

We rewrite the factor in \eqref{eq:prop:act:pf} as
\begin{multline}
\exp\left( \sum_{k=1}^{\infty}q^k (1 - t^{-k}) \frac{z_1^{-k} + \cdots + z_n^{-k}}{k} p_k \right)
\\
=
\prod_{i=1}^n \prod_{i'=1}^{\infty} \exp\left( \sum_{k=1}^{\infty} q^{k}(1 - t^{-k}) \frac{z_i^{-k} a_{i'}^k }{k} \right)
=
\prod_{i=1}^n \prod_{i'=1}^{\infty} \frac{1 - t^{-1}q \frac{a_{i'}}{z_i} }{1 - q \frac{a_{i'}}{z_i}}.
\end{multline}

Take any complex coefficient formal power series $f(x) = \sum_{i=0}^{\infty} s_i x^i$, with $s_0 = 1$.
Using the expansion $\ln(1 + x) = x - \frac{x^2}{2} + \frac{x^3}{3} - \cdots$, we define $\sum_{k=1}^{\infty} w_k x^k := \ln(f(x))$.
Then we have $f(x) = \exp(\sum_{i=1}^{\infty} w_i x^i)$.
Note that for any $k \in \mathbb{Z}_+$, any power series $h(p_k)$ in $p_k$, and any $C \in \mathbf{Z}$, by expanding the operators one can check that $\exp\left(\frac{C \partial}{\partial p_k} \right)h(p_k) = h(p_k + C)$.
Therefore,
\begin{multline}
\exp\left(
\sum_{k=1}^{\infty}(z_1^k + \cdots z_n^k)(1 - q^{-k}) \frac{\partial}{\partial p_k}
\right)
\prod_{i=1}^{\infty} f(a_i)
\\
=
\prod_{k=1}^{\infty} \left(
\exp\left(
(z_1^k + \cdots z_n^k)(1 - q^{-k}) \frac{\partial}{\partial p_k}
\right)
\exp\left( w_k p_k\right)
\right)
\\
=
\exp\left(
\sum_{k=1}^{\infty}w_k \left( (z_1^k + \cdots z_n^k)(1 - q^{-k}) + p_k \right)
\right)
= \prod_{i=1}^n \frac{f(z_i)}{f(q^{-1}z_i)} \prod_{i=1}^{\infty} f(x_i) ,
\end{multline}
where $f(x)^{-1}$ is understood as the power series $\sum_{j=0}^{\infty} \left( -\sum_{i=1}^{\infty} s_i x^i \right)^{j}$.

Applying both sides of (\ref{eq:prop:act:pf}) to $\prod_{i=1}^{\infty} f(a_i) \in \tilde{\Lambda}$, we obtain the following formula:
\begin{multline} \label{eq:prop:act:pf2}
\mathbf{D}_{-n} \left( \prod_{i=1}^{\infty} f(a_i) \right)
= (-1)^{n-1}
\left(\prod_{i=1}^{\infty} f(a_i)\right)
\\
\times
\Res \left[
\mathfrak{B}(z_{1}, \cdots, z_{n}; q, t)
\prod_{i=1}^n \left( \prod_{i'=1}^{\infty} \frac{1 - t^{-1}q \frac{a_{i'}}{z_i} }{1 - q \frac{a_{i'}}{z_i}}
\cdot
\frac{f(z_i)}{f(q^{-1}z_i)}
z_i^{-1} \right)
\right].
\end{multline}

Now we pass from infinitely many to finitely many variables.
Define $\tilde{\Lambda}_N$ to be the ring of symmetric formal power series (with complex coefficients) in $N$ variables $a_1, \cdots, a_N$,
and $\tilde{\Lambda}_N[\mathbf{Z}]$ to be the ring of symmetric formal power series in $a_1, \cdots, a_N$, with coefficients in $\mathbf{Z}$.
Then $\Res$ can also be defined to act on $\tilde{\Lambda}_N[\mathbf{Z}]$, with image in $\tilde{\Lambda}_N$.
Let $\pi_N: \tilde{\Lambda} \rightarrow \tilde{\Lambda}_N$ be the projection setting $0 = a_{N+1} = a_{N+2} = \cdots$;
and $\iota_N : \tilde{\Lambda}_N \rightarrow \tilde{\Lambda}$ be an embedding, sending each $P_{\lambda}(a_1, \cdots, a_N; q, t)$, $\lambda \in \mathbb{Y}_N$, to $P_{\lambda}(\cdot ; q, t)$.
Then $\pi_N \circ \iota_N$ is the identity map of $\tilde{\Lambda}_N$.

We claim that for any $F \in \tilde{\Lambda}$, if $\pi_N(F) = 0$, then $\pi_N(\mathbf{D}_{-n}(F)) = 0$.
Indeed, write $F = \sum_{\lambda \in \mathbb{Y}} u_{\lambda} P_{\lambda}(\cdot; q, t)$;
by $\pi_N(F) = 0$, $u_{\lambda} = 0$ for any $\lambda \in \mathbb{Y}_N$.
Since $P_{\lambda}(\cdot; q, t)$ are eigenvectors of $\mathbf{D}_{-n}$, the coefficient of $P_{\lambda}(\cdot; q, t)$ in $\mathbf{D}_{-n}(F)$ is zero for any $\lambda \in \mathbb{Y}_N$.
As $\pi_N$ sends every $P_{\lambda}(\cdot; q, t)$ to $0$ for $\lambda \in \mathbb{Y} \backslash \mathbb{Y}_N$, we conclude that $\pi_N(\mathbf{D}_{-n}(F)) = 0$.

Define $\mathbf{D}_{-n}^N : \tilde{\Lambda}_N \rightarrow \tilde{\Lambda}_N$ by $\mathbf{D}_{-n}^N = \pi_N \circ \mathbf{D}_{-n} \circ \iota_N$.
Then $\mathbf{D}_{-n}^N \circ \pi_N = \pi_N \circ \mathbf{D}_{-n}$, since for any $F \in \tilde{\Lambda}$,
$\pi_N ( \iota_N \circ \pi_N (F) - F) = \pi_N \circ \iota_N \circ \pi_N (F) - \pi_N (F) = 0$, thus $\pi_N ( \mathbf{D}_{-n} \circ \iota_N \circ \pi_N (F) - \mathbf{D}_{-n}(F)) = 0$, which is just $\mathbf{D}_{-n}^N \circ \pi_N (F) - \pi_N \circ \mathbf{D}_{-n}(F) = 0$.

It's also easy to check that for each $\lambda \in \mathbb{Y}_N$, $P_{\lambda}(a_1, \cdots, a_N; q, t)$ is an eigenvector of $\mathbf{D}_{-n}^N$, with eigenvalue
$\left( ( 1-t^{-n}) \sum_{i=1}^{N} (q^{\lambda_i}t^{-i+1})^n + t^{-Nn} \right)$.
Then for any power series in $\tilde{\Lambda}$ that converges on $B_r^N$, 
the action of $\mathbf{D}_{-n}^N$ is the same as the action of $\mathcal{D}_{-n}^N$.

Note that $\pi_N \prod_{i=1}^{\infty} f(a_i) = \prod_{i=1}^N f(a_i)$, for formal power series $f(x) = \sum_{i=0}^{\infty} s_i x^i$ with $s_0 = 1$.
Hence one has $\mathbf{D}_{-n}^N \prod_{i=1}^N f(a_i) = \pi_N \circ \mathbf{D}_{-n} \prod_{i=1}^{\infty} f(a_i)$; and by \eqref{eq:prop:act:pf2} this equals
\begin{multline} \label{eq:prop:act:pf3}
(-1)^{n-1}
\left( \prod_{i=1}^{N} f(a_i) \right)
\\
\times
\Res \left[
\mathfrak{B}(z_{1}, \cdots, z_{n}; q, t)
\mathfrak{F}(z_{1}, \cdots, z_{n}; a_1, \cdots, a_{N}; q, t)
\prod_{i=1}^n
\frac{f(z_i)}{f(q^{-1}z_i)}
z_i^{-1}
\right].
\end{multline}

When $f: B_r \rightarrow \mathbb{C}$ is an analytic function with $f(0) = 1$, the action of $\mathcal{D}_{-n}^N$ can be computed as \eqref{eq:prop:act:pf3}, by expanding $f(z_i)$ and $f(q^{-1}z_i)$ as power series.
The same is true for any analytic $f: B_r \rightarrow \mathbb{C}$ such that $f(0) \neq 0$, by multiplying a constant.
Further, the map $\Res$ can be identified with contour integrals of $z_1, \cdots, z_n$, with the part inside $\Res$ being the integrand, and the contours must be taken in a way such that:
first, the coefficient for each $\prod_{i=1}^N x_i^{m_i}$, which is a series in $\mathbf{Z}$, converges;
second, the power series for $a_1, \cdots, a_N$ converges.
It suffices to ensure that each $|tz_{i+1}| > |qz_{i}|$ for $1 \leq i \leq n-1$,
$|z_i| < |tz_j|$ for $1\leq i < j \leq n$, $|qa_{i'}| < |z_i|$ for $1 \leq i \leq n$, $1 \leq i' \leq N$,
and each power series $\frac{f(z_i)}{f(q^{-1}z_i)}$ converges.
These are guaranteed by the conditions given in Proposition \ref{prop:act}.
\end{proof}

\subsection{Joint higher order moments}   \label{ssec:hom}
In this section, we first prove a formula obtained by applying the operators
\begin{equation}
\frac{1}{\theta k_l} \epsilon^{-1}(t^{-N_{l}k_{l}} - \mathcal{D}_{-k_{l}}^{N_{l}}), \quad \cdots,\quad \frac{1}{\theta k_1} \epsilon^{-1}(t^{-N_{1}k_1} - \mathcal{D}_{-k_1}^{N_1})
\end{equation}
one by one to both sides of \eqref{eq:cay}.
The formula implies that these operators can be used to compute certain expectations of Macdonald processes .
Under \eqref{eq:lim:trans} and $\epsilon \rightarrow 0_+$, these expectations become moments of $\beta$--Jacobi corners processes.
We also evaluate the action of these operators as nested contour integrals.
Then in the next section, we take $\epsilon \rightarrow 0_+$ in the contour integrals as well, to obtain the integral formula of moments of $\beta$--Jacobi corners processes.
This is somewhat standard in the study of Macdonald processes, see \cite{borodin2014macdonald}, \cite{borodin2016observables}, and \cite{borodin2015general}.
However, the operators $\mathcal{D}_{-n}^N$ are very different from the ones used in those articles.

\begin{prop} \label{prop:ev:lhs}
Let $N_1 \leq \cdots \leq N_l, k_1, \cdots, k_l$ be positive integers, $b_1, \cdots, b_M \in \mathbb{C}$, 
and $0 < q, t < 1$.
Then we have the following identity of functions defined on $B_{r}^{N_l}$, $r = \min_{1\leq i \leq M} |b_i|^{-1}$:
\begin{multline} \label{eq:var:ind}
\prod_{i=1}^{l} \frac{1}{\theta k_i} (t^{-N_{i}k_{i}} - \mathcal{D}_{-k_{i}}^{N_i})
\prod_{1\leq i\leq N_l, 1\leq j\leq M} \frac{\prod_{k=1}^{\infty}(1 - ta_ib_jq^{k-1}) }{\prod_{k=1}^{\infty}(1 - a_ib_jq^{k-1}) } 
\\
= \sum_{\mu^1 \in \mathbb{Y}_{N_1}, \cdots, \mu^l \in \mathbb{Y}_{N_l} }
\prod_{i=1}^{l}
\left( \frac{1}{\theta k_i}  (t^{-k_i}-1) \sum_{j=1}^{N_i} (q^{\mu^i_j}t^{-j+1})^{k_i} \right)
Q_{\mu^l}(b_1, \cdots, b_M; q, t) \\
\times P_{\mu^1}(a_1, \cdots, a_{N_1}; q, t)
\prod_{1\leq i < l} \mathfrak{T}_{i+1\rightarrow i},
\end{multline}
where
\begin{equation}
\mathfrak{T}_{i+1\rightarrow i} =
\begin{cases}
    P_{\mu^{i+1}/\mu^{i}}(a_{N_i + 1}, \cdots, a_{N_{i+1}}; q, t),  & N_{i} < N_{i+1}\\
    \mathds{1}_{\mu^i = \mu^{i+1}},  & N_{i} = N_{i+1} .
\end{cases}
\end{equation}
The first product on the left hand side means applying the operators in the following way: 
first we apply $\frac{1}{\theta k_l} (t^{-N_{l}k_{l}} - \mathcal{D}_{-k_{l}}^{N_l})$ on variables $a_1, \cdots, a_{N_l}$, then $\frac{1}{\theta k_{l-1} } (t^{-N_{l-1}k_{l-1}} - \mathcal{D}_{-k_{l-1}}^{N_{l-1} })$ on variables $a_1, \cdots, a_{N_{l-1}}$, $\cdots$, $\frac{1}{\theta k_1} (t^{-N_{1}k_{1}} - \mathcal{D}_{-k_{1}}^{N_1})$ on variables $a_1, \cdots, a_{N_1}$.
\end{prop}

\begin{proof}
The proof is similar to \cite[Proposition 4.9]{borodin2016observables}. We argue by induction on $l$.
For $l = 1$,
in the Cauchy identity \eqref{eq:cay}, set $M_1=N_1$, $M_2=M$.
By applying the operator $\frac{1}{\theta k_1}(t^{-N_1k_1} - \mathcal{D}_{-k_1}^{N_1})$,
acting on variables $a_1, \cdots, a_{N_1}$, to both sides,
we get the desired equation.

For general $l$,
we assume that the statement is true for $l - 1$; specifically, we have that
\begin{multline} \label{eq:var:ind:mid}
\prod_{i=2}^{l} \frac{1}{\theta k_i} (t^{-N_{i}k_{i}} - \mathcal{D}_{-k_{i}}^{N_i})
\prod_{1\leq i\leq N_l, 1\leq j\leq M} \frac{\prod_{k=1}^{\infty}(1 - ta_ib_jq^{k-1}) }{\prod_{k=1}^{\infty}(1 - a_ib_jq^{k-1}) } \\
= \sum_{\mu^{2} \in \mathbb{Y}_{N_{2}}, \cdots, \mu^l \in \mathbb{Y}_{N_l} }
\prod_{i=2}^{l}
\left( \frac{1}{\theta k_i}  (t^{-k_i}-1) \sum_{j=1}^{N_i} (q^{\mu^i_j}t^{-j+1})^{k_i} \right)
Q_{\mu^l}(b_1, \cdots, b_M; q, t) \\
\times P_{\mu^{2}}(a_1, \cdots, a_{N_{2}}; q, t)
\prod_{2\leq i < l} \mathfrak{T}_{i+1\rightarrow i} .
\end{multline}
If $N_1 = N_{2}$,
\begin{equation}
P_{\mu^{2}}(a_1, \cdots, a_{N_{2}}; q, t)  = \sum_{\mu^1 \in \mathbb{Y}_{N_1}} P_{\mu^{1}}(a_1, \cdots, a_{N_{1}}; q, t) \mathds{1}_{\mu^1 = \mu_{2}} ;
\end{equation}
If $N_1 < N_{2}$, by using \eqref{eq:mpexp} repeatedly, we get
\begin{multline}
P_{\mu^{2}}(a_1, \cdots, a_{N_{2}}; q, t)  \\
= \sum_{\mu^1 \in \mathbb{Y}_{N_1}} P_{\mu^{1}}(a_1, \cdots, a_{N_{1}}; q, t) \sum_{\nu^1 \in \mathbb{Y}_{N_1 + 1}, \cdots, \nu^{N_{2} - N_1} \in \mathbb{Y}_{N_{2}}} \prod_{1\leq i \leq N_{2} - N_1} P_{\nu^{i}/\nu^{i-1}} (a_{N_{1} + i}; q, t) \\
= \sum_{\mu^1 \in \mathbb{Y}_{N_1}} P_{\mu^{1}}(a_1, \cdots, a_{N_{1}}; q, t)
P_{\mu^{2}/\mu^{1}} (a_{N_{1}+1}, \cdots, a_{N_{2}}; q, t),
\end{multline}
where $\nu^0 = \mu_r$ and $\nu^{N_{2} - N_1} = \mu_{2}$,
and the last line follows from (\ref{eq:tran}).

In either case we have
\begin{equation}
P_{\mu^{2}}(a_1, \cdots, a_{N_{2}}; q, t)
= \sum_{\mu^1 \in \mathbb{Y}_{N_1}} P_{\mu^{1}}(a_1, \cdots, a_{N_{1}}; q, t)\mathfrak{T}_{i+1\rightarrow i}.
\end{equation}
Plug this into (\ref{eq:var:ind:mid}) and
apply the operator $\frac{1}{\theta k_1} (t^{-N_{1}k_{1}} - \mathcal{D}_{-k_{1}}^{N_1})$
to both sides,
we immediately obtain \eqref{eq:var:ind}.
\end{proof}

Now we evaluate the action of the operators, in the special case where $b_i = t^{\alpha + i - 1}$ for $1 \leq i \leq M$, by using Proposition \ref{prop:act} multiple times.
\begin{prop} \label{prop:act:mul}
Let $\mathfrak{B}$ and $\mathfrak{F}$ as defined in Proposition \ref{prop:act}.
In addition,
for any positive integer $m$, $\tilde{m}$, variables $w_1, \cdots, w_m$,
$\tilde{w}_1, \cdots, \tilde{w}_{\tilde{m}}$, and parameters $q, t$,
denote
\begin{equation}
\mathfrak{C}(w_1, \cdots, w_m; \tilde{w}_1, \cdots, \tilde{w}_{\tilde{m}}; q, t) =
\prod_{i=1}^m \prod_{i'=1}^{\tilde{m}} \frac{\left(1 - \frac{w_i}{\tilde{w}_{i'}}\right) \left(1 - \frac{qw_i}{t\tilde{w}_{i'}}\right)}  {\left(1 - \frac{w_i}{t\tilde{w}_{i'} }\right)\left(1 - \frac{qw_i}{\tilde{w}_{i'}}\right)}  .
\end{equation}
Then for fixed real parameters $0 < q, t < 1$, $\alpha \in \mathbb{R}$,
and positive integers
$N_1 \leq \cdots \leq N_l, k_1, \cdots, k_l$, and $\alpha + M > k_1 + \cdots + k_l$,
we have the following identity of functions defined on $B_{1}^{N_l}$:
\begin{multline}  \label{eq:act:mul:exp}
\prod_{i=1}^{l} \mathcal{D}_{-k_{i}}^{N_i}
\prod_{i = 1}^{N_l} \frac{\prod_{k=1}^{\infty}(1 - a_i t^{\alpha + M} q^{k-1}) }{\prod_{k=1}^{\infty}(1 - a_i t^{\alpha} q^{k-1}) }
= 
\frac{(-1)^{k_1 + \cdots + k_l - l}}{(2 \pi \im)^{k_1 + \cdots + k_l}}
\oint \cdots \oint
\prod_{i = 1}^{N_l} \frac{\prod_{k=1}^{\infty}(1 - a_i t^{\alpha + M} q^{k-1}) }{\prod_{k=1}^{\infty}(1 - a_i t^{\alpha} q^{k-1}) } \\
\times
\prod_{i=1}^{l}\mathfrak{B}(z_{i,1}, \cdots, z_{i,k_i}; q, t)
\prod_{i=1}^{l}\mathfrak{F}(z_{i,1}, \cdots, z_{i,k_i}; a_1, \cdots, a_{N_i}; q, t) \\
\times
\prod_{i<j} \mathfrak{C}(z_{i,1}, \cdots, z_{i,k_i}; z_{j,1}, \cdots, z_{j,k_j}; q, t)
\prod_{i=1}^l \prod_{i'=1}^{k_i} \left( \frac{1 - q^{-1}t^{\alpha}z_{i,i'}}{1 - q^{-1}t^{\alpha + M} z_{i,i'}} \frac{dz_{i,i'}}{z_{i,i'}} \right)  ,
\end{multline}
where each operator $\mathcal{D}_{-k_{i}}^{N_i}$ acts on variables $a_1, \cdots, a_{N_i}$,
and the contours are nested and satisfy the following:
for each $1\leq i\leq l$ and $1 \leq i' < k_i$ we have $|z_{i,i'}| < t|z_{i,i'+1}|$;
and for each $1\leq i < l$, we have $|z_{i,k_i}| < t|z_{i+1,1}|$ ;
also, $q < |z_{1, 1}|$, and $|z_{l, k_l}| < qt^{-\alpha - M}$.
\end{prop}

\begin{rem}
The constraints we impose on the contours imply that each of the contours encloses $0$ and all $qa_i$, but none of them encloses $qt^{-\alpha - M}$.
The requirement that $\alpha + M > k_1 + \cdots + k_l$ ensures the existence of the desired contours.
\end{rem}

\begin{proof}[Proof of Proposition \ref{prop:act:mul}]
We prove by induction on $l$.

For the base case where $l = 1$, we apply Proposition \ref{prop:act} to the function
\begin{equation}
f(x) = \frac{\prod_{k=1}^{\infty}(1 - xt^{\alpha + M}q^{k-1}) }{\prod_{k=1}^{\infty}(1 - xt^{\alpha}q^{k-1}) }  .
\end{equation}
Specifically, $f(x)$ is analytic in $B_{t^{-\alpha}}$, with $f(0) = 1$.
And 
\begin{equation}
\frac{f(x)}{f(q^{-1}x)}
=
\frac{1 - xt^{\alpha}q^{-1}}{1 - xt^{\alpha + M}q^{-1}}
\end{equation}
is analytic in $B_{qt^{-\alpha - M}}$.
For $a_1,\cdots, a_{N_1} \in B_1 \subset B_{t^{-\alpha}}$,
we can construct contours of $z_{1,1}, \cdots, z_{1, k_1}$ such that $q < |z_{1, 1}|$, $|z_{1, k_1}| < qt^{-\alpha - M}$, and $|z_{1, i}| < t |z_{1, i+1}|$ for each $1 \leq i < k_1$, which satisfies the requirements in Proposition \ref{prop:act}.
The expression given by Proposition \ref{prop:act} is precisely (\ref{eq:act:mul:exp}) for $l=1$.

For more general $l \geq 2$,
assume that the statement is true for $l-1$;
then we have the following identity for functions defined on $B_1^{N_l}$:
\begin{multline}  \label{eq:act:mul:pf1}
\prod_{i=2}^{l} \mathcal{D}_{-k_{i}}^{N_i}
\prod_{i=1}^{N_l} \frac{\prod_{k=1}^{\infty}(1 - a_it^{\alpha + M}q^{k-1}) }{\prod_{k=1}^{\infty}(1 - a_it^{\alpha}q^{k-1}) }
=
\frac{(-1)^{k_2 + \cdots + k_l - l}}{(2 \pi \im)^{k_2 + \cdots + k_l}}
\oint \cdots \oint
\prod_{i=1}^{N_l} \frac{\prod_{k=1}^{\infty}(1 - a_it^{\alpha + M}q^{k-1}) }{\prod_{k=1}^{\infty}(1 - a_it^{\alpha}q^{k-1}) }
\\
\times
\prod_{i=2}^{l}\mathfrak{B}(z_{i,1}, \cdots, z_{i,k_i}; q, t)
\prod_{i=2}^{l}\mathfrak{F}(z_{i,1}, \cdots, z_{i,k_i}; a_1, \cdots, a_{N_i}; q, t) \\
\times
\prod_{1<i<j} \mathfrak{C}(z_{i,1}, \cdots, z_{i,k_i}; z_{j,1}, \cdots, z_{j,k_j}; q, t)
\prod_{i=2}^l \prod_{i'=1}^{k_i} \left( \frac{1 - q^{-1}t^{\alpha}z_{i,i'}}{1 - q^{-1}t^{\alpha + M} z_{i,i'}} \cdot \frac{1}{z_{i,i'}} \right)
\prod_{i=2}^l \prod_{i'=1}^{k_i} dz_{i, i'} ,
\end{multline}
where the contours are constructed in the following way:
for each $2\leq i\leq l$ and $1 \leq i' < k_i$ we have $|z_{i,i'}| < t|z_{i,i'+1}|$;
and for each $2\leq i < l$, we have $|z_{i,k_i}| < t|z_{i+1,1}|$ ;
also, $q < |z_{2, 1}|$, and $|z_{l, k_l}| < qt^{-\alpha - M}$.
Since $\alpha + M > k_1 + \cdots + k_l$, we can let the contours move continuously and require that $|z_{2, 1}| > qt^{-k_1 - 1}$. 

Now apply the operator $\mathcal{D}_{-k_1}^{N_1}$ to both sides of (\ref{eq:act:mul:pf1}), acting on variables $a_1, \cdots, a_{N_1}$.
On the right hand side, we can change the order of the operator $\mathcal{D}_{-k_1}^{N_1}$ and the integral, by the linearity and continuity of $\mathcal{D}_{-k_1}^{N_1}$ stated in Proposition \ref{prop:dowd}.
To apply $\mathcal{D}_{-k_1}^{N_1}$ to the integrand, 
we just need to consider the following function
\begin{equation}
f(x) = \frac{\prod_{k=1}^{\infty} (1 - x t^{\alpha + M} q^{k-1}) }{\prod_{k=1}^{\infty} (1-x t^{\alpha} q^{k-1})} \prod_{i=2}^l \prod_{i'=1}^{k_i} \frac{z_{i,i'} - t^{-1}qx}{z_{i,i'} - qx} .
\end{equation}
This $f(x)$ is analytic in $B_1$, with $f(1) = 0$.
Also note that, as $|z_{i, i'}| > qt^{-k_1 - 1}$ for each $2\leq i \leq l$ and $1\leq i' \leq k_i$, the function
\begin{equation}
\frac{f(x)}{f(q^{-1}x)}
=
\frac{1 - xt^{\alpha}q^{-1}}{1 - xt^{\alpha + M}q^{-1}}
\prod_{i=2}^l \prod_{i'=1}^{k_i} \frac{z_{i,i'} - t^{-1}qx}{z_{i,i'} - qx}
\cdot
\frac{z_{i,i'} - x}{z_{i,i'} - t^{-1}x}
\end{equation}
is analytic inside $B_{qt^{-k_1}}$.
By Proposition \ref{prop:act}, one has
\begin{multline}  \label{eq:act:mul:pf2}
\mathcal{D}_{-k_{1}}^{N_1}
\left(
\prod_{i=1}^{N_l} \frac{\prod_{k=1}^{\infty}(1 - a_i t^{\alpha+M} q^{k-1}) }{\prod_{k=1}^{\infty}(1 - a_i t^{\alpha} q^{k-1}) }
\prod_{i=2}^{l}\mathfrak{F}(z_{i,1}, \cdots, z_{i,k_i}; a_1, \cdots, a_{N_i}; q, t)
\right)
\\
=
\prod_{i=1}^{N_l} \frac{\prod_{k=1}^{\infty}(1 - a_i t^{\alpha+M} q^{k-1}) }{\prod_{k=1}^{\infty}(1 - a_i t^{\alpha} q^{k-1}) }
\prod_{i=2}^{l}\mathfrak{F}(z_{i,1}, \cdots, z_{i,k_i}; a_1, \cdots, a_{N_i}; q, t)
\\
\times \frac{(-1)^{k_1 - 1}}{(2\pi \im)^{k_1}} \oint \cdots \oint
\mathfrak{B}(z_{1,1}, \cdots, z_{1,k_1}; q, t)
\mathfrak{F}(z_{1,1}, \cdots, z_{1,k_1}; a_1, \cdots, a_{N_1}; q, t)
\\
\times
\prod_{i=2}^l \mathfrak{C}(z_{1,1}, \cdots, z_{1,k_1}; z_{i,1}, \cdots, z_{i,k_i}; q, t)
\prod_{i'=1}^{k_1} \left( \frac{1 - q^{-1}t^{\alpha}z_{1,i'}}{1 - q^{-1}t^{\alpha + M} z_{1,i'}} \frac{dz_{1,i'}}{z_{1,i'}} \right)  ,
\end{multline}
for any $a_1, \cdots, a_{N_1} \in B_1$,
and the contours are constructed such that
for each $1\leq i < k_1$ we have $|z_{1,i}| < t|z_{1,i+1}|$,
$q < |z_{1, 1}|$, and $|z_{1, k_1}| < qt^{-k_1}$.

Putting (\ref{eq:act:mul:pf1}) and (\ref{eq:act:mul:pf2}) together we get exactly (\ref{eq:act:mul:exp}).
\end{proof}

\subsection{Limit transition}  \label{ssec:lt}
We finish the proof of Theorem \ref{thm:var:dis} in this section.
We changes variables as \eqref{eq:lim:trans}, and send $\epsilon \rightarrow 0_+$ in formula \eqref{eq:act:mul:exp}.
In the integral we want to set
\begin{equation}  \label{eq:lim:trans:more}
z_{i,i'} = \exp(\epsilon u_{i,i'}), \quad 1\leq i \leq l, 1\leq i' \leq k_i ,
\end{equation}
and all the contours $u_{i,i'}$ are nested in a certain way to give valid contours of $z_{i, i'}$.

There is a difficulty in doing this: \eqref{eq:lim:trans:more} implies that as $\epsilon \rightarrow 0_+$, $z_{i, i'}$ approaches to $1$.
However, originally each $z_{i, i'}$ encloses $0$.
The idea to resolve this is to split each contour of $z_{i, i'}$ (\ref{eq:act:mul:exp}) into two: one enclosing $0$ and another enclosing all of $qa_1, \cdots, qa_{N_l}$.
It turns out that most terms with contours enclosing $0$ are evaluated to zero or cancel out.

In more details, we associate each $z_{i,i'}$ with two contours $\mathfrak{U}_{i,i'}$ and $\mathfrak{V}_{i,i'}$, satisfying:
for each $1 \leq i \leq l$ and $1 \leq i' < k_i$, $\mathfrak{U}_{i, i'}$ is inside $t\mathfrak{U}_{i, i'+1}$, $\mathfrak{V}_{i, i'}$ is inside $t\mathfrak{V}_{i, i'+1}$;
for each $1\leq i < l$, $\mathfrak{U}_{i, k_i}$ is inside $t\mathfrak{U}_{i + 1, 1}$, $\mathfrak{V}_{i, k_i}$ is inside $t\mathfrak{V}_{i+1, 1}$.
Also, each of $\mathfrak{U}_{i,i'}$ encloses $0$, but none of $qa_1, \cdots, qa_{N_l}$,
while each of $\mathfrak{V}_{i, i'}$ encloses $qa_1, \cdots, qa_{N_l}$, but not $0$.
All of these contours are inside $B_{qt^{-\alpha-M}}$.
Such contours exist as long as $1 - t$ is small enough.

Let $\Pi$ be the power set of $\{z_{i,i'} | 1\leq i \leq l, 1\leq i' \leq k_i\}$, then for each $\Upsilon \in \Pi$, denote $\Upsilon_{i,i'}$ to be $\mathfrak{U}_{i,i'}$ if $z_{i,i'} \in \Upsilon$, and $\mathfrak{V}_{i,i'}$ if $z_{i,i'} \not\in \Upsilon$.
Let
\begin{multline}   \label{eq:dcfu}
\mathfrak{Q}_{\Upsilon} =
\oint_{\Upsilon_{1,1}} \cdots \oint_{\Upsilon_{l, k_l}}
\prod_{i=1}^{l}\mathfrak{B}(z_{i,1}, \cdots, z_{i,k_i}; q, t)
\prod_{i=1}^{l}\mathfrak{F}(z_{i,1}, \cdots, z_{i,k_i}; a_1, \cdots, a_{N_i}; q, t) \\
\times
\prod_{i<j} \mathfrak{C}(z_{i,1}, \cdots, z_{i,k_i}; z_{j,1}, \cdots, z_{j,k_j}; q, t)
\prod_{i=1}^l \prod_{i'=1}^{k_i} \left( \frac{1 - q^{-1}t^{\alpha}z_{i,i'}}{1 - q^{-1}t^{\alpha + M} z_{i,i'}} \frac{dz_{i,i'}}{z_{i,i'}} \right)  .
\end{multline}
Then \eqref{eq:act:mul:exp} can be written as
\begin{equation}   \label{eq:dcf}
\frac{(-1)^{k_1 + \cdots + k_l - l}}{(2 \pi \im)^{k_1 + \cdots + k_l}}
\prod_{i = 1}^{N_l} \frac{\prod_{k=1}^{\infty}(1 - a_i t^{\alpha + M} q^{k-1}) }{\prod_{k=1}^{\infty}(1 - a_i t^{\alpha} q^{k-1}) } \\
\sum_{\Upsilon \in \Pi}
\mathfrak{Q}_{\Upsilon} .
\end{equation}

In the following Lemmas, we show that under the limit transition \eqref{eq:lim:trans}, 
most of $\mathfrak{Q}_{\Upsilon}$ converge to zero, or are canceled out with one another, and the only left one $\mathfrak{Q}_{\emptyset}$.

The first Lemma is an extension of \cite[Appendix A, Lemma 5]{Fyodorov2016}.
\begin{lemma}  \label{lemma:oform}
$\mathfrak{Q}_{\Upsilon} = 0$ unless
for each $1 \leq i \leq l$,
$\Upsilon \bigcap \left\{ z_{i,1}, \cdots, z_{i, k_i} \right\}$
is either empty or of the form
$\left\{ z_{i,s_i}, z_{i,s_i + 1}, \cdots, z_{i, r_i} \right\}$,
where $1 \leq s_i \leq r_i \leq k_i$.
\end{lemma}

\begin{proof}
Let us order the variables $z_{i, i'}$ as in the nesting of the contours, from inner to outer:
$z_{1,1}, z_{1, 2}, \cdots, z_{1, k_1}, z_{2,1}, \cdots, z_{l, k_l}$.
For any $\Upsilon \in \Pi$, we evaluate the integral \eqref{eq:dcfu} for those variables that are in $\Upsilon$,
and in that order.
The order is to ensure that the integrals are evaluated from inner to outer, then when evaluate the integral for each variable, the only possible pole is at the origin.

For any given $1 \leq i \leq l$, suppose that we've evaluated the integrals for all $z_{j, i'} \in \Upsilon$, $1 \leq j < i$, $1 \leq i' \leq k_j$.
Let $1 \leq s_{i} \leq k_{i}$ be the smallest index (if any) such that $z_{i, s_{i}}$ belongs to $\Upsilon$.
Now let's evaluate the integral $z_{i,s_{i}}$.
This is done by multiplying the integrand by 
$z_{i, s_{i}}$, and sending $z_{i, s_{i}} \rightarrow 0$.
Most factors are computed in an obvious way, except for
\begin{equation}  \label{eq:lemma:sus:fa}
\frac{\sum_{i'=1}^{k_{i}} \frac{z_{i,k_{i}} t^{k_{i}-i'}}{z_{i,i'} q^{k_{i}-i'}} }{\left(1 - \frac{tz_{i,2}}{qz_{i,1}}\right) \cdots \left(1 - \frac{tz_{i,k_{i}}}{qz_{i,k_{i}-1}}\right)}
\end{equation}
in $\mathfrak{B}(z_{i,1}, \cdots, z_{i,k_{i}}; q, t)$.
As $z_{i, s_{i}} \rightarrow 0$, it becomes
\begin{equation}  \label{eq:intoo}
\frac{ - \frac{z_{i,k_{i}} t^{k_{i} - s_{i} - 1}}{z_{i,s_{i} + 1} q^{k_{i}-s_{i}-1}} }{\left(1 - \frac{tz_{i,2}}{qz_{i,1}}\right) \cdots \left(1 - \frac{tz_{i,s_{i}-1}}{qz_{i,s_{i}-2}}\right)
\left(1 - \frac{tz_{i,s_{i}+2}}{qz_{i,s_{i}+1}}\right) \cdots \left(1 - \frac{tz_{i,k_{i}}}{qz_{i,k_{i}-1}}\right)}  .
\end{equation}
If there is any $s_i < i' \leq k_i$ such that $z_{i, i'} \not\in \Upsilon$, we let $r_i$, with $s_i \leq r_{i} < k_i$, be the index satisfying that $z_{i, i'} \in \Upsilon$ for any $s_i \leq i' \leq r_{i}$, but $z_{i, r_{i} + 1} \neq \Upsilon$.
Then we evaluate the integral of $z_{i, i'}$ for all $s_{i} < i' \leq r_{i}$.
This is done in the same way as evaluating the integral of $z_{i, s_{i}}$: we multiply the integrand by $z_{i, i'}$, then send $z_{i, i'} \rightarrow 0$.
Under $z_{i, i'} \rightarrow 0$, for all $s_{i} < i' \leq r_{i}$, the factor \eqref{eq:intoo} finally becomes
\begin{equation}  \label{eq:intooo}
\frac{ (-1)^{r_{i} - s_{i} + 1} \frac{z_{i,k_{i}} t^{k_{i} - r_{i} - 1}}{z_{i,r_{i} + 1} q^{k_{i}-r_{i}-1}} }{\left(1 - \frac{tz_{i,2}}{qz_{i,1}}\right) \cdots \left(1 - \frac{tz_{i,s_{i}-1}}{qz_{i,s_{i}-2}}\right)
\left(1 - \frac{tz_{i,r_{i}+2}}{qz_{i,r_{i}+1}}\right) \cdots \left(1 - \frac{tz_{i,k_{i}}}{qz_{i,k_{i}-1}}\right)}  .
\end{equation}
If there is any $i'$, such that $r_i < i' \leq k_i$ and $z_{i, i'} \in \Upsilon$, let $w_{i}$ be the smallest one.
The next integral to evaluate is the one of $z_{i, w_{i}}$.
However, when sending $z_{i, w_{i}} \rightarrow 0$, the factor \eqref{eq:intooo} becomes zero. Thus the integral is zero.
\end{proof}

Using the arguments in the proof of Lemma \ref{lemma:oform}, we have the following identity.
\begin{lemma}   \label{lemma:new}
Let $\Pi'$ be the collection of all $\Upsilon \in \Pi$, such that for each $1 \leq i \leq l$,
$\Upsilon \bigcap \left\{ z_{i,1}, \cdots, z_{i, k_i} \right\}$
is either empty or of the form
$\left\{ z_{i,s_i}, z_{i,s_i+1} \cdots, z_{i, r_i} \right\}$
for some $1 \leq s_i \leq r_i \leq k_i$.
Then for any $\Upsilon \in \Pi'$, there is
\begin{multline}  \label{eq:new:0}
\mathfrak{Q}_{\Upsilon}
=
\oint \cdots \oint_{\{ \Upsilon_{i, i'} : z_{i, i'} \not\in \Upsilon \}}
\prod_{i=1}^{l}\mathfrak{Y}_i(\Upsilon)
\prod_{i=1}^{l}\mathfrak{F}(\{ z_{i,1}, \cdots, z_{i,k_i}\} \backslash \Upsilon ; a_1, \cdots, a_{N_i}; q, t) \\
\times
\prod_{i<j} \mathfrak{C}(\{ z_{i,1}, \cdots, z_{i,k_i} \} \backslash \Upsilon ; \{ z_{j,1}, \cdots, z_{j,k_j} \} \backslash \Upsilon ; q, t)
\prod_{z_{i,i'} \not\in \Upsilon } \left( \frac{1 - q^{-1}t^{\alpha}z_{i,i'}}{1 - q^{-1}t^{\alpha + M} z_{i,i'}} \frac{dz_{i,i'}}{z_{i,i'}} \right)  ,
\end{multline}
where
\begin{multline}  \label{eq:new:1}
\mathfrak{Y}_i(\Upsilon) = 
(2\pi\im)^{r_i - s_i + 1} (-1)^{r_i - s_i + \mathds{1}_{r_i \neq k_i}} 
\\
\times
\frac{\frac{z_{i,k_i} t^{k_i - r_i - \mathds{1}_{r_i \neq k_i} }}{z_{i,r_i + \mathds{1}_{r_i \neq k_i} } q^{k_i-r_i- \mathds{1}_{r_i \neq k_i} }} t^{-N_i(r_i - s_i + 1)}}{
\prod_{\substack{1\leq i' < s_i - 1, \\ \mathrm{or}\; r_i + 1 \leq i' < k_i}}\left(1 - \frac{tz_{i,i'+1}}{qz_{i,i'}}\right)
}
\prod_{i' < j': z_{i,i'}, z_{i,j'} \not\in \Upsilon }
\frac{ \left(1 - \frac{z_{i,i'}}{z_{i,j'}}\right) \left(1 - \frac{qz_{i,i'}}{tz_{i,j'}}\right) }{ \left(1 - \frac{z_{i,i'}}{tz_{i,j'}}\right)\left(1 - \frac{qz_{i,i'}}{z_{i,j'}}\right) }
,
\end{multline}
if $\left\{ z_{i,1}, \cdots, z_{i, k_i} \right\} \bigcap \Upsilon = \left\{ z_{i,s_i}, \cdots, z_{i, r_i} \right\}$ for some $1 \leq s_i \leq r_i \leq k_i$;
and $\mathfrak{Y}_i(\Upsilon) = \mathfrak{B}(z_{i,1}, \cdots, z_{i, k_i}; q, t)$ if $\left\{ z_{i,1}, \cdots, z_{i, k_i} \right\} \bigcap \Upsilon = \emptyset$.
\end{lemma}

Based on Lemma \ref{lemma:oform}, we use \eqref{eq:new:0} to compute the action of $t^{- N k} - \mathcal{D}_{-k}^{N}$.
\begin{lemma}  \label{lemma:tipi}
Let $\tilde{\Pi} \subset \Pi'$ contain all $\Upsilon$ such that $\left\{ z_{i,1}, \cdots, z_{i, k_i} \right\} \not\subset \Upsilon$ for any $1 \leq i \leq l$.
Then we have
\begin{multline}   \label{eq:lemma:tipi}
\prod_{i=1}^{l} \left( t^{-N_i k_i} - \mathcal{D}_{-k_{i}}^{N_i} \right)
\prod_{i = 1}^{N_l} \frac{\prod_{k=1}^{\infty}(1 - a_i t^{\alpha + M} q^{k-1}) }{\prod_{k=1}^{\infty}(1 - a_i t^{\alpha} q^{k-1}) }
\\
=
\frac{(-1)^{k_1 + \cdots + k_l}}{(2 \pi \im)^{k_1 + \cdots + k_l}}
\prod_{i = 1}^{N_l} \frac{\prod_{k=1}^{\infty}(1 - a_i t^{\alpha + M} q^{k-1}) }{\prod_{k=1}^{\infty}(1 - a_i t^{\alpha} q^{k-1}) }
\sum_{\Upsilon \in \tilde{\Pi}}
\mathfrak{Q}_{\Upsilon}
.
\end{multline}
\end{lemma}

\begin{proof}
The difference between the left hand side of \eqref{eq:act:mul:exp} and \eqref{eq:lemma:tipi} is in shifts by $t^{-N_i k_i}$, which are from $\left\{z_{i,1}, \cdots, z_{i,k_i} \right\} \subset \Upsilon$.

For any $U \subset \{1, \cdots, l\}$, let $\Pi_U \subset \Pi'$ contain all $\Upsilon$, satisfying that $\{ z_{i, 1}, \cdots, z_{i, k_i} \} \in \Upsilon$ for any $i \not\in U$.
For any $\Upsilon \in \Pi_U$, define $\mathfrak{Q}_{U, \Upsilon}$ to be a nested contour integral expression, obtained as following: starting from the integrand of \eqref{eq:dcfu}, we first remove every factor that contains $z_{i, i'}$ for $i \in \{1, \cdots, l\} \backslash U$, and then integrate $z_{i,i'}$ for all $i \in U$, $1 \leq i' \leq k_i$.
Now we evaluate both the expressions $\mathfrak{Q}_{\Upsilon}$ and $\mathfrak{Q}_{U, \Upsilon}$, from the inner contours to outer ones.
We apply Lemma \ref{lemma:new} to $\mathfrak{Q}_{\Upsilon}$, then for each $i \in \{1, \cdots, l\} \backslash U$, the $\mathfrak{Y}_i(\Upsilon)$ in \eqref{eq:new:1} equals $(-1)^{k_i - 1}(2\pi\im)^{k_i}t^{-N_i k_i}$.
Then we conclude that
\begin{equation}   \label{eq:lemma:tipip}
\mathfrak{Q}_{\Upsilon} = \prod_{i \in \{1, \cdots, l\} \backslash U}\left( (-1)^{ k_i - 1} (2\pi\im)^{ k_i} t^{-  N_i k_i} \right)\mathfrak{Q}_{U, \Upsilon} .
\end{equation}
By Proposition \ref{prop:act:mul}, we have that
\begin{multline}   \label{eq:lemma:tipi2}
\prod_{i\in U} \mathcal{D}_{-k_{i}}^{N_i}
\prod_{i = 1}^{N_l} \frac{\prod_{k=1}^{\infty}(1 - a_i t^{\alpha + M} q^{k-1}) }{\prod_{k=1}^{\infty}(1 - a_i t^{\alpha} q^{k-1}) }
\\
=
\frac{(-1)^{ \sum_{i \in U} k_i - |U|}}{(2 \pi \im)^{ \sum_{i \in U} k_i }}
\prod_{i = 1}^{N_l} \frac{\prod_{k=1}^{\infty}(1 - a_i t^{\alpha + M} q^{k-1}) }{\prod_{k=1}^{\infty}(1 - a_i t^{\alpha} q^{k-1}) }
\sum_{\Upsilon \in \Pi_U}
\mathfrak{Q}_{U, \Upsilon}
.
\end{multline}
Plugging \eqref{eq:lemma:tipip} into \eqref{eq:lemma:tipi2}, we obtain 
\begin{multline}   \label{eq:lemma:tipi1}
\prod_{i \in U} \mathcal{D}_{-k_{i}}^{N_i} \prod_{i \in \{1, \cdots, l\}\backslash U } t^{-N_i k_i}
\prod_{i = 1}^{N_l} \frac{\prod_{k=1}^{\infty}(1 - a_i t^{\alpha + M} q^{k-1}) }{\prod_{k=1}^{\infty}(1 - a_i t^{\alpha} q^{k-1}) }
\\
=
\frac{(-1)^{k_1 + \cdots + k_l + l }}{(2 \pi \im)^{k_1 + \cdots + k_l}}
\prod_{i = 1}^{N_l} \frac{\prod_{k=1}^{\infty}(1 - a_i t^{\alpha + M} q^{k-1}) }{\prod_{k=1}^{\infty}(1 - a_i t^{\alpha} q^{k-1}) }
\sum_{\Upsilon \in \Pi_U}
\mathfrak{Q}_{\Upsilon}
.
\end{multline}
We obtain \eqref{eq:lemma:tipi} by multiplying $(-1)^{|U|}$ to both sides of \eqref{eq:lemma:tipi1} and summing over all $U \subset \{1, \cdots, l\}$.
\end{proof}

Next, we send $\epsilon \rightarrow 0_+$ in the result of Lemma \ref{lemma:tipi}. 
\begin{lemma}   \label{lemma:lmt}
For $\Upsilon \in \tilde{\Pi}$, if there is $1 \leq w \leq l$ such that $\left\{ z_{w,1}, \cdots, z_{w, k_w} \right\} \bigcap \Upsilon = \left\{ z_{w,s_w}, \cdots, z_{w, r_w} \right\}$, for some $1 < s_w \leq r_w < k_w$,
then
\begin{equation}  \label{eq:lemma:lmt1}
\lim_{\epsilon \rightarrow 0_+}
\left. \epsilon^{-l}
\mathfrak{Q}_{\Upsilon}
\right\vert_{a_i = t^{i-1}, t = \exp(-\theta \epsilon), q = \exp(-\epsilon)}
= 0.
\end{equation}
For $\Upsilon', \Upsilon'' \in \tilde{\Pi}$, if there is $1 \leq w \leq l$ such that $\left\{ z_{w,1}, \cdots, z_{w, k_w} \right\} \bigcap \Upsilon' = \left\{ z_{w, 1}, \cdots, z_{w, r_w} \right\}$,
$\left\{ z_{w,1}, \cdots, z_{w, k_w} \right\} \bigcap \Upsilon'' = \left\{ z_{w, s_w}, \cdots, z_{w, k_w} \right\}$,
with $r_w = k_w - s_w + 1$,
and for any $i \neq w$ we have
$\left\{ z_{i,1}, \cdots, z_{i, k_{i}} \right\} \bigcap \Upsilon' = \left\{ z_{i,1}, \cdots, z_{i, {k_{i}}} \right\} \bigcap \Upsilon''$,
then
\begin{equation}    \label{eq:lemma:lmt2}
\lim_{\epsilon \rightarrow 0_+}
\left.
\epsilon^{-l}
(\mathfrak{Q}_{\Upsilon'} + \mathfrak{Q}_{\Upsilon''})
\right\vert_{a_i = t^{i-1}, t = \exp(-\theta \epsilon), q = \exp(-\epsilon)}
= 0.
\end{equation}
\end{lemma}

\begin{proof}
For any $\Upsilon \in \tilde{\Pi}$, 
in \eqref{eq:new:0}
do the change of variables $a_i = t^{i-1}, t = \exp(-\theta \epsilon), q = \exp(-\epsilon), z_{i, i'} = \exp(\epsilon u_{i, i'})$, and send $\epsilon \rightarrow 0_+$.
Since the right hand side of \eqref{eq:new:0} involves only $z_{i, i'} \not\in \Upsilon$, those $u_{i, i'}$ are well defined.
We have $d z_{i, i'} = \epsilon \exp(\epsilon u_{i, i'}) du_{i, i'}$.
Therefore, each variable in the integral \eqref{eq:new:0} produces an $\epsilon$ factor.
The term $\mathfrak{Y}_i(\Upsilon)$ in \eqref{eq:new:0} is of order $\epsilon^{- (s_i - 2 + \mathds{1}_{s_i = 1}) - (k_i - r_i - 1 + \mathds{1}_{r_i = k_i})} = \epsilon^{- |\{ z_{i,1}, \cdots, z_{i,k_i} \} \backslash \Upsilon| + 2 - \mathds{1}_{s_i = 1} - \mathds{1}_{r_i = k_i} }$,
so $\left. \epsilon^{-l}
\mathfrak{Q}_{\Upsilon}
\right\vert_{a_i = t^{i-1}, t = \exp(-\theta \epsilon), q = \exp(-\epsilon)}$
is of order $\epsilon^{l - \sum_{i=1}^l (\mathds{1}_{s_i = 1} + \mathds{1}_{r_i = k_i})}$.
Since $\Upsilon \in \tilde{\Pi}$, $\mathds{1}_{s_i = 1} + \mathds{1}_{r_i = k_i} \leq 1$; and if $\mathds{1}_{s_i = 1} + \mathds{1}_{r_i = k_i} = 0$ for any $i$, we have \eqref{eq:lemma:lmt1}.

Now we consider a pair $\Upsilon'$ and $\Upsilon''$ as described in the statement of this Lemma.
Then for any $i \neq w$, $\mathfrak{Y}_{i}(\Upsilon') = \mathfrak{Y}_{i}(\Upsilon'')$.
For $w$, we can identify $z_{w,i'}$ in $\mathfrak{Y}_w(\Upsilon'')$ with $z_{w,i' + r_w}$ in $\mathfrak{Y}_w(\Upsilon')$; then we conclude that $\mathfrak{Y}_w(\Upsilon') + \mathfrak{Y}_w(\Upsilon'')$ equals (under the notations in $\mathfrak{Y}_w(\Upsilon')$)
\begin{equation}
(2\pi\im)^{r_w}
(-1)^{r_w} \frac{\left( \frac{z_{w,k_w} t^{k_w - r_w - 1}}{z_{w,r_w + 1} q^{k_w-r_w-1}} -1 \right) t^{-N_wr_w}}
{
\prod_{ r_w + 1 \leq i' < k_w }\left(1 - \frac{tz_{w,i'+1}}{qz_{w,i'}}\right)
}
\prod_{i' < j': z_{w,i'}, z_{w,j'} \not\in \Upsilon }
\frac{ \left(1 - \frac{z_{w,i'}}{z_{w,j'}}\right) \left(1 - \frac{qz_{w,i'}}{tz_{w,j'}}\right) }{ \left(1 - \frac{z_{w,i'}}{tz_{w,j'}}\right)\left(1 - \frac{qz_{w,i'}}{z_{w,j'}}\right) }.
\end{equation}
By setting $a_i = t^{i-1}, t = \exp(-\theta \epsilon), q = \exp(-\epsilon), z_{i, i'} = \exp(\epsilon u_{i, i'})$, and sending $\epsilon \rightarrow 0_+$, this expression is in the order of $\epsilon^{- k_w + r_w + 2} = \epsilon^{ - |\{ z_{w,1}, \cdots, z_{w,k_w} \} \backslash \Upsilon| + 2}$.
Thus $\epsilon^{-l} \left( \mathfrak{Q}_{\Upsilon'} + \mathfrak{Q}_{\Upsilon''} \right)$ is in the order of $\epsilon^{l - \sum_{i \neq w}^l (\mathds{1}_{s_i = 1} + \mathds{1}_{r_i = k_i})}$, which decays to zero when $\epsilon \rightarrow 0_+$.
Then we conclude \eqref{eq:lemma:lmt2}.
\end{proof}

\begin{proof}[Proof of Theorem \ref{thm:var:dis}]
From Lemma \ref{lemma:tipi} and Lemma \ref{lemma:lmt} we see that
\begin{multline}
\lim_{\epsilon \rightarrow 0_+} \left. \frac{\epsilon^{-l} \prod_{i=1}^{l} (t^{-N_{i}k_{i}} - \mathcal{D}_{-k_{i}}^{N_i})
\prod_{i = 1}^{N_l} \frac{\prod_{k=1}^{\infty}(1 - a_i t^{\alpha + M} q^{k-1}) }{\prod_{k=1}^{\infty}(1 - a_i t^{\alpha} q^{k-1}) } }
{
\prod_{i = 1}^{N_l} \frac{\prod_{k=1}^{\infty}(1 - a_i t^{\alpha + M} q^{k-1}) }{\prod_{k=1}^{\infty}(1 - a_i t^{\alpha} q^{k-1}) }
}
\right\vert
\substack{
\vphantom{\int}
\\
\scriptstyle a_i = t^{i-1}, \\ \scriptstyle t = \exp(-\theta \epsilon), \\ \scriptstyle q = \exp(-\epsilon)}
\\
=
\lim_{\epsilon \rightarrow 0_+}
\frac{\epsilon^{-l} (-1)^{k_1 + \cdots + k_l}}{(2 \pi \im)^{k_1 + \cdots + k_l}}
\mathfrak{Q}_{\emptyset} .
\end{multline}
Notice that in $\mathfrak{Q}_{\emptyset}$, each contour encloses all of $qa_1, \cdots, qa_{N_l}$ but not $0$;
then we can set each $z_{i, i'} = \exp(\epsilon u_{i, i'})$, with $u_{i, i'}$ independent of $\epsilon$, satisfying the stated requirements.
Evaluating the limit gives 
\begin{multline}  \label{eq:prop:ev:rhs}
\lim_{\epsilon \rightarrow 0_+}
\left. \frac{\epsilon^{-l} \prod_{i=1}^{l} (t^{-N_{i}k_{i}} - \mathcal{D}_{-k_{i}}^{N_i})
\prod_{i = 1}^{N_l} \frac{\prod_{k=1}^{\infty}(1 - a_i t^{\alpha + M} q^{k-1}) }{\prod_{k=1}^{\infty}(1 - a_i t^{\alpha} q^{k-1}) } }
{
\prod_{i = 1}^{N_l} \frac{\prod_{k=1}^{\infty}(1 - a_i t^{\alpha + M} q^{k-1}) }{\prod_{k=1}^{\infty}(1 - a_i t^{\alpha} q^{k-1}) }
}
\right\vert
\substack{
\vphantom{\int}
\\
\scriptstyle a_i = t^{i-1}, \\ \scriptstyle t = \exp(-\theta \epsilon), \\ \scriptstyle q = \exp(-\epsilon)}
=
\frac{(-1)^l \prod_{i=1}^l k_i}{(2\pi\im)^{k_1 + \cdots + k_l}} \oint \cdots \oint
\\
\times
\prod_{i=1}^{l} \mathfrak{I}(u_{i,1}, \cdots, u_{i,k_i}; \alpha, M, \theta, N_i)
\prod_{i < j} \mathfrak{L}(u_{i,1}, \cdots, u_{i,k_i}; u_{j,1}, \cdots, u_{j,k_j}; \theta)
\prod_{i=1}^l \prod_{i'=1}^{k_i} du_{i,i'}
,
\end{multline}
where for each $i = 1, \cdots, l$,
the contours of $u_{i,1}, \cdots, u_{i, k_i}$ enclose $-\theta (N_i - 1)$ but not $\theta(\alpha + M)$,
and $|u_{i,1}|\ll \cdots \ll |u_{i,k_i}|$.
For $1\leq i \leq i+1 \leq l$,
we also require that $|u_{i, k_{i}}| \ll |u_{i+1, 1}|$.

By Proposition \ref{prop:jotdis}, the identity given by Proposition \ref{prop:ev:lhs} can be interpreted as
\begin{multline} \label{eq:jhm:pf1}
\left.
\frac{\prod_{i=1}^{l} \frac{1}{\theta k_i} (t^{-N_{i}k_{i}} - \mathcal{D}_{-k_{i}}^{N_i})
\prod_{i = 1}^{N_l} \frac{\prod_{k=1}^{\infty}(1 - a_i t^{\alpha + M} q^{k-1}) }{\prod_{k=1}^{\infty}(1 - a_i t^{\alpha} q^{k-1}) }
}{
\prod_{i = 1}^{N_l} \frac{\prod_{k=1}^{\infty}(1 - a_i t^{\alpha + M} q^{k-1}) }{\prod_{k=1}^{\infty}(1 - a_i t^{\alpha} q^{k-1}) }
}
\right\vert_{
\scriptstyle a_i = t^{i-1}}
\\
=
\mathbb{E}\left(\prod_{i=1}^{l}
\left( \frac{1}{\theta k_i}   (t^{ - k_i}-1) \sum_{j=1}^{N_i} (q^{\lambda^{N_i}_j}t^{-j+1})^{k_i} \right) \right)
\end{multline}
where the joint distribution of $\lambda^{N_1}, \cdots, \lambda^{N_l}$ is distributed like a Macdonald process, with parameters $M$, $\{a_1, \cdots, a_{N_l}, 0, \cdots \}$, and $\{b_i\}_{i=1}^M$.

By Theorem \ref{thm:lim} we have that
\begin{multline}  \label{eq:jhm:pf2}
\lim_{\epsilon \rightarrow 0_+}
\left.
\mathbb{E}\left(\prod_{i=1}^{l}
\left( \frac{1}{\theta k_i}  \epsilon^{-1} (t^{- k_i}-1) \sum_{j=1}^{N_i} (q^{\lambda^{N_i}_j}t^{-j+1})^{k_i} \right) \right)
\right\vert
\substack{
\vphantom{\int}
\\
\vphantom{\int}
\\
\scriptstyle t = \exp(-\theta \epsilon), \\ \scriptstyle q = \exp(-\epsilon)}
\\
=
\mathbb{E}\left( \mathfrak{P}_{k_1}(x^{N_1}) \cdots \mathfrak{P}_{k_l}(x^{N_l}) \right)  .
\end{multline}
Putting \eqref{eq:prop:ev:rhs}, \eqref{eq:jhm:pf1}, and \eqref{eq:jhm:pf2} together finishes the proof.
\end{proof}







\section{Law of Large Numbers}  \label{sec:tm}

In this section we present the proofs of Theorems \ref{thm:mom}, \ref{thm:dia}, \ref{thm:mea}, and \ref{thm:jpr}.
We first prove Theorem \ref{thm:mom} in Section \ref{ssec:mom}, as a direct application of Theorem \ref{thm:var:dis}.
In Section \ref{ssec:dia}, we prove Theorem \ref{thm:dia}, by first proving the convergence of the interlacing diagram in moments (Proposition \ref{prop:dia:cd}), and then showing that this is equivalent to convergence in uniform topology.
In Section \ref{ssec:mea} we deduce Theorem \ref{thm:mea} from Theorem \ref{thm:dia} via integration by parts.
We pass $\theta \rightarrow \infty$ in Section \ref{ssec:jpr} and obtain Theorem \ref{thm:jpr} using results from Section \ref{ssec:dia}.

\subsection{First moment of adjacent rows: proof of Theorems \ref{thm:mom}}   \label{ssec:mom}

\begin{proof}[Proof of Theorem \ref{thm:mom}]
Taking $l = 1$ in Theorem \ref{thm:var:dis} we get
\begin{multline} \label{eq:dismomdif}
\mathbb{E}\left(\mathfrak{P}_k(x^N) - \mathfrak{P}_k(x^{N-1})\right)
=
\frac{(-\theta)^{-1}}{(2\pi\im)^k}
\oint\cdots\oint \frac{1}{(u_2-u_1 + 1 - \theta) \cdots (u_k-u_{k-1} + 1 - \theta)} \\
\times \prod_{i<j}\frac{(u_j-u_i)(u_j-u_i + 1 - \theta)}{(u_j-u_i + 1)(u_j-u_i - \theta)}
\left( \prod_{i=1}^k  \frac{u_i - \theta}{u_i + (N-1)\theta} - \prod_{i=1}^k \frac{u_i - \theta}{u_i + (N-2)\theta}  \right) \\
\times \prod_{i=1}^k \frac{\theta \alpha - u_i}{\theta (\alpha+M) - u_i }
du_i .
\end{multline}

Setting $u_i = L\theta v_i$, we send $L\rightarrow \infty$ under \eqref{eq:lsch}.
Note that
\begin{multline}
\lim_{L\rightarrow \infty} L\left( \prod_{i=1}^k  \frac{u_i - \theta}{u_i + (N-1)\theta} - \prod_{i=1}^k \frac{u_i - \theta}{u_i + (N-2)\theta} \right) \\
= \lim_{L\rightarrow \infty} L \prod_{i=1}^k  \frac{u_i - \theta}{u_i + (N-1)\theta} \left(1 - \prod_{i=1}^k \frac{u_i + (N-1)\theta}{u_i + (N-2)\theta} \right)
= - \left( \prod_{i=1}^k \frac{v_i}{v_i + \hat{N}} \right) \left(\sum_{i=1}^k \frac{1}{v_i + \hat{N}} \right).
\end{multline}
Therefore
\begin{multline}
\lim_{L\rightarrow \infty} \mathbb{E}\left(\mathfrak{P}_k(x^N) - \mathfrak{P}_k(x^{N-1})\right)
= \frac{1}{(2\pi\im)^k}\oint\cdots\oint \frac{1}{(v_2-v_1) \cdots (v_k-v_{k-1} )} \\
\times \left( \prod_{i=1}^k \frac{v_i}{v_i + \hat{N}} \cdot \frac{ \hat{\alpha} - v_i}{ \hat{\alpha} + \hat{M} - v_i }
dv_i  \right) \left(\sum_{i=1}^k \frac{1}{v_i + \hat{N}} \right) ,
\end{multline}
where the contours enclose $-\hat{N}$ but not $\hat{\alpha} + \hat{M}$,
and $|v_1|\ll \cdots \ll |v_k|$.
By Corollary \ref{cor:dr}, this is simplified to (\ref{eq:mom:diff}).

The decay of variance will be proved as a special case of Lemma \ref{lemma:cov:s}.
\end{proof}

\subsection{Convergence of diagrams: proof of Theorem \ref{thm:dia}} \label{ssec:dia}


\begin{prop}  \label{prop:dia:cd}
Let $\varphi$ be defined as in Theorem \ref{thm:dia}.
For any nonnegative integer $k$,
under the limit scheme (\ref{eq:lsch}) we have
\begin{equation} \label{eq:cd:prop}
\lim_{L\rightarrow \infty} \int_{0}^1 w^{\tilde{x}^N, \tilde{x}^{N-1}}(u) u^k du = \int_{0}^1 \varphi(u) u^k du ,
\end{equation}
in probability.
\end{prop}

The proof of Proposition \ref{prop:dia:cd} relies on the following identity.
\begin{lemma} \label{lemma:dia:id}
For $\varphi$ as in Theorem \ref{thm:dia}, we have
\begin{equation} \label{eq:cd:a1}
\frac{1}{2\pi\im}\oint_{\Gamma} \left(\frac{v}{v + \hat{N}}\cdot \frac{v - \hat{\alpha}}{v - \hat{\alpha} - \hat{M}}\right)^k \frac{1}{v+\hat{N}} dv = \frac{1}{2} \int_{0}^1 \varphi''(u) u^k du ,
\end{equation}
where $\Gamma$ is a positive oriented contour enclosing $-\hat{N}$ but not $\hat{\alpha} + \hat{M}$.
\end{lemma}

\begin{proof}
Fix the contour $\Gamma$,
and let $\mathcal{Q}$ be the constant
\begin{equation}
\mathcal{Q} :=  \inf_{v \in \Gamma} \left| \frac{v + \hat{N}}{v}\cdot \frac{v - \hat{\alpha} - \hat{M}}{v - \hat{\alpha}} \right| .
\end{equation}
For any $z \in \mathbb{C}$, $|z| < \mathcal{Q}$,
summing the geometric series we get
\begin{multline}  \label{eq:lemma:dia:id:pf}
\sum_{k=0}^{\infty} z^k \frac{1}{2\pi\im}\oint_{\Gamma} \left(\frac{v}{v + \hat{N}}\cdot \frac{v - \hat{\alpha}}{v - \hat{\alpha} - \hat{M}}\right)^k \frac{1}{v+\hat{N}} dv \\
= \frac{1}{2\pi\im}\oint_{\Gamma} \sum_{k=0}^{\infty} \left(z \cdot \frac{v}{v + \hat{N}}\cdot \frac{v - \hat{\alpha}}{v - \hat{\alpha} - \hat{M}}\right)^k \frac{1}{v+\hat{N}} dv
= \frac{1}{2\pi\im}\oint_{\Gamma}
\frac{v - \hat{\alpha} - \hat{M}}{(1 - z)(v - \mathcal{R}_1)(v - \mathcal{R}_2)}
dv ,
\end{multline}
where
\begin{equation}
\begin{split}
\mathcal{R}_1 & = \frac{(1 - z)\hat{\alpha} + \hat{M} - \hat{N} - \sqrt{((1 - z)\hat{\alpha} + \hat{M} + \hat{N})^2 - 4z\hat{M}\hat{N} } }{2(1 - z)} ,
 \\
\mathcal{R}_2 & = \frac{(1 - z)\hat{\alpha} + \hat{M} - \hat{N} + \sqrt{((1 - z)\hat{\alpha} + \hat{M} + \hat{N})^2 - 4z\hat{M}\hat{N} } }{2(1 - z)} .
\end{split}
\end{equation}

The definition of the contour $\Gamma$ implies that there exists $0 < \mathcal{Q}' < \mathcal{Q}$,
such that $\Gamma$ encloses $\mathcal{R}_1$, but not $\mathcal{R}_2$,
for all $z\in \mathbb{C}$, $|z| < \mathcal{Q}
'$.
Then (\ref{eq:lemma:dia:id:pf}) is evaluated as a residue at $\mathcal{R}_1$; more precisely,
\begin{multline} \label{eq:cd:lhs}
\sum_{k=0}^{\infty} z^k \frac{1}{2\pi\im}\oint_{\Gamma} \left(\frac{v}{v + \hat{N}}\cdot \frac{v - \hat{\alpha}}{v - \hat{\alpha} - \hat{M}}\right)^k \frac{1}{v+\hat{N}} dv \\
= \frac{1}{2\pi \im (1-z)} \oint_{\Gamma} \frac{v - \hat{\alpha} - \hat{M}}{(v - \mathcal{R}_1)(v - \mathcal{R}_2)} dv
= \frac{1}{1-z} \cdot \frac{\mathcal{R}_1 - \hat{\alpha} - \hat{M}}{\mathcal{R}_1 - \mathcal{R}_2} .
\end{multline}

On the other hand, for $|z|< 1$ we have that
\begin{multline}  \label{eq:cd:lhs:pf1}
\sum_{k=0}^{\infty} \frac{1}{2} \int_{0}^1 \varphi''(u) u^k z^k du \\
=  \frac{C(\hat{M}, \hat{N})}{1-z} + \sum_{k=0}^{\infty} \frac{1}{2\pi} \int_{\gamma_1}^{\gamma_2}  \frac{\hat{M} - \hat{N} + (\hat{N}+\hat{M} + \hat{\alpha})(1-u)}{(\hat{N}+\hat{M} + \hat{\alpha})(1-u)}\cdot \frac{z^k u^k}{\sqrt{(\gamma_2 - u)(u - \gamma_1)}} du \\
= \frac{C(\hat{M}, \hat{N})}{1-z} + \frac{1}{2\pi} \int_{\gamma_1}^{\gamma_2}  \frac{\hat{M} - \hat{N} + (\hat{N}+\hat{M} + \hat{\alpha})(1-u)}{(\hat{N}+\hat{M} + \hat{\alpha})(1-u)(1-zu)}\cdot \frac{1}{\sqrt{(\gamma_2 - u)(u - \gamma_1)}} du ,
\end{multline}
where $\gamma_1$ and $\gamma_2$ are defined as in Theorem \ref{thm:dia}.
The last integral can be evaluated in the following way:
define $\eta : [\gamma_2, \infty) \rightarrow \mathbb{C}, $ as
\begin{equation}
\eta(w) = \frac{1}{2\pi \im} \cdot \frac{\hat{M} - \hat{N} + (\hat{N}+\hat{M} + \hat{\alpha})(1-w)}{(\hat{N}+\hat{M} + \hat{\alpha})(1-w)(1-zw)}\cdot \frac{1}{\sqrt{(w - \gamma_1)(w - \gamma_2)}},
\end{equation}
and then extend the definition of $\eta$ to $\mathbb{C} \backslash \left( [\gamma_1, \gamma_2] \bigcup \{1\} \bigcup \{z^{-1}\} \right)$, by taking the analytic continuation.
The integral in the last line of (\ref{eq:cd:lhs:pf1}) is equal to the contour integral of $\eta$ around $[\gamma_1, \gamma_2]$.
Then it suffices to subtract the residues of $\eta(w)$ at $1$ and $z^{-1}$, which are
\begin{equation}
\frac{1}{1 - z} \left( C(\hat{M}, \hat{N}) - \frac{1}{2} \right) ,
\quad
\frac{ (2\hat{M} + \hat{\alpha})z - (\hat{M} + \hat{N} + \hat{\alpha}) }{2(1-z)\sqrt{((1 - z)\hat{\alpha} + \hat{M} + \hat{N})^2 - 4z\hat{M}\hat{N} }}
\end{equation}
respectively.
We thus conclude that (\ref{eq:cd:lhs:pf1}) equals
\begin{equation}
\frac{1}{2(1-z)} + \frac{ \hat{M} + \hat{N} + \hat{\alpha} - 2\hat{M}z - z\hat{\alpha} }{2(1-z)\sqrt{((1 - z)\hat{\alpha} + \hat{M} + \hat{N})^2 - 4z\hat{M}\hat{N} }} ,
\end{equation}
which coincides with (\ref{eq:cd:lhs}).

In other words, for any $z \in \mathbb{C}$, $|z| < \min\{ \mathcal{Q}', 1\}$,
we have
\begin{equation} \label{eq:cd:a2}
\sum_{k=0}^{\infty} z^k \frac{1}{2\pi\im}\oint \left(\frac{v}{v + \hat{N}}\cdot \frac{v - \hat{\alpha}}{v - \hat{\alpha} - \hat{M}}\right)^k \frac{1}{v+\hat{N}} dv
=
\sum_{k=0}^{\infty} z^k \frac{1}{2} \int_{0}^1 \varphi''(u) u^k du .
\end{equation}
The uniqueness of the Taylor series expansion then implies (\ref{eq:cd:a1}).
\end{proof}

\begin{proof} [Proof of Proposition \ref{prop:dia:cd}]
First, by Remark \ref{rem:secder}, Theorem \ref{thm:mom}, and Lemma \ref{lemma:dia:id},
for any nonnegative integer $k$,
under the limit scheme (\ref{eq:lsch}) we have
\begin{equation} \label{eq:cd:a0}
\lim_{L\rightarrow \infty} \int_{0}^1 \frac{d^2}{du^2} w^{\tilde{x}^N, \tilde{x}^{N-1}}(u) u^k du =
\lim_{L\rightarrow \infty} 2 \left( \mathfrak{P}_k(x^N) - \mathfrak{P}_k(x^{N-1}) \right) =
\int_{0}^1 \varphi''(u) u^k du ,
\end{equation}
in probability.
Integrating by parts twice,
this implies \eqref{eq:cd:prop}.
\end{proof}

The following result connects Proposition \ref{prop:dia:cd} and Theorem \ref{thm:dia}.

\begin{lemma} \protect{\cite[Lemma 5.7]{Ivanov2002}} \label{lemma:eqcov}
For any fixed interval $[a, b] \subset \mathbb{R}$,
let $\Sigma$ be the set of all functions $\rho:\mathbb{R} \rightarrow \mathbb{R}$,
that are supported in $[a, b]$
and satisfy $|\rho(u_1) - \rho(u_2)| \leq |u_1 - u_2|$, $\forall u_1, u_2 \in [a, b]$.
Then the weak topology defined by the functionals
\begin{equation}
\rho \rightarrow \int \rho(u) u^k du, \quad k = 0, 1, \cdots
\end{equation}
coincides with the uniform topology given by the supremum norm $\| \rho \| = \sup |\rho(u)|$ .
\end{lemma}

\begin{proof}[Proof of Theorem \ref{thm:dia}]
By Lemma \ref{lemma:eqcov}, the convergence of $w^{\tilde{x}^N, \tilde{x}^{N-1}}$ (in probability) under
the uniform topology is equivalent to the convergence (in probability) of each moment,
and the later is precisely Proposition \ref{prop:dia:cd}.
\end{proof}

\subsection{Convergence of discrete signed measures: proof of Theorem \ref{thm:mea}}  \label{ssec:mea}


\begin{proof}[Proof of Theorem \ref{thm:mea}]
It suffices to consider a function $f$ whose derivative is nondecreasing, since each function of finite variation can be written as the difference of two nondecreasing functions.
Now since $f'$ is nondecreasing,
we can define $g: [0, 1] \rightarrow \mathbb{R}$ such that $g(u)$ is the right limit of $f'$ at $u$, for any $u \in [0, 1)$;
and $g(1)$ the left limit of $f'$ at $1$.
Then $g$ is also nondecreasing, and bounded, and right continuous.
Also $g = f'$ almost everywhere.

Let $\eta$ be a measure on $[0, 1]$, such that $\eta([0, u]) = g(u)$, for any $u \in [0, 1]$.

Integrating by parts (and with Remark \ref{rem:secder}) we have
\begin{multline}
\int_0^1 f d\mu^{\tilde{x}^N, \tilde{x}^{N-1}}
= \frac{-f(0)\frac{d}{du} w^{\tilde{x}^N, \tilde{x}^{N-1}}(0) + f(1)\frac{d}{du} w^{\tilde{x}^N, \tilde{x}^{N-1}}(1)}{2} - \frac{1}{2} \int_{0}^1  f'(u) \frac{d}{du} w^{\tilde{x}^N, \tilde{x}^{N-1}}(u) du \\
= \frac{f(0) + f(1)}{2} - \frac{1}{2} \int_{0}^1  g(u) \frac{d}{du} w^{\tilde{x}^N, \tilde{x}^{N-1}}(u) du
\\
= \frac{f(0) + f(1)}{2} - \frac{1}{2}
\left( g(1) {w^{\tilde{x}^N, \tilde{x}^{N-1}}}(1) - g(0) {w^{\tilde{x}^N, \tilde{x}^{N-1}}}(0) \right)
+ \frac{1}{2} \int_{0}^1 {w^{\tilde{x}^N, \tilde{x}^{N-1}}}(u) d\eta .
\end{multline}

By Theorem \ref{thm:dia},
we have
\begin{equation}
\lim_{L\rightarrow \infty} {w^{\tilde{x}^N, \tilde{x}^{N-1}}}(0) = \varphi(0) , \quad
\lim_{L\rightarrow \infty} {w^{\tilde{x}^N, \tilde{x}^{N-1}}}(1) = \varphi(1) ,
\end{equation}
in probability.
Since
\begin{multline}
\left| \frac{1}{2} \int_{0}^1 {w^{\tilde{x}^N, \tilde{x}^{N-1}}}(u) d\eta
- \frac{1}{2} \int_{0}^1 \varphi(u) d\eta \right| \leq
\frac{1}{2} \int_{0}^1 \left| \left({w^{\tilde{x}^N, \tilde{x}^{N-1}}}(u) - \varphi(u)\right) \right| d\eta \\
\leq \frac{1}{2} \sup_{u\in \mathbb{R}} \left| {w^{\tilde{x}^N, \tilde{x}^{N-1}}}(u) - \varphi(u) \right|
\eta([0, 1]),
\end{multline}
we have
\begin{equation}
\lim_{L\rightarrow \infty} \frac{1}{2} \int_{0}^1 {w^{\tilde{x}^N, \tilde{x}^{N-1}}}(u) d\eta = \frac{1}{2} \int_{0}^1 \varphi(u) d\eta ,
\end{equation}
in probability.
Again integrating by parts, we conclude
\begin{multline}
\lim_{L\rightarrow \infty } \int_0^1 f d\mu^{\tilde{x}^N, \tilde{x}^{N-1}}
= \frac{1}{2} \left( f(0) + f(1) -
g(1) \varphi(1) + f'(0) \varphi(0)
+ \int_{0}^1 \varphi(u) d\eta \right) \\
= \frac{1}{2} \int_{0}^1 f(u) \varphi''(u) du  ,
\end{multline}
which is precisely (\ref{eq:thm:mea}).
\end{proof}

\subsection{Asymptotics of roots of Jacobi polynomials: proof of Theorem \ref{thm:jpr}}  \label{ssec:jpr}

We prove Theorem \ref{thm:jpr} by utilizing a limit transition between the distribution $\mathbb{P}^{\alpha, M, \theta}$ on $\chi^M$
and the roots of $\mathcal{F}_{\min(M, N)}^{\alpha - 1, |M-N|}$.

\begin{theorem}\protect{\cite[Theorem 5.1]{borodin2015general}} \label{thm:trt}
Let $(x^1, x^2, \cdots) \in \chi^M$ be distributed as $\mathbb{P}^{\alpha, M, \theta}$,
and let $j_{M, N, \alpha, i}$ be the $i$th root (in increasing order) of
$\mathcal{F}_{\min(M, N)}^{\alpha - 1, |M-N|}$,
for $1 \leq i \leq \min(M, N)$.
Then we have
\begin{equation}
\lim_{\theta \rightarrow \infty} x^{N}_i = j_{M, N, \alpha, i} ,
\end{equation}
in probability.
\end{theorem}


\begin{proof}[Proof of Theorem \ref{thm:jpr}]
The interlacing relationship for the roots immediately follows Theorem \ref{thm:trt} and the interlacing relationship for the sequences $x_i^N$ and $x_i^{N-1}$.

In Theorem \ref{thm:var:dis}, take $l = 1$, set $u_{1, i} = \theta w_i$, and send $\theta \rightarrow \infty$.
Using Theorem \ref{thm:trt}, we compute for $k = 1, 2, \cdots$:
\begin{multline}
\sum_{i=1}^{N} j_{M, N, \alpha, i}^k
= \frac{-1}{(2\pi\im)^k} \oint\cdots \oint \frac{1}{(w_2 - w_1 - 1) \cdots (w_k - w_{k-1} - 1)} \\
\times\prod_{i=1}^k \frac{w_i - 1}{w_i + N - 1} \cdot \frac{\alpha - w_i}{\alpha + M - w_i} dw_i  ,
\end{multline}
where each contour encloses $-N+1$ but not $M + \alpha$,
and $|w_1|\ll \cdots \ll |w_k|$.

Under (\ref{eq:lsch}), setting $w_i = L\upsilon_i$, we have
\begin{multline}   \label{eq:thm:jpr:pf}
\lim_{L\rightarrow \infty}
\left( \sum_{i=1}^{N} j_{M, N, \alpha, i}^k - \sum_{i=1}^{N-1} j_{M, N-1, \alpha, i}^k \right) \\
= \lim_{L\rightarrow \infty} \frac{1}{(2\pi\im)^k} \oint\cdots \oint \frac{1}{(w_2 - w_1 - 1) \cdots (w_k - w_{k-1} - 1)} \\
\times\prod_{i=1}^k \frac{w_i - 1}{w_i + N - 1} \cdot \frac{\alpha - w_i}{\alpha + M - w_i} dw_i
\left(\prod_{i=1}^k \frac{w_i + N - 1}{w_i + N - 2} - 1\right) \\
= \frac{1}{(2\pi\im)^k} \oint\cdots \oint \frac{1}{(\upsilon_2 - \upsilon_1) \cdots (\upsilon_k - \upsilon_{k-1})}
\prod_{i=1}^k \frac{\upsilon_i}{\upsilon_i + \hat{N}} \cdot \frac{\hat{\alpha} - \upsilon_i}{\hat{\alpha} + \hat{M} - \upsilon_i} d\upsilon_i
\left(\sum_{i=1}^k \frac{1}{\upsilon_i + \hat{N}} \right) .
\end{multline}

We apply Corollary \ref{cor:dr} to (\ref{eq:thm:jpr:pf}) to do dimension reduction, and get
\begin{equation}
\lim_{L\rightarrow \infty}
\left( \sum_{i=1}^{N} j_{M, N, \alpha, i}^k - \sum_{i=1}^{N-1} j_{M, N-1, \alpha, i}^k \right)
=
\frac{1}{2\pi\im} \oint \left(\frac{\upsilon}{\upsilon + \hat{N}}\cdot \frac{\upsilon - \hat{\alpha}}{\upsilon - \hat{\alpha} - \hat{M}}\right)^k \frac{1}{\upsilon+\hat{N}} d\upsilon  ,
\end{equation}
where the contour on the right hand side is around $-\hat{N}$ but not $\hat{\alpha} + \hat{M}$.

The right hand side now exactly fits Lemma \ref{lemma:dia:id},
and we conclude that
\begin{equation}
\lim_{L\rightarrow \infty}
\int_0^1 \iota_{M,N,\alpha}''(u)u^k du =
\lim_{L\rightarrow \infty}
2\left( \sum_{i=1}^{N} j_{M, N, \alpha, i}^k - \sum_{i=1}^{N-1} j_{M, N-1, \alpha, i}^k \right)
=
\int_0^1 \varphi''(u) u^k du  .
\end{equation}

Finally, we integrate by parts, which leads to
\begin{equation}
\lim_{L\rightarrow \infty}
\int_0^1 \iota_{M,N,\alpha}(u)u^k du =
\int_0^1 \varphi(u) u^k du  .
\end{equation}
By Lemma \ref{lemma:eqcov}, the above is equivalent to the statement of Theorem \ref{thm:jpr}.
\end{proof}

\section{Central Limit Theorem and Gaussianity of fluctuations}  \label{sec:clt}

The ultimate goal of this section is to present the proofs of Theorem \ref{thm:clt1} and Theorem \ref{thm:clt2}.
To prove weak convergence of the joint distribution of a vector to a Gaussian vector, it suffices to check that all the cumulants converge to the corresponding ones of a Gaussian vector.
We first introduce our notations for cumulants, and recall a basic result about the cumulants of multivariate Gaussian distribution.
\begin{defn}
For any positive integer $h$, let $\Theta_{h}$ be the collection of all unordered partitions of $\{1, \cdots, h\}$:
\begin{equation}
\Theta_{h} = \left\{ \left\{ U_1, \cdots, U_t \right\} \left| t\in \mathbb{Z}_+, \bigcup_{i=1}^t U_i = \{1, \cdots, h\}, U_i \bigcap U_j = \emptyset,
U_i \neq \emptyset, \forall 1\leq i < j \leq t \right. \right\} .
\end{equation}
For a random vector $\mathbf{u} = \{ u_i \}_{i=1}^w$, and any $v_1, \cdots, v_h \in \{u_1, \cdots, u_w \}$, define the \emph{(order $h$) cumulant} $\kappa(v_1, \cdots, v_h)$ as
\begin{equation}   \label{eq:defn:cum}
\kappa(v_1, \cdots, v_h) = \sum_{\substack{t \in \mathbb{Z}_{>0} \\ \left\{ U_1, \cdots, U_t \right\} \in \Theta_h}} (-1)^{t-1} (t-1)! \prod_{r=1}^t \mathbb{E}\left[ \prod_{i \in U_r} v_i \right].
\end{equation}
\end{defn}
The definition implies that all cumulants of order up to $h$ of a vector exist if and only if all joint moments of order up to $h$ of the same vector exist.
The moments and joint cumulants uniquely determine each other.

We further give another (and more commonly used) definition of cumulants.
\begin{defn}
For a random vector $\mathbf{u} = \{ u_i \}_{i=1}^w$, the \emph{characteristic function} $\phi: \mathbb{R}^w \rightarrow \mathbb{C}$ is defined as
\begin{equation}
\phi(\lambda_1, \cdots, \lambda_w) = \mathbb{E}\left( e^{\im(\lambda_1 u_1 + \cdots + \lambda_w u_w)} \right) .
\end{equation}
\end{defn}
\begin{defn}
Let $\mathbf{u} = \{ u_i \}_{i=1}^w$ be a random vector with characteristic function $\phi$,
then the \emph{cumulants} are the coefficients of the Taylor expansion of $\log (\phi(\lambda_1, \cdots, \lambda_w))$, if it exists in a neighborhood of the origin:
\begin{equation}
\log (\phi(\lambda_1, \cdots, \lambda_w))
=
\sum_{h=1}^{\infty} \frac{1}{h!}\sum_{a_1, \cdots, a_h = 1}^w \kappa(u_{a_1}, \cdots, u_{a_h}) \prod_{j=1}^h \im \lambda_{a_j} ,
\end{equation}
and we require that $\kappa(u_{a_1}, \cdots, u_{a_h})$ is symmetric on $a_1, \cdots, a_h$, for fixed $h$.
\end{defn}
This definition imposes more requirements on the random vector, since we need the existence of the Taylor expansion of $\log (\phi(\lambda_1, \cdots, \lambda_w))$.
This can be generalized by using the derivatives of $\log (\phi(\lambda_1, \cdots, \lambda_w))$ at the origin.
For more discussions, see e.g. \cite[Section 3.1, 3.2]{book:766347}, where the second definition is taken, and \eqref{eq:defn:cum} is proved as a proposition.

From the second definition, the following result immediately follows.
\begin{lemma}  \label{lemma:geq}
A random vector is Gaussian if and only if all of its cumulants with order 3 or more are zero.
\end{lemma}

Now let's consider the cumulants of the vectors in Theorem \ref{thm:clt1} and Theorem \ref{thm:clt2}.
For the vectors \eqref{eq:clt1:d:cov} and \eqref{eq:clt1:cov} in Theorem \ref{thm:clt1}, the first order cumulants obviously vanish.
In Section \ref{ssec:cc}, we show that the second order cumulants, which are precisely the covariances, converge to the desired values.
Then it suffices to show that the higher order cumulants converge to zero as $L \rightarrow \infty$.
In Section \ref{ssec:gf}, we first prove a formula about the decay of the cumulants with certain constraints (Proposition \ref{prop:gtp}), and use a linear combination of it to get the decay of any high order cumulants (Proposition \ref{prop:gtp2}).
These results lead to Theorem \ref{thm:clt1} (actually the decay of high order cumulants is faster than needed).
Finally, in Section \ref{ssec:gs}, we integrate Proposition \ref{prop:gtp2} and the computations of covariances over levels, and prove Theorem \ref{thm:clt2}.




\subsection{Computation of covariances}  \label{ssec:cc}

The first step of our proof of Theorems \ref{thm:clt1}, \ref{thm:clt2} is the covariance computation, presented in this section.

Throughout this section, let $k_1, k_2$ and $N_1, N_2$ be positive integers.
In addition to (\ref{eq:lsch}), we also let
\begin{equation}  \label{eq:lsch4}
\lim_{L\rightarrow \infty} \frac{N_1}{L} = \hat{N}_1, \quad \lim_{L\rightarrow \infty} \frac{N_2}{L} = \hat{N}_2 ,
\end{equation}
where $\hat{N}_1$ and $\hat{N}_2$ are positive real numbers.

\begin{lemma}  \label{lemma:cov:as}
Under the limit scheme (\ref{eq:lsch}) and (\ref{eq:lsch4}), let us assume additionally either (I) $N_1 \leq N_2$ for all $L$ large enough, or (II) $N_1 > N_2$ for all $L$ large enough.
Then
\begin{multline}  \label{eq:cov:as}
\lim_{L\rightarrow \infty} L \cdot
\mathbb{E}\left[
\left(
\mathfrak{P}_{k_1}(x^{N_1}) - \mathfrak{P}_{k_1}(x^{N_1-1})
\right)
\mathfrak{P}_{k_2}(x^{N_2})
\right]
- L \cdot \mathbb{E}
\left[\mathfrak{P}_{k_1}(x^{N_1}) - \mathfrak{P}_{k_1}(x^{N_1-1}) \right]
\mathbb{E}\left[\mathfrak{P}_{k_2}(x^{N_2}) \right]
\\
=
-
\frac{\theta^{-1}k_1}{(2\pi\im)^{2}}
\oint \oint
\frac{1}{(v_1 - v_2)^2}
\frac{1}{v_1 + \hat{N}_1}
\prod_{i=1}^2
\left( \frac{v_i}{v_i+\hat{N}_i}\cdot \frac{v_i-\hat{\alpha}}{v_i - \hat{\alpha} - \hat{M}} \right)^{k_i} dv_i,
\end{multline}
where the contours enclose poles at $-\hat{N}_1$ and $-\hat{N}_2$, but not $\hat{\alpha} + \hat{M}$, and are nested: $|v_1| \ll |v_2|$ in case I,
and $|v_1| \gg |v_2|$ in case II.
\end{lemma}

\begin{proof}
Applying Theorem \ref{thm:var:dis} to the following four terms: $\mathbb{E}\left[\mathfrak{P}_{k_1}(x^{N_1})\mathfrak{P}_{k_2}(x^{N_2})\right]$, $\mathbb{E}\left[\mathfrak{P}_{k_1}(x^{N_1-1})\mathfrak{P}_{k_2}(x^{N_2})\right]$, $\mathbb{E}\left[\mathfrak{P}_{k_1}(x^{N_1})\right]\mathbb{E}\left[\mathfrak{P}_{k_2}(x^{N_2})\right]$, and $\mathbb{E}\left[\mathfrak{P}_{k_1}(x^{N_1-1})\right]\mathbb{E}\left[\mathfrak{P}_{k_2}(x^{N_2})\right]$, we get
\begin{multline}  \label{eq:cov:as:pf1}
\mathbb{E}\left[
\left(
\mathfrak{P}_{k_1}(x^{N_1}) - \mathfrak{P}_{k_1}(x^{N_1-1})
\right)
\mathfrak{P}_{k_2}(x^{N_2})
\right]
- \mathbb{E}
\left[\mathfrak{P}_{k_1}(x^{N_1}) - \mathfrak{P}_{k_1}(x^{N_1-1}) \right]
\mathbb{E}\left[\mathfrak{P}_{k_2}(x^{N_2}) \right]
\\
=
\frac{(-\theta)^{-2}}{(2\pi\im)^{k_1 + k_2}}
\oint \cdots \oint \prod_{i=1}^2 \left[ \frac{1}{(u_{i,2}-u_{i,1}+1-\theta)\cdots (u_{i,k_i} - u_{i,k_i-1}+1-\theta)} \right. \\
\times
\left.
\prod_{1\leq i'<j'\leq k_i}\frac{(u_{i,j'}-u_{i,i'}) (u_{i,j'}-u_{i,i'}+1-\theta)}{(u_{i,j'}-u_{i,i'}-\theta) (u_{i,j'}-u_{i,i'}+1)}
\prod_{i'=1}^{k_i}\frac{u_{i,i'}-\theta}{u_{i,i'}+(N_i-1)\theta}\cdot \frac{u_{i,i'}-\theta\alpha}{u_{i,i'}-\theta\alpha - \theta M}du_{i,i'} \right]
 \\
\times \left(1 -  \prod_{i'=1}^{k_1}\frac{u_{1,i'}+(N_1-1)\theta}{u_{1,i'}+(N_1-2)\theta} \right)
\left( \prod_{ \substack{1\leq i'\leq k_1, \\ 1\leq j'\leq k_2}}\frac{(u_{1,i'}-u_{2,j'}) (u_{1,i'}-u_{2,j'}+1-\theta)}{(u_{1,i'}-u_{2,j'}-\theta) (u_{1,i'}-u_{2,j'}+1)} - 1 \right) ,
\end{multline}
where the contours for $u_{i, k_1}, \cdots, u_{i, k_i}$ enclose $-\theta (N_i - 2)$ and $-\theta (N_i - 1)$
but not $\theta(\alpha + M)$, for $i = 1, 2$.
We also require that
$|u_{1,1}|\ll \cdots \ll |u_{1,k_1}|$, $|u_{2,1}|\ll \cdots \ll |u_{2,k_2}|$;
and $|u_{1,k_1}| \ll |u_{2,1}|$ when $N_1 \leq N_2$, $|u_{2,k_2}| \ll |u_{1,1}|$ when $N_1 > N_2$.
The four terms obtained from expanding the two factors in the last line correspond to the four terms to which we apply Theorem \ref{thm:var:dis}.

Set $u_{i,i'} = L\theta v_{i,i'}$ for $i = 1, 2$, and any $1\leq i' \leq k_i$, and send $L \rightarrow \infty$.
Observe that
\begin{equation}  \label{eq:cov:obs2}
\prod_{ \substack{1\leq i'\leq k_1, \\ 1\leq j'\leq k_2}}\frac{(u_{1,i'}-u_{2,j'}) (u_{1,i'}-u_{2,j'}+1-\theta)}{(u_{1,i'}-u_{2,j'}-\theta) (u_{1,i'}-u_{2,j'}+1)} - 1
=
L^{-2} \cdot \left( \sum_{\substack{1\leq i'\leq k_1, \\ 1\leq j'\leq k_2}} \frac{\theta^{-1}}{(v_{1,i'} - v_{2,j'})^2}  \right)
+ O(L^{-4}).
\end{equation}
and 
\begin{equation}  \label{eq:cov:obs1}
1 -  \prod_{i'=1}^{k_1}\frac{u_{1,i'}+(N_1-1)\theta}{u_{1,i'}+(N_1-2)\theta}
=
- L^{-1} \cdot \left( \sum_{i'=1}^{k_1} \frac{1}{ v_{1,i'} + \hat{N}_1} \right) + O(L^{-2}) .
\end{equation}
Therefore, (\ref{eq:cov:as:pf1}) multiplied by $L$ converges to
\begin{multline}
-
\frac{\theta^{-1}}{(2\pi\im)^{k_1 + k_2}}
\oint \cdots \oint
\left( \sum_{\substack{1\leq i'\leq k_1, \\ 1\leq j'\leq k_2}} \frac{1}{(v_{1,i'} - v_{2,j'})^2}  \right)
  \left( \sum_{i'=1}^{k_1} \frac{1}{ v_{1,i'} + \hat{N}_1} \right)
\\
\times
\prod_{i=1}^2
\left( \frac{1}{(v_{i,2}-v_{i,1})\cdots (v_{i,k_i} - v_{i,k_i-1})}
\left( \prod_{i'=1}^{k_i}\frac{v_{i,i'}}{v_{i,i'}+\hat{N}_i}\cdot \frac{v_{i,i'}-\hat{\alpha}}{v_{i,i'}- \hat{\alpha} - \hat{M}}dv_{i,i'} \right)
\right) .
\end{multline}
Applying Corollary \ref{cor:dr} to $v_{i, k_1}, \cdots, v_{i, k_i}$ and $v_{j, k_1}, \cdots, v_{j, k_j}$, respectively, we get (\ref{eq:cov:as}).
\end{proof}

\begin{lemma}  \label{lemma:cov:d}
Assume that $N_1 < N_2$ for $L$ large enough, then
\begin{multline}  \label{eq:lemma:cov:d}
\lim_{L\rightarrow \infty} L^2
\mathbb{E}\left[
\prod_{i=1}^2 \left(
\mathfrak{P}_{k_i}(x^{N_i}) - \mathfrak{P}_{k_i}(x^{N_i-1}) - \mathbb{E}\left(\mathfrak{P}_{k_i}(x^{N_i}) - \mathfrak{P}_{k_i}(x^{N_i-1}) \right) \right)
\right]
\\
=
\frac{\theta^{-1}k_1k_2}{(2\pi\im)^{2}}
\oint \oint
\frac{1}{(v_1 - v_2)^2}
\prod_{i=1}^2
\frac{dv_i}{v_i + \hat{N}_i}
\left( \frac{v_i}{v_i+\hat{N}_i}\cdot \frac{v_i-\hat{\alpha}}{v_i - \hat{\alpha} - \hat{M}} \right)^{k_i} ,
\end{multline}
where the contours enclose poles at $-\hat{N}_1$ and $-\hat{N}_2$, but not $\hat{\alpha} + \hat{M}$, and are nested with $|v_1| \ll |v_2|$.
\end{lemma}

\begin{proof}
The proof is very similar to the proof of Lemma \ref{lemma:cov:as}.
By Theorem \ref{thm:var:dis} we obtain that
\begin{multline}
\mathbb{E}\left[
\prod_{i=1}^2 \left(
\mathfrak{P}_{k_i}(x^{N_i}) - \mathfrak{P}_{k_i}(x^{N_i-1}) - \mathbb{E}\left(\mathfrak{P}_{k_i}(x^{N_i}) - \mathfrak{P}_{k_i}(x^{N_i-1}) \right) \right)
\right]
\\
=
\frac{(-\theta)^{-2}}{(2\pi\im)^{k_1 + k_2}}
\oint \cdots \oint \prod_{i=1}^2 \left[ \frac{1}{(u_{i,2}-u_{i,1}+1-\theta)\cdots (u_{i,k_i} - u_{i,k_i-1}+1-\theta)} \right.
\\
\times
\prod_{1\leq i'<j'\leq k_i}\frac{(u_{i,j'}-u_{i,i'}) (u_{i,j'}-u_{i,i'}+1-\theta)}{(u_{i,j'}-u_{i,i'}-\theta) (u_{i,j'}-u_{i,i'}+1)}
\prod_{i'=1}^{k_i}\frac{u_{i,i'}-\theta}{u_{i,i'}+(N_i-1)\theta}\cdot \frac{u_{i,i'}-\theta\alpha}{u_{i,i'}-\theta\alpha - \theta M}du_{i,i'}
\\
\left.
\times \left(1 -  \prod_{i'=1}^{k_i}\frac{u_{i,i'}+(N_i-1)\theta}{u_{i,i'}+(N_i-2)\theta} \right) \right]
\left( \prod_{ \substack{1\leq i'\leq k_1, \\ 1\leq j'\leq k_2}}\frac{(u_{1,i'}-u_{2,j'}) (u_{1,i'}-u_{2,j'}+1-\theta)}{(u_{1,i'}-u_{2,j'}-\theta) (u_{1,i'}-u_{2,j'}+1)} - 1 \right) ,
\end{multline}
where the contours for $u_{i, k_1}, \cdots, u_{i, k_i}$ enclose $-\theta (N_i - 2)$ and $-\theta (N_i - 1)$
but not $\theta(\alpha + M)$, for $i = 1, 2$.
We also require that
$|u_{1,1}|\ll \cdots \ll |u_{1,k_1}| \ll |u_{2,1}|\ll \cdots \ll |u_{2,k_2}|$.

Again set $u_{i,i'} = L\theta v_{i,i'}$ for $i = 1, 2$ and any $1\leq i' \leq k_i$.
Sending $L\rightarrow \infty$, using (\ref{eq:cov:obs1}) and (\ref{eq:cov:obs2}), and applying Corollary \ref{cor:dr} to $v_{i, k_1}, \cdots, v_{i, k_i}$ and $v_{j, k_1}, \cdots, v_{j, k_j}$, respectively,
we eventually get (\ref{eq:lemma:cov:d}).
\end{proof}

\begin{lemma}  \label{lemma:cov:s}
Under the limit scheme \eqref{eq:lsch}, we have
\begin{multline}   \label{eq:lemma:cov:s}
\lim_{L\rightarrow \infty} L \cdot
\mathbb{E}\left[
\prod_{i=1}^2 \left(
\mathfrak{P}_{k_i}(x^{N}) - \mathfrak{P}_{k_i}(x^{N-1}) - \mathbb{E}\left(\mathfrak{P}_{k_i}(x^{N}) - \mathfrak{P}_{k_i}(x^{N-1}) \right) \right)
\right]
\\
=
-
\frac{\theta^{-1}k_1k_2}{2\pi\im(k_1 + k_2)}
\oint
\frac{dv}{(v + \hat{N})^2}
\left( \frac{v}{v+\hat{N}}\cdot \frac{v-\hat{\alpha}}{v - \hat{\alpha} - \hat{M}} \right)^{k_2 + k_2} ,
\end{multline}
where the contours enclose poles at $-\hat{N}$ but not $\hat{\alpha} + \hat{M}$.
\end{lemma}

\begin{proof}
We can write the expectation as
\begin{multline}
\mathbb{E}\left[
\prod_{i=1}^2 \left(
\mathfrak{P}_{k_i}(x^{N}) - \mathfrak{P}_{k_i}(x^{N-1}) - \mathbb{E}\left(\mathfrak{P}_{k_i}(x^{N}) - \mathfrak{P}_{k_i}(x^{N-1}) \right) \right)
\right]
\\
=
\mathbb{E}\left[ \left( \mathfrak{P}_{k_1}(x^N) - \mathfrak{P}_{k_1}(x^{N-1}) \right) \mathfrak{P}_{k_2}(x^N) \right]
- \mathbb{E}\left[ \mathfrak{P}_{k_2}(x^{N-1}) \left( \mathfrak{P}_{k_1}(x^N) - \mathfrak{P}_{k_1}(x^{N-1}) \right) \right] \\
-
\mathbb{E}\left( \mathfrak{P}_{k_1}(x^N) - \mathfrak{P}_{k_1}(x^{N-1}) \right) \mathbb{E} \left(  \mathfrak{P}_{k_2}(x^N) \right)
+
\mathbb{E}\left( \mathfrak{P}_{k_2}(x^{N-1}) \right) \mathbb{E} \left( \mathfrak{P}_{k_1}(x^N) - \mathfrak{P}_{k_1}(x^{N-1}) \right) .
\end{multline}

Now apply Lemma \ref{lemma:cov:as}, and we obtain
\begin{multline}  \label{eq:lemma:cov:s:pf1}
\lim_{L\rightarrow \infty} L \cdot
\mathbb{E}\left[
\prod_{i=1}^2 \left(
\mathfrak{P}_{k_i}(x^{N}) - \mathfrak{P}_{k_i}(x^{N-1}) - \mathbb{E}\left(\mathfrak{P}_{k_i}(x^{N}) - \mathfrak{P}_{k_i}(x^{N-1}) \right) \right)
\right]
\\
=
\frac{k_1 \theta^{-1}}{(2\pi\im)^2 }
\oint \oint \left[
\left( \frac{v_{2}}{v_{2}+\hat{N}}\cdot \frac{v_{2}-\hat{\alpha}}{v_{2}- \hat{\alpha} - \hat{M}} \right)^{k_1}
\left( \frac{v_{1}}{v_{1}+\hat{N}}\cdot \frac{v_{1}-\hat{\alpha}}{v_{1}- \hat{\alpha} - \hat{M}} \right)^{k_2}
\frac{1}{(v_2 + \hat{N})(v_{2} - v_1)^2} \right.\\
- \left.
\left( \frac{v_{1}}{v_{1}+\hat{N}}\cdot \frac{v_{1}-\hat{\alpha}}{v_{1}- \hat{\alpha} - \hat{M}} \right)^{k_1}
\left( \frac{v_{2}}{v_{2}+\hat{N}}\cdot \frac{v_{2}-\hat{\alpha}}{v_{2}- \hat{\alpha} - \hat{M}} \right)^{k_2}
\frac{1}{(v_1 + \hat{N})(v_{2} - v_1)^2} \right] dv_1 dv_2  ,
\end{multline}
where
the contours of $v_1$ and $v_2$ enclose $-\hat{N}$
but not $\hat{\alpha} + \hat{M}$;
and $|v_{1}|\ll |v_{2}|$.

Interchanging $k_1$ and $k_2$, for the same limit we have
\begin{multline}  \label{eq:lemma:cov:s:pf2}
\frac{k_2 \theta^{-1}}{(2\pi\im)^2 }
\oint \oint \left[
\left( \frac{v_{2}}{v_{2}+\hat{N}}\cdot \frac{v_{2}-\hat{\alpha}}{v_{2}- \hat{\alpha} - \hat{M}} \right)^{k_2}
\left( \frac{v_{1}}{v_{1}+\hat{N}}\cdot \frac{v_{1}-\hat{\alpha}}{v_{1}- \hat{\alpha} - \hat{M}} \right)^{k_1}
\frac{1}{(v_2 + \hat{N})(v_{2} - v_1)^2} \right.\\
- \left.
\left( \frac{v_{1}}{v_{1}+\hat{N}}\cdot \frac{v_{1}-\hat{\alpha}}{v_{1}- \hat{\alpha} - \hat{M}} \right)^{k_2}
\left( \frac{v_{2}}{v_{2}+\hat{N}}\cdot \frac{v_{2}-\hat{\alpha}}{v_{2}- \hat{\alpha} - \hat{M}} \right)^{k_1}
\frac{1}{(v_1 + \hat{N})(v_{2} - v_1)^2} \right] dv_1 dv_2  ,
\end{multline}where
the contours of $v_1$ and $v_2$ enclose $-\hat{N}$
but not $\hat{\alpha} + \hat{M}$;
and $|v_{1}|\ll |v_{2}|$.

Notice that $(\ref{eq:lemma:cov:s:pf1}) \times \frac{k_2}{k_1 + k_2} + (\ref{eq:lemma:cov:s:pf2}) \times \frac{k_1}{k_1 + k_2}$ equals
\begin{multline}   \label{eq:lemma:cov:s:pf3}
-
\frac{\theta^{-1}}{(2\pi\im)^2 }
\frac{k_1 k_2}{k_1 + k_2}
\oint \oint
\frac{1}{(v_1 + \hat{N})(v_2 + \hat{N})(v_{2} - v_1)} \\
\times \left(
\left( \frac{v_{2}}{v_{2}+\hat{N}}\cdot \frac{v_{2}-\hat{\alpha}}{v_{2}- \hat{\alpha} - \hat{M}} \right)^{k_1}
\left( \frac{v_{1}}{v_{1}+\hat{N}}\cdot \frac{v_{1}-\hat{\alpha}}{v_{1}- \hat{\alpha} - \hat{M}} \right)^{k_2} \right. \\
+ \left.
\left( \frac{v_{1}}{v_{1}+\hat{N}}\cdot \frac{v_{1}-\hat{\alpha}}{v_{1}- \hat{\alpha} - \hat{M}} \right)^{k_1}
\left( \frac{v_{2}}{v_{2}+\hat{N}}\cdot \frac{v_{2}-\hat{\alpha}}{v_{2}- \hat{\alpha} - \hat{M}} \right)^{k_2} \right) dv_1 dv_2
\end{multline}
where
the contours enclose $-\hat{N}$
but not $\hat{\alpha} + \hat{M}$, and $|v_1| \ll |v_2|$.
By Theorem \ref{thm:dr} this equals the right hand side of \eqref{eq:lemma:cov:s}.
\end{proof}

\subsection{Decay of cumulants: proof of Theorem \ref{thm:clt1}}  \label{ssec:gf}

In this section we present formulas about the decay of the high order cumulants of \eqref{eq:clt1:d:cov} and \eqref{eq:clt1:cov}.
They will lead to the proof of Theorem \ref{thm:clt1}.


\begin{prop}  \label{prop:gtp}
Let $k_1, \cdots, k_h$ and $N_1 \leq \cdots \leq N_h$ be positive integers,
and let $D \subset \{1, \cdots, h\}$ be a subset of indices, satisfying that for any $1 \leq i < j \leq h$ and $j \in D$, there is $N_i < N_j$.

For any $i \in D$, denote
\begin{equation}   \label{eq:prop:gtp:nt1}
\mathfrak{E}_i =
\mathfrak{P}_{k_i}(x^{N_i}) - \mathfrak{P}_{k_i}(x^{N_i-1}) - \mathbb{E}\left(\mathfrak{P}_{k_i}(x^{N_i}) - \mathfrak{P}_{k_i}(x^{N_i-1}) \right) ,
\end{equation}
and for any $i \not\in D$, denote
\begin{equation}   \label{eq:prop:gtp:nt2}
\mathfrak{E}_i =
\mathfrak{P}_{k_i}(x^{N_i}) - \mathbb{E}\left(\mathfrak{P}_{k_i}(x^{N_i}) \right) .
\end{equation}
Then for any $\eta < h - 2 + |D|$, we have
\begin{equation} \label{eq:gtp}
\lim_{L\rightarrow \infty} L^{\eta} \kappa\left( \mathfrak{E}_1, \cdots, \mathfrak{E}_h \right)
 = 0 .
\end{equation}
\end{prop}

\begin{proof}
For a random vector, adding a constant vector only adds a first order term to the log of its characteristic function;
thus for $h \geq 2$, 
\begin{multline}   \label{eq:prop:gtp:pf:00}
\lim_{L\rightarrow \infty}
L^{\eta} \kappa\left( \mathfrak{E}_1, \cdots, \mathfrak{E}_h \right) = 
\lim_{L\rightarrow \infty} L^{\eta} \sum_{\substack{t \in \mathbb{Z}_{>0} \\ \left\{ U_1, \cdots, U_t \right\} \in \Theta_h}}  (-1)^{t-1} (t-1)!
\\
\times
\prod_{r=1}^t \mathbb{E}\left[\prod_{i\in U_r \bigcap D} \left( \mathfrak{P}_{k_i}(x^{N_i}) - \mathfrak{P}_{k_i}(x^{N_i - 1})  \right)
\prod_{i\in U_r \backslash D} \mathfrak{P}_{k_i}(x^{N_i}) \right] .
\end{multline}


For any fixed $\left\{ U_1, \cdots, U_t \right\} \in \Theta_{h}$,
we apply Theorem \ref{thm:var:dis} to expectations in the following form:
\begin{equation}
\mathbb{E}\left[\prod_{i\in U_r \bigcap D} (- 1 )^{\lambda_i} \mathfrak{P}_{k_i}(x^{N_i - \lambda_i})
\prod_{i\in U_r \backslash D} \mathfrak{P}_{k_i}(x^{N_i}) \right],
\end{equation}
where $1 \leq r \leq t$, and each $\lambda_i \in \{0, 1\}$.
We multiply the contour integrals \eqref{eq:thm:var:dis:st} over all $r=1,\dots,t$, and then sum the result over all choices of $\{ \lambda_i \}_{i\in D}$.
Note that for any $i \in D$, $i > 1$, we have assumed $N_{i-1} \leq N_i - 1 < N_i$;
thus the nesting order of the contour integrals is the same for different choices of $\{ \lambda_i \}_{i\in D}$.
Therefore, we have
\begin{multline}   \label{eq:lem:ctr:pf1}
\prod_{r=1}^t \mathbb{E}\left[\prod_{i\in U_r \bigcap D} \left( \mathfrak{P}_{k_i}(x^{N_i}) - \mathfrak{P}_{k_i}(x^{N_i - 1})  \right)
\prod_{i\in U_r \backslash D} \mathfrak{P}_{k_i}(x^{N_i}) \right]
\\
=
\frac{(-\theta)^{-h}}{(2\pi\im)^{k_1 + \cdots + k_{h}}} \oint \cdots \oint
\prod_{i \in \{ 1, \cdots, h\} \backslash D} \mathfrak{I}(u_{i,1}, \cdots, u_{i,k_i}; \alpha, M, \theta, N_i)
\\
\times
\prod_{i \in D}
\left(
\mathfrak{I}(u_{i,1}, \cdots, u_{i,k_i}; \alpha, M, \theta, N_i)
-
\mathfrak{I}(u_{i,1}, \cdots, u_{i,k_i}; \alpha, M, \theta, N_i - 1)
\right)
\\
\times \prod_{r=1}^{t}  \prod_{\substack{i < j,\\ i, j \in U_r } } \mathfrak{L}(u_{i,1}, \cdots, u_{i,k_i}; u_{j,1}, \cdots, u_{j,k_j}; \theta)
\prod_{i=1}^{h} \prod_{i'=1}^{k_i} du_{i,i'}
,
\end{multline}
where the contours are nested such that
for each $1\leq i \leq h$,
we have $|u_{i,1}| \ll \cdots \ll |u_{i,k_i}|$;
for $1 \leq i \leq h-1$,
we have $|u_{i, k_i}| \ll |u_{i+1, 1}|$.

For each $1\leq i < j \leq h$,
denote
\begin{equation}
\mathfrak{M}_{i,j} = \mathfrak{L}(u_{i,1}, \cdots, u_{i,k_i}; u_{j,1}, \cdots, u_{j,k_j}; \theta) - 1.
\end{equation}
Now sum \eqref{eq:lem:ctr:pf1} over all $\left\{ U_1, \cdots, U_t \right\} \in \Theta_{h}$, and rewrite the expression with the notation $\mathfrak{M}_{i,j}$:
\begin{multline}  \label{eq:prop:gtp:pf1}
\sum_{\substack{t \in \mathbb{Z}_{>0} \\ \left\{ U_1, \cdots, U_t \right\} \in \Theta_h}}  (-1)^{t-1} (t-1)!
\prod_{r=1}^t \mathbb{E}\left[\prod_{i\in U_r \bigcap D} \left( \mathfrak{P}_{k_i}(x^{N_i}) - \mathfrak{P}_{k_i}(x^{N_i - 1})  \right)
\prod_{i\in U_r \backslash D} \mathfrak{P}_{k_i}(x^{N_i}) \right]
\\
=
\frac{(-\theta)^{-h}}{(2\pi\im)^{k_1 + \cdots + k_{h}}} \oint \cdots \oint
\prod_{i \in \{1, \cdots, h\} \backslash D} \mathfrak{I}(u_{i,1}, \cdots, u_{i,k_i}; \alpha, M, \theta, N_i)
\\
\times
\prod_{i \in D}
\left(
\mathfrak{I}(u_{i,1}, \cdots, u_{i,k_i}; \alpha, M, \theta, N_i)
-
\mathfrak{I}(u_{i,1}, \cdots, u_{i,k_i}; \alpha, M, \theta, N_i - 1)
\right)
\\
\times \sum_{\substack{t \in \mathbb{Z}_{>0} \\ \left\{ U_1, \cdots, U_t \right\} \in \Theta_h}}  (-1)^{t-1} (t-1)! \prod_{r=1}^{t} \prod_{\substack{i < j,\\ i, j \in U_r } } \left(\mathfrak{M}_{i,j} + 1 \right)
\prod_{i=1}^{h} \prod_{i'=1}^{k_i} du_{i,i'}
.
\end{multline}


We use a combinatorial argument to simplify the last line of \eqref{eq:prop:gtp:pf1}.
For any set $U \subset \{ 1, \cdots, h\}$, let $\mathcal{T}(U)$ be the set of all undirected simple graphs whose vertices are labeled by $U$,
and $\mathcal{L}(U) \subset \mathcal{T}(U)$ be the set of such graphs that are connected.
We also denote
\begin{equation}
T(U) = \sum_{\Omega \in \mathcal{T}(U)} \prod_{\substack{i < j,\\ (i, j) \in \Omega }} \mathfrak{M}_{i,j} , \quad
R(U) = \sum_{\Omega \in \mathcal{L}(U)} \prod_{\substack{i < j,\\ (i, j) \in \Omega }} \mathfrak{M}_{i,j} .
\end{equation}
Then we have
\begin{equation}  \label{eq:new0}
T(U) 
=
\sum_{t \in \mathbb{Z}_{>0}, \{U_1, \cdots, U_t \}}
\prod_{k=1}^t R(U_k) ,
\end{equation}
where the sum is over all partitions of $U$.
By induction on $|U|$ (or (generalized) M\"{o}bius inversion formula), we invert \eqref{eq:new0}, and get
\begin{equation}  \label{eq:new1}
R(U)
=
\sum_{t \in \mathbb{Z}_{>0}, \{U_1, \cdots, U_t \}}
(-1)^{t-1}(t-1)!\prod_{k=1}^t T(U_k) ,
\end{equation}
where the sum is over all partitions of $U$.
Take $U = \{1, \cdots, h\}$, then the right hand side of \eqref{eq:new1} is precisely the last line of \eqref{eq:prop:gtp:pf1}.

We replace the last row of \eqref{eq:prop:gtp:pf1} by $R(\{1, \cdots, h\})$.
We further set $u_{i,i'} = L\theta v_{i,i'}$ for any $1\leq i \leq h$, $1\leq i' \leq k_i$.
By changing notations, \eqref{eq:prop:gtp:pf1} multiplied by $L^{\eta}$ can be written as
\begin{equation}  \label{eq:prop:gtp:pf3}
\frac{(- 1)^{-h} }{(2\pi\im)^{k_1 + \cdots + k_{h}}} \oint \cdots \oint
\prod_{i = 1}^h \mathcal{S}_i
\prod_{i \in D} \mathcal{R}_i
\sum_{ \Omega \in \mathcal{L}(\{1, \cdots, h\}) }
L^{\eta + h - |D| - 2|\Omega|}
\prod_{\substack{i < j, \\ (i,j) \in \Omega}} \mathfrak{N}_{i,j}
\prod_{i=1}^{h} \prod_{i'=1}^{k_i} dv_{i,i'},
\end{equation}
where $|\Omega|$ is the number of edges in $\Omega$, and 
\begin{multline}
\mathcal{S}_i
= \frac{1}{\left( v_{i,2}-v_{i,1} + (\theta^{-1} - 1)L^{-1} \right)\cdots \left( v_{i,k_i}-v_{i,k_i - 1} + (\theta^{-1} - 1)L^{-1} \right)} \\
\times \prod_{1\leq i'<j' \leq m}\frac{(v_{i,j'}-v_{i,i'}) (v_{i,j'}-v_{i,i'}+(\theta^{-1} - 1)L^{-1})}{(v_{i,j'}-v_{i,i'}- L^{-1}) (v_{i,j'}-v_{i,i'}+\theta^{-1}L^{-1})}
\prod_{i'=1}^{k_i}\frac{v_{i,i'}-L^{-1} }{v_{i,i'} + \hat{N}_i - L^{-1}}\cdot \frac{v_{i,i'}-\hat{\alpha}}{v_{i,i'}-\hat{\alpha} - \hat{M}} ,
\end{multline}
\begin{equation}
\mathcal{R}_i =
L
\left( 1 - \prod_{i'=1}^{k_i} \frac{v_{i,i'} + \hat{N}_i - L^{-1}}{v_{i,i'} + \hat{N}_i - 2L^{-1}} \right),
\end{equation}
and
\begin{equation}
\mathfrak{N}_{i,j}
=
L^2 \mathfrak{M}_{i,j}
=
L^2
\left(
\prod_{1\leq i'\leq k_i, 1\leq j' \leq k_j}\frac{ (v_{i,i'} - v_{j,j'}) (v_{i,i'} - v_{j,j'} + (\theta^{-1} - 1)L^{-1}) }{(v_{i,i'} - v_{j,j'} - L^{-1}) (v_{i,i'} - v_{j,j'} + \theta^{-1}L^{-1})} - 1 \right).
\end{equation}
Also, the contours in (\ref{eq:prop:gtp:pf3}) are nested, such that for any $1\leq i \leq h$,
we have $|v_{i,1}| \ll \cdots \ll |v_{i,k_i}|$;
for $1 \leq i \leq h-1$,
we have $|u_{i, k_i}| \ll |u_{i+1, 1}|$;
and all contours enclose all $ - \hat{N}_i$, but not $\hat{\alpha} + \hat{M}$.

We briefly explain the exponent of $L$ in \eqref{eq:prop:gtp:pf3}:
the change of variables produces $L^{k_1 + \cdots + k_h}$,
each $\mathcal{S}_i$ produces $L^{-k_i + 1}$,
each $\mathcal{R}_i$ produces $L^{-1}$,
and each $\mathfrak{N}_{i,j}$ produces $L^{-2}$.

As $L \rightarrow \infty$, each $\mathcal{S}_i$ obviously converges.
Besides, we also have
\begin{equation}
\mathcal{R}_i
=
- \sum_{i'=1}^{k_i} \frac{1}{v_{i,i'} + \hat{N}_i - 2L^{-1}} + O(L^{-1})
\end{equation}
and
\begin{equation}
\mathfrak{N}_{i,j}
=
\prod_{1\leq i'\leq k_i, 1\leq j' \leq k_j}\frac{ \theta^{-1} }{(v_{i,i'} - v_{j,j'} - L^{-1}) (v_{i,i'} - v_{j,j'} + \theta^{-1}L^{-1})} + O(L^{-2}) ,
\end{equation}
so when send $L \rightarrow \infty$, the integrand in \eqref{eq:prop:gtp:pf3} converges to zero if $\eta + h - |D| - 2|\Omega| < 0$.


For any $\Omega \subset \mathcal{L}(\{1, \cdots, h\})$, we have $|\Omega| \geq h - 1$.
Since we require that $\eta < h - 2 + |D|$, we have $\eta + h - |D| - 2|\Omega| < 0$.
We also note that the convergence of the integrand in \eqref{eq:prop:gtp:pf3} is uniform in the variables $v_{i,i'}$ in the contours.
This means that the integral also converges to zero, and we finish the proof.
\end{proof}
\begin{rem}
Ideas similar to the combinatoric arguments to obtain \eqref{eq:new1} can be found in standard arguments of cluster expansions in statistical physics; see e.g. \cite[Section 2]{faris2010combinatorics}, \cite[Chapter 9]{book:18320} for more discussions.
\end{rem}

From the proof, we also conclude that the convergence is uniform in $\hat{N}_1, \cdots, \hat{N}_h$.
\begin{cor}   \label{cor:clt2:bd}
For any $G \in \mathbb{R}_{>0}$,
there is a constant $C(\hat{\alpha}, \hat{M}, G, k_1, \cdots, k_h)$
, independent of $L$, $\hat{N}_1, \cdots, \hat{N}_h$, and $D$,
such that for any $0 \leq \hat{N}_1 \leq \cdots \leq \hat{N}_h \leq G$, satisfying $\hat{N}_i < \hat{N}_j$ for any $1\leq i < j \leq h$, $j \in D$, we have
\begin{equation}  \label{eq:clt2:bd}
L^{\eta} \left|
\kappa\left( \mathfrak{E}_1, \cdots, \mathfrak{E}_h \right)
\right|
\leq C(\hat{\alpha}, \hat{M}, G, k_1, \cdots, k_h)
\end{equation}
for any $L > C(\hat{\alpha}, \hat{M}, G, k_1, \cdots, k_h)$.
\end{cor}

\begin{proof}
In \eqref{eq:prop:gtp:pf3}, we can fix the contours for all $v_{i, i'}$, such that they are nested and each encloses the line segment $[-G, 0]$ but not $\hat{\alpha} + \hat{M}$, when $L$ large enough.
Then each of $|\mathcal{S}_j|$, $|\mathcal{R}_j|$, and $|\mathfrak{N}_{i, j}|$ is upper bounded by a constant relying on the chosen contours, and so is the integral.
\end{proof}

We proceed to remove the strict ordering constraints in Proposition \ref{prop:gtp}
\begin{prop}  \label{prop:gtp2}
Let $k_1, \cdots, k_h$ and $N_1 \leq \cdots \leq N_h$ be positive integers,
and let $D \subset \{1, \cdots, h\}$.
Let the notation $\mathfrak{E}_i$ be the same as in Proposition \ref{prop:gtp}, for $i \in D$ and $i \not\in D$, respectively.

Suppose that the number of different values among $\left\{ \hat{N}_i \right\}_{i \in D}$ is $s$,
then for any $\eta < h - 2 + s$, we have
\begin{equation} \label{eq:gtp2}
\lim_{L\rightarrow \infty} L^{\eta} \kappa\left( \mathfrak{E}_1, \cdots, \mathfrak{E}_h \right)
 = 0 .
\end{equation}
\end{prop}
\begin{proof}
Suppose that $\hat{N}_{i_1}, \cdots, \hat{N}_{i_s}$ include all the $s$ values in $\{ \hat{N}_i \}_{i \in D}$, and $i_1, \cdots, i_s \in D$.
For any $i \in D \backslash \{i_1, \cdots, i_s\}$, denote
\begin{equation}
\mathfrak{E}_i' =
\mathfrak{P}_{k_i}(x^{N_i}) - \mathbb{E}\left(\mathfrak{P}_{k_i}(x^{N_i}) \right) ,\quad
\mathfrak{E}_i'' =
\mathfrak{P}_{k_i}(x^{N_i-1}) - \mathbb{E}\left(\mathfrak{P}_{k_i}(x^{N_i-1}) \right) .
\end{equation}
Then via the identity $\mathfrak{E}_i = \mathfrak{E}_i' - \mathfrak{E}_i''$,
we can write $\kappa\left( \mathfrak{E}_1, \cdots, \mathfrak{E}_h \right)$ as a sum of $2^{|D| - s}$ cumulants.
We get \eqref{eq:gtp2} by applying Proposition \ref{prop:gtp} to each of them and summing them up.
\end{proof}

Based on Proposition \ref{prop:gtp2}, we finish the proof of Theorem \ref{thm:clt1}.

\begin{proof}[Proof of Theorem \ref{thm:clt1}]
By Lemma \ref{lemma:geq}, to show that (\ref{eq:clt1:d:cov}) and (\ref{eq:clt1:cov}) jointly weakly converge to a Gaussian vector, it suffices to show that each cumulant of order greater than two converges to zero, and the first two moments converge to the desired values.
Since all the first order moments are zero, and covariances are given by \cite[Theorem 4.1]{borodin2015general}, Lemma \ref{lemma:cov:d}, and Lemma \ref{lemma:cov:s}, it suffices to consider the cumulants of order greater than two.
For any positive integers $h \geq 3$, $k_1, \cdots, k_h$ and $N_1 \leq \cdots \leq N_h$, and $D \subset \{1, \cdots, h\}$ a subset of indices,
let the notation $\mathfrak{E}_i$ be the same as in Proposition \ref{prop:gtp}, for $i \in D$ and $i \not\in D$, respectively.
We need that
\begin{equation} \label{eq:thm:clt1:pf1}
\lim_{L\rightarrow \infty} L^{\frac{|D|}{2}}
\kappa\left( \mathfrak{E}_1, \cdots, \mathfrak{E}_h \right)
 = 0 ,
\end{equation}

By Proposition \ref{prop:gtp2},
under the limit scheme \eqref{eq:lsch}, \eqref{eq:lsch2} we have
\begin{equation}  \label{eq:thm:clt1:pf11}
\lim_{L\rightarrow \infty} L^{h - 1 - \epsilon} \kappa\left( \mathfrak{E}_1, \cdots, \mathfrak{E}_h \right) = 0 .
\end{equation}
for any $\epsilon > 0$.
Since when $h \geq 3$, $h - 1 > \frac{h}{2} \geq \frac{|D|}{2}$, we immediately obtain \eqref{eq:thm:clt1:pf1}.
\end{proof}

\begin{rem} \label{rem:add2}
According to Corollary \ref{cor:clt2:bd}, for any $G \in \mathbb{R}_{>0}$, the convergence of (\ref{eq:thm:clt1:pf11}) is uniform in $\hat{N}_1, \cdots, \hat{N}_h \in [0, G]$.
\end{rem}

\subsection{Integration over levels: proof of Theorem \ref{thm:clt2}}  \label{ssec:gs}




\begin{proof}[Proof of Theorem \ref{thm:clt2}]
By rescaling, it suffices to consider the case where $G = 1$.
To simplify notations, denote
\begin{equation}
\mathfrak{C}_i(y) = \mathfrak{P}_{k_i}(x^{\lfloor y \rfloor}) - \mathfrak{P}_{k_i}(x^{\lfloor y \rfloor-1}) - \mathbb{E}\left(\mathfrak{P}_{k_i}(x^{\lfloor y \rfloor}) - \mathfrak{P}_{k_i}(x^{\lfloor y \rfloor-1})\right)   .
\end{equation}

By Lemma \ref{lemma:geq}, to prove that \eqref{eq:clt2:def} is asymptotically Gaussian, it suffices to show that all cumulants of order greater than two converge to zero.
That is, it suffices to show that, for any positive integers $h \geq 3, k_1, \cdots, k_h$, $N_1 \leq \cdots \leq N_h$,
and any functions $g_1, \cdots, g_{h} \in L^{\infty}([0, 1])$, we should have
\begin{equation}   \label{eq:thm:clt2:pf0}
\lim_{L\rightarrow \infty} L^{h}
\kappa\left(
\int_0^1 g_1(y) \mathfrak{C}_1(Ly) dy, \cdots,
\int_0^1 g_h(y) \mathfrak{C}_h(Ly) dy
\right) 
= 0 .
\end{equation}
By the multi-linearity of cumulants, the left hand side of \eqref{eq:thm:clt2:pf0} equals
\begin{equation}   \label{eq:thm:clt2:pf005}
\lim_{L \rightarrow \infty}
\int_0^1 \cdots \int_0^1 L^h
\kappa\left( \mathfrak{C}_1(Ly_1), \cdots, \mathfrak{C}_h(Ly_h) \right)
\prod_{i=1}^h g_i(y_i)dy_i   .
\end{equation}
The expression (\ref{eq:thm:clt2:pf005}) can be split into a (finite) sum of integrals in the form of
\begin{multline} \label{eq:thm:clt2:pf01}
\int_0^1 \cdots \int_0^1
\mathds{1}_{\lfloor Ly_1 \rfloor = \cdots = \lfloor Ly_{c_1} \rfloor
< \lfloor Ly_{c_1+1} \rfloor = \cdots = \lfloor Ly_{c_2} \rfloor
< \cdots
< \lfloor Ly_{c_{s-1}+1} \rfloor = \cdots = \lfloor Ly_{c_s} \rfloor }
\\
\times
L^h
\kappa\left( \mathfrak{C}_1(Ly_1), \cdots, \mathfrak{C}_h(Ly_h) \right)
\prod_{i=1}^h g_i(y_i)dy_i   ,
\end{multline}
where $1 \leq s \leq h$ and $1\leq c_1 < \cdots < c_s = h$.
Since each $g_i$ is almost everywhere bounded, there is a number $K$, independent of $L$, such that (\ref{eq:thm:clt2:pf01}) is bounded above by
\begin{multline} \label{eq:thm:clt2:pf02}
\int_0^1 \cdots \int_0^1
K \cdot
\mathds{1}_{\lfloor Ly_1 \rfloor = \cdots = \lfloor Ly_{c_1} \rfloor
< \lfloor Ly_{c_1+1} \rfloor = \cdots = \lfloor Ly_{c_2} \rfloor
< \cdots
< \lfloor Ly_{c_{s-1}+1} \rfloor = \cdots = \lfloor Ly_{c_s} \rfloor }
\\
\times
L^h
\left|
\kappa\left( \mathfrak{C}_1(Ly_1), \cdots, \mathfrak{C}_h(Ly_h) \right)
\right|
\prod_{i=1}^h
dy_i   .
\end{multline}
Note that each $\mathfrak{C}_j(Ly_j)$ is constant when $y_j \in \left[\frac{m}{L}, \frac{m+1}{L} \right)$, for any $0 \leq m \leq L - 1$.
For each $i = 1, \cdots, s$, we can replace each of $y_{c_{i-1} + 1}, \cdots, y_{c_i}$ by the same variable $z_i$ (here and below $c_{0} = 0$).
Thus (\ref{eq:thm:clt2:pf02}) becomes
\begin{equation}  \label{eq:thm:clt2:pf1}
\int_0^1 \cdots \int_0^1
K \cdot
\mathds{1}_{\lfloor Lz_1 \rfloor
< \cdots
< \lfloor Lz_{s} \rfloor}
L^s
\left|
\kappa\left( \mathfrak{W}_1, \cdots, \mathfrak{W}_h \right)
\right|
\prod_{i=1}^s  dz_i,
\end{equation}
here for any $1 \leq i \leq s$ and $1\leq j \leq h$, such that $c_{i-1}< j \leq c_i$, we denote $\mathfrak{W}_j = \mathfrak{C}_j(Lz_i)$.

Since $h \geq 3$, we have $s < h - 2 + s$. By Proposition \ref{prop:gtp2}, as $L \rightarrow \infty$ the integrand of (\ref{eq:thm:clt2:pf1}) converges to $0$ for each $z_1, \cdots, z_s$;
and by Corollary \ref{cor:clt2:bd} the integrand is bounded regardless of $L$ and $z_1, \cdots, z_s$.
Using the dominated convergence theorem we conclude that (\ref{eq:thm:clt2:pf1}) converges to $0$, as $L \rightarrow \infty$.
This implies (\ref{eq:thm:clt2:pf0}).

It remains to match the covariances.
For any functions $g_1, g_2 \in L^{\infty}([0, 1])$, we have
\begin{multline} \label{eq:thm:clt2:pf03}
\lim_{L\rightarrow \infty}
L^2
\mathbb{E} \left( \prod_{i=1}^{2} \int_0^1 g_i(y) \mathfrak{C}_i(Ly) dy \right)
=
\lim_{L\rightarrow \infty}
\int_0^1 \int_0^1
L^2
\mathbb{E} \left( \prod_{i=1}^2
\mathfrak{C}_i(Ly_i)
\right)
\prod_{i=1}^2
g_i(y_i) dy_i
\\
= 
\lim_{L\rightarrow \infty}
\iint_{\lfloor Ly_1 \rfloor < \lfloor Ly_2 \rfloor}
L^2
\left[
\mathbb{E} \left(
\mathfrak{C}_1(Ly_1)
\mathfrak{C}_2(Ly_2)
\right)
g_1(y_1)
g_2(y_2) 
\right.
\\
\left.
+
\mathbb{E} \left(
\mathfrak{C}_1(Ly_2)
\mathfrak{C}_2(Ly_1)
\right)
g_1(y_2)
g_2(y_1) 
\right]
dy_1 dy_2
\\
+
\lim_{L\rightarrow\infty}
\int_0^1
L^3
\left(\iint_{\left[ \frac{\lfloor Ly \rfloor}{L}, \frac{\lfloor Ly + 1\rfloor}{L} \right]^2} \prod_{i=1}^2 g_i(y_i) dy_i \right)
\mathbb{E} \left( \prod_{i=1}^2 \mathfrak{C}_i(Ly) \right) dy  .
\end{multline}
By Lebesgue Differentiation Theorem, for $y$ at almost everywhere we have
\begin{equation}
\lim_{L\rightarrow \infty} L^2 \iint_{\left[ \frac{\lfloor Ly \rfloor}{L}, \frac{\lfloor Ly + 1\rfloor}{L} \right]^2} \prod_{i=1}^2 g_i(y_i) dy_i = g_1(y)g_2(y).
\end{equation}
From the expectations computed in Section \ref{ssec:cc}, the integrands in (\ref{eq:thm:clt2:pf03}) converges pointwise.
By Corollary \ref{cor:clt2:bd}, and since $g_1$ and $g_2$ are bounded,
the integrands in \eqref{eq:thm:clt2:pf03} are uniformly bounded.
We can thus move the limits inside the integrals.
By Lemma \ref{lemma:cov:d} and \ref{lemma:cov:s}, 
\eqref{eq:thm:clt2:pf03} equals
\begin{multline}
\iint_{0 \leq y_1 < y_2 \leq 1}
\frac{\theta^{-1}}{(2\pi\im)^{2}}
\oint \oint
\frac{k_1 k_2}{(v_1 - v_2)^2(v_1 + y_1 )(v_2 + y_2 )}
\\
\times 
\left[
g_1(y_1) g_2(y_2)
\left( \frac{v_1}{v_1+y_1}\cdot \frac{v_1-\hat{\alpha}}{v_1- \hat{\alpha} - \hat{M}} \right)^{k_1}
\left( \frac{v_2}{v_2+y_2}\cdot \frac{v_2-\hat{\alpha}}{v_2- \hat{\alpha} - \hat{M}} \right)^{k_2}
\right.
\\
\left.
 + g_1(y_2) g_2(y_1)
\left( \frac{v_1}{v_1+y_1}\cdot \frac{v_1-\hat{\alpha}}{v_1- \hat{\alpha} - \hat{M}} \right)^{k_2}
\left( \frac{v_2}{v_2+y_2}\cdot \frac{v_2-\hat{\alpha}}{v_2- \hat{\alpha} - \hat{M}} \right)^{k_1}
\right]
dv_1 dv_2 
dy_1 dy_2
\\
- \int_0^1
\frac{\theta^{-1} g_1(y) g_2(y)}{2\pi\im}
\oint
\frac{k_1 k_2}{(k_1 + k_2)(v + y)^2} \left( \frac{v}{v+y}\cdot \frac{v-\hat{\alpha}}{v - \hat{\alpha} - \hat{M}}\right)^{k_1 + k_2} dv dy ,
\end{multline}
where in the first integral, the contours of $v_1, v_2$ enclose $-y_1, -y_2$, respectively, but not $\hat{\alpha} + \hat{M}$, and $|v_1| \ll |v_2|$;
and in the second integral, the contour of $v$ encloses $-y$ but not $\hat{\alpha} + \hat{M}$.
In slightly different notations this is precisely \eqref{eq:clt2:cov}.
\end{proof}

\section{Connecting the limit field with the Gaussian Free Field}  \label{sec:gff}

In this section we interpret Theorem \ref{thm:clt1} and \ref{thm:clt2} as convergence of the height functions (see Definition \ref{defn:hf}) towards a Gaussian random field.

In Section \ref{ssec:gff:la}, we give the proof of Theorem \ref{thm:gff:la}.
It is based on Theorem \ref{thm:clt1}, and computing the covariances of the 1--dimensional integrals of the pullback of the Gaussian Free Field.
In Section \ref{ssec:gff:sm}, we discuss the 2--dimensional integrals.
We first compute the covariances of the 2--dimensional integrals against the pullback of the Gaussian Free Field.
Then, with a bound on the covariances, we prove Lemma \ref{lemma:derk:defn}, and extend the definition of $\mathfrak{Z}_{g, k}$ to any $g\in L^2([0, 1])$.
Based on these and Theorem \ref{thm:clt2}, we finish the proof of Theorem \ref{thm:gff:sm}.


\subsection{Identification of the 1--dimensional integral: proof of Theorem \ref{thm:gff:la}}  \label{ssec:gff:la}

\begin{proof}[Proof of Theorem \ref{thm:gff:la}]
Denote $N_i = \left\lfloor L\hat{N}_i \right\rfloor$, for $i = 1, \cdots, h$, and $N'_i = \left\lfloor L\hat{N}'_i \right\rfloor$, for $i = 1, \cdots, h'$.
Through integration by parts, 
\eqref{eq:prop:gff:la1} and \eqref{eq:lemma:if:2} are respectively reduced to
\begin{equation}  \label{eq:thm:gff:la:pf1}
\left( - \frac{L^{\frac{1}{2}}}{k_i + 1}\left( \mathfrak{P}_{k_i}(x^{N_i}) - \mathfrak{P}_{k_i}(x^{N_i-1}) - \mathbb{E}\left(\mathfrak{P}_{k_i}(x^{N_i}) - \mathfrak{P}_{k_i}(x^{N_i-1})\right) \right) \right)_{i=1}^h
\end{equation}
and
\begin{equation}
\left( - \frac{1}{k'_i + 1}\left( \mathfrak{P}_{k'_i}(x^{N'_i}) - \mathbb{E}\left(\mathfrak{P}_{k'_i}(x^{N'_i}) \right) \right) \right)_{i=1}^{h'} .
\end{equation}
By Theorem \ref{thm:clt1} we conclude that, as $L \rightarrow \infty$, they (weakly) converge to Gaussian jointly, and the limit vectors are independent.
For $\hat{N}_i = \hat{N}_j$, the covariance of the $i$th and $j$th component in the $L \rightarrow \infty$ limit of \eqref{eq:thm:gff:la:pf1} is
\begin{equation}  \label{eq:prop:gff:pf:cov}
 - \frac{1}{k_i + k_j + 2} \cdot
\frac{\theta^{-1}}{2\pi\im }
\oint
\frac{1}{(v + \hat{N}_i)^2} \left( \frac{v}{v+\hat{N}_i}\cdot \frac{v-\hat{\alpha}}{v - \hat{\alpha} - \hat{M}}\right)^{k_i + k_j + 2} dv
\end{equation}
where
the contour encloses $-\hat{N}_i = -\hat{N}_j$
but not $\hat{\alpha} + \hat{M}$.
For $\hat{N}_i \neq \hat{N}_j$, the covariance of the $i$th and $j$th component is $0$.

Now let's turn to evaluate \eqref{eq:prop:gff:la2}.
By Lemma \ref{lemma:pbcov1}, for $\delta > 0$ the distribution of
\begin{equation}   \label{eq:thm:gff:la:pf2}
\delta^{-\frac{1}{2}} \left( \int_0^1 u^{k_i} \mathcal{K}(u, \hat{N}_i + \delta) du - \int_0^1 u^{k_i} \mathcal{K}(u, \hat{N}_i) du \right)_{i=1}^h  ,
\end{equation}
is also Gaussian, and the covariance of the $i$th and $j$th ($i < j$) component is as following: when $\hat{N}_i < \hat{N}_j$, for $\delta < \hat{N}_j - \hat{N}_i$ the covariance is
\begin{multline}  \label{eq:prop:gff:la:pf1}
\frac{\delta^{-1}}{(k_i + 1)(k_j + 1)} \cdot
\frac{\theta^{-1}}{(2\pi\im)^2 }
\oint \oint
\left(
\left( \frac{v_1}{v_1+\hat{N}_i+\delta}\cdot \frac{v_1-\hat{\alpha}}{v_1 - \hat{\alpha} - \hat{M}}\right)^{k_i + 1}
-
\left( \frac{v_1}{v_1+\hat{N}_i}\cdot \frac{v_1-\hat{\alpha}}{v_1 - \hat{\alpha} - \hat{M}}\right)^{k_i + 1}
\right) \\
\times
\left(
\left( \frac{v_2}{v_2+\hat{N}_j+\delta}\cdot \frac{v_2-\hat{\alpha}}{v_2 - \hat{\alpha} - \hat{M}}\right)^{k_j + 1}
-
\left( \frac{v_2}{v_2+\hat{N}_j}\cdot \frac{v_2-\hat{\alpha}}{v_2 - \hat{\alpha} - \hat{M}}\right)^{k_j + 1}
\right)
\frac{dv_1 dv_2}{(v_1 - v_2)^2},
\end{multline}
where $|v_1| \ll |v_2|$, and the contours enclose $-\hat{N}_i$, $-\hat{N}_j$, $-\hat{N}_i - \delta$, and $-\hat{N}_j - \delta$, but not $\hat{\alpha} + \hat{M}$;
when $\hat{N}_i = \hat{N}_j$ the covariance is
\begin{multline}  \label{eq:prop:gff:la:pf2}
\frac{\delta^{-1}}{(k_i + 1)(k_j + 1)} \cdot
\frac{\theta^{-1}}{(2\pi\im)^2}
\oint \oint
\left(
\left( \frac{v_1}{v_1+\hat{N}_i+\delta}\cdot \frac{v_1-\hat{\alpha}}{v_1 - \hat{\alpha} - \hat{M}}\right)^{k_i + 1}
\left( \frac{v_2}{v_2+\hat{N}_i+\delta}\cdot \frac{v_2-\hat{\alpha}}{v_2 - \hat{\alpha} - \hat{M}}\right)^{k_j + 1}
\right.
\\
-
\left( \frac{v_1}{v_1+\hat{N}_i}\cdot \frac{v_1-\hat{\alpha}}{v_1 - \hat{\alpha} - \hat{M}}\right)^{k_i + 1}
\left( \frac{v_2}{v_2+\hat{N}_i+\delta}\cdot \frac{v_2-\hat{\alpha}}{v_2 - \hat{\alpha} - \hat{M}}\right)^{k_j + 1}
\\
-
\left( \frac{v_1}{v_1+\hat{N}_i}\cdot \frac{v_1-\hat{\alpha}}{v_1 - \hat{\alpha} - \hat{M}}\right)^{k_j + 1}
\left( \frac{v_2}{v_2+\hat{N}_i+\delta}\cdot \frac{v_2-\hat{\alpha}}{v_2 - \hat{\alpha} - \hat{M}}\right)^{k_i + 1}
\\
+
\left.
\left( \frac{v_1}{v_1+\hat{N}_i}\cdot \frac{v_1-\hat{\alpha}}{v_1 - \hat{\alpha} - \hat{M}}\right)^{k_j + 1}
\left( \frac{v_2}{v_2+\hat{N}_i}\cdot \frac{v_2-\hat{\alpha}}{v_2 - \hat{\alpha} - \hat{M}}\right)^{k_i + 1}
\right) 
\frac{dv_1 dv_2}{(v_1 - v_2)^2},
\end{multline}
where $|v_1| \ll |v_2|$, and the contours enclose $-\hat{N}_i$ and $-\hat{N}_i - \delta$, but not $\hat{\alpha} + \hat{M}$.

Notice that by sending $\delta \rightarrow 0_+$,
\eqref{eq:thm:gff:la:pf2} converges to a Gaussian as well, because the covariances converge.
Indeed, when $\delta \rightarrow 0_+$, \eqref{eq:prop:gff:la:pf1} converges to zero,
and \eqref{eq:prop:gff:la:pf2} converges to
\begin{multline}  \label{eq:prop:gff:la:pf3}
- \frac{1}{k_j + 1} \cdot
\frac{\theta^{-1}}{(2\pi\im)^2}
\oint \oint
\left(
\frac{1}{v_1 + \hat{N}_i}
\left( \frac{v_1}{v_1+\hat{N}_i}\cdot \frac{v_1-\hat{\alpha}}{v_1 - \hat{\alpha} - \hat{M}}\right)^{k_i + 1}
\left( \frac{v_2}{v_2+\hat{N}_i}\cdot \frac{v_2-\hat{\alpha}}{v_2 - \hat{\alpha} - \hat{M}}\right)^{k_j + 1}
\right.
\\
-
\left.
\frac{1}{v_2 + \hat{N}_i}
\left( \frac{v_1}{v_1+\hat{N}_i}\cdot \frac{v_1-\hat{\alpha}}{v_1 - \hat{\alpha} - \hat{M}}\right)^{k_i + 1}
\left( \frac{v_2}{v_2+\hat{N}_i}\cdot \frac{v_2-\hat{\alpha}}{v_2 - \hat{\alpha} - \hat{M}}\right)^{k_j + 1}
\right)
\frac{dv_1 dv_2}{(v_1 - v_2)^2} 
\\
=
\frac{1}{k_j + 1} \cdot
\frac{\theta^{-1}}{(2\pi\im)^2}
\oint \oint
\frac{dv_1 dv_2}{(v_1 - v_2)}\cdot
\frac{1}{(v_1 + \hat{N}_i)(v_2 + \hat{N}_i)}
\\
\times
\left( \frac{v_1}{v_1+\hat{N}_i}\cdot \frac{v_1-\hat{\alpha}}{v_1 - \hat{\alpha} - \hat{M}}\right)^{k_i + 1}
\left( \frac{v_2}{v_2+\hat{N}_i}\cdot \frac{v_2-\hat{\alpha}}{v_2 - \hat{\alpha} - \hat{M}}\right)^{k_j + 1}  .
\end{multline}
We switch $k_i$ and $k_j$, and obtain an equivalent expression
\begin{multline}  \label{eq:prop:gff:la:pf4}
\frac{1}{k_i + 1} \cdot
\frac{\theta^{-1}}{(2\pi\im)^2}
\oint \oint
\frac{dv_1 dv_2}{(v_1 - v_2)}\cdot
\frac{1}{(v_1 + \hat{N}_i)(v_2 + \hat{N}_i)}
\\
\times
\left( \frac{v_1}{v_1+\hat{N}_i}\cdot \frac{v_1-\hat{\alpha}}{v_1 - \hat{\alpha} - \hat{M}}\right)^{k_j + 1}
\left( \frac{v_2}{v_2+\hat{N}_i}\cdot \frac{v_2-\hat{\alpha}}{v_2 - \hat{\alpha} - \hat{M}}\right)^{k_i + 1}  .
\end{multline}
We get \eqref{eq:prop:gff:pf:cov} by applying Theorem \ref{thm:dr} to $\frac{k_j + 1}{k_i + k_j + 2} \times$\eqref{eq:prop:gff:la:pf3} $+ \frac{k_i + 1}{k_i + k_j + 2}\times$ \eqref{eq:prop:gff:la:pf4}.
\end{proof}
\begin{rem} \label{rem:add3}
As the convergence is proved by applying Theorem \ref{thm:clt1}, from Remark \ref{rem:add2} we conclude that, for any $G \in \mathbb{R}_{>0}$, and the vector \eqref{eq:prop:gff:la1}, the convergence of all its cumulants of order at least three are uniform in $\hat{N}_1, \cdots, \hat{N}_h \in [0, G]$, and its covariances are bounded uniformly for $\hat{N}_1, \cdots, \hat{N}_h \in [0, G]$
\end{rem}


\subsection{Identification of the 2--dimensional integral: proof of Theorem \ref{thm:gff:sm}}  \label{ssec:gff:sm}

Now let us discuss the random variable $\mathfrak{Z}_{g, k}$ from Definition \ref{defn:pair}.
\begin{prop}  \label{prop:derk:smt}
Let $k_1, \cdots, k_h$ be integers, $G \in \mathbb{R}_{>0}$, and let
$g_1, \cdots, g_h \in C^{\infty}([0, G])$, satisfying $g_i(G) = 0$ for each $1 \leq i \leq h$.
The distribution of the vector
$\left( \mathfrak{Z}_{g_i, k_i} \right)_{i=1}^h$
is joint centered Gaussian, and the covariance between the $i$th and $j$th component is
\begin{multline}  \label{eq:prop:derk:smp:cov}
\iint_{0 \leq y_1 < y_2 \leq G}
\frac{\theta^{-1}}{(2\pi \im)^2 }
\oint\oint  \frac{1}{(v_1 - v_2)^2 (v_1 + y_1)(v_2 + y_2)}
\\
\times
\left[
g_i(y_2)g_j(y_1)
\left(\frac{v_1}{v_1 + y_1} \cdot \frac{v_1 - \hat{\alpha}}{v_1 - \hat{\alpha} - \hat{M}}\right)^{k_j + 1} \left(\frac{v_2}{v_2 + y_2} \cdot \frac{v_2 - \hat{\alpha}}{v_2 - \hat{\alpha} - \hat{M}}\right)^{k_i + 1}
\right.
\\
+
\left.
g_i(y_1)g_j(y_2)
\left(\frac{v_1}{v_1 + y_1} \cdot \frac{v_1 - \hat{\alpha}}{v_1 - \hat{\alpha} - \hat{M}}\right)^{k_i + 1} \left(\frac{v_2}{v_2 + y_2} \cdot \frac{v_2 - \hat{\alpha}}{v_2 - \hat{\alpha} - \hat{M}}\right)^{k_j + 1} \right] dv_1 dv_2 dy_1 dy_2
\\
-
\int_0^1 \frac{g_i(y)g_j(y)\theta^{-1}}{2\pi \im (k_i + k_j + 2)}
\oint \frac{1}{(v + y)^2}
\left(\frac{v}{v + y} \cdot \frac{v - \hat{\alpha}}{v - \hat{\alpha} - \hat{M}}\right)^{k_i + k_j + 2} dv dy,
\end{multline}
where for the first two summands, the contours enclose poles at $-y_1$ and $-y_2$,
but not $\hat{\alpha} + \hat{M}$, and are nested with $v_2$ larger;
for the last summand, the contour encloses pole at $-y$ but not $\hat{\alpha} + \hat{M}$.
\end{prop}

\begin{proof}
By Lemma \ref{lemma:pbcov2}, the vector $\left( \mathfrak{Z}_{g_i, k_i} \right)_{i=1}^h$
is centered Gaussian, and the covariance between the $i$th and $j$th component is
\begin{multline}  \label{prop:derk:smt:pf0}
\int_0^G \int_0^G \frac{\left(\frac{d}{dy}{g}_{i}(y_1)\right) \left(\frac{d}{dy}{g}_{j}(y_2)\right) \theta^{-1}}{(2\pi \im)^2 (k_i + 1)(k_j + 1)} \oint \oint \frac{1}{(v_1 - v_2)^2} \\
\times \left(\frac{v_1}{v_1 + y_1} \cdot \frac{v_1 - \hat{\alpha}}{v_1 - \hat{\alpha} - \hat{M}}\right)^{k_i + 1} \left(\frac{v_2}{v_2 + y_2} \cdot \frac{v_2 - \hat{\alpha}}{v_2 - \hat{\alpha} - \hat{M}}\right)^{k_j + 1} dv_1 dv_2 dy_1 dy_2 ,
\end{multline}
where the contours enclose poles at $-y_1$ and $-y_2$, but not $\hat{\alpha} + \hat{M}$,
and are nested: when $y_1 \leq y_2$, $v_2$ is larger; when $y_1 \geq y_2$, $v_1$ is larger.

We fix the contours when $y_1 \leq y_2$ and $y_2 < y_1$, respectively, and switch the order of integrals:
\begin{multline}  \label{prop:derk:smt:pf1}
\oint \oint \iint_{0\leq y_1 \leq y_2 \leq G} \frac{ \left(\frac{d}{dy}{g}_{i}(y_1)\right) \left(\frac{d}{dy}{g}_{j}(y_2)\right)\theta^{-1}}{(2\pi \im)^2 (k_i + 1)(k_j + 1)} \frac{1}{(v_1 - v_2)^2} \\
\times \left(\frac{v_1}{v_1 + y_1} \cdot \frac{v_1 - \hat{\alpha}}{v_1 - \hat{\alpha} - \hat{M}}\right)^{k_i + 1} \left(\frac{v_2}{v_2 + y_2} \cdot \frac{v_2 - \hat{\alpha}}{v_2 - \hat{\alpha} - \hat{M}}\right)^{k_j + 1}  dy_1 dy_2dv_1 dv_2 \\
+
\oint \oint \iint_{0\leq y_2 < y_1 \leq G} \frac{\left(\frac{d}{dy} {g}_{i}(y_1)\right) \left(\frac{d}{dy}{g}_{j}(y_2)\right) \theta^{-1}}{(2\pi \im)^2 (k_i + 1)(k_j + 1)} \frac{1}{(v_1 - v_2)^2} \\
\times \left(\frac{v_1}{v_1 + y_1} \cdot \frac{v_1 - \hat{\alpha}}{v_1 - \hat{\alpha} - \hat{M}}\right)^{k_i + 1} \left(\frac{v_2}{v_2 + y_2} \cdot \frac{v_2 - \hat{\alpha}}{v_2 - \hat{\alpha} - \hat{M}}\right)^{k_j + 1}  dy_1 dy_2dv_1 dv_2 ,
\end{multline}
where in the first summand, $|v_1| \ll |v_2|$; and in the second summand, $|v_2| \ll |v_1|$.

We then integrate by parts for $y_1$.
The first summand in (\ref{prop:derk:smt:pf1}) has boundary terms at $0$ and $y_2$, while the second summand has boundary terms at $y_2$ and $G$.
The boundary term at $0$ vanishes since the contour of $v_1$ encloses no pole when $y_1 = 0$;
and the boundary term at $G$ vanishes since $g_i(G) = 0$.

Let us show that the boundary terms at $y_2$ in the two summands cancel out.
Indeed, each of them is an integral of $v_1$, $v_2$, $y_2$,
of the same expression.
The only difference is the nesting of the contours.
We fix the contour of $v_2$ in these two terms, and integrate $v_1$ along two circles, one inside and another one outside $v_2$, with different orientations.
Then we only need to compute the residue of $v_1$ at $v_2$.
The result is an integral of $v_2$ and $y_2$, and equals zero, because the integrand is the $v_2$ derivative of another function.
Then we conclude that \eqref{prop:derk:smt:pf1} equals
\begin{multline}  \label{prop:derk:smt:pf2}
\oint\oint \iint_{0 \leq y_1 \leq y_2 \leq G} \frac{g_i(y_1)\left(\frac{d}{dy}{g}_{j}(y_2)\right)\theta^{-1}}{(2\pi \im)^2 (k_j + 1)} \cdot \frac{1}{(v_1 - v_2)^2(v_1 + y_1)}
\\
\times
\left(\frac{v_1}{v_1 + y_1} \cdot \frac{v_1 - \hat{\alpha}}{v_1 - \hat{\alpha} - \hat{M}}\right)^{k_i + 1} \left(\frac{v_2}{v_2 + y_2} \cdot \frac{v_2 - \hat{\alpha}}{v_2 - \hat{\alpha} - \hat{M}}\right)^{k_j + 1} dy_1 dy_2  dv_1 dv_2
\\
+
\oint\oint \iint_{0 \leq y_2 < y_1 \leq G} \frac{g_i(y_1)\left(\frac{d}{dy}{g}_{j}(y_2)\right)\theta^{-1}}{(2\pi \im)^2 (k_j + 1)} \cdot \frac{1}{(v_1 - v_2)^2(v_1 + y_1)}
\\
\times
\left(\frac{v_1}{v_1 + y_1} \cdot \frac{v_1 - \hat{\alpha}}{v_1 - \hat{\alpha} - \hat{M}}\right)^{k_i + 1} \left(\frac{v_2}{v_2 + y_2} \cdot \frac{v_2 - \hat{\alpha}}{v_2 - \hat{\alpha} - \hat{M}}\right)^{k_j + 1} dy_1 dy_2  dv_1 dv_2
\end{multline}
where the contours are nested: in the first summand $|v_1| \ll |v_2|$, and in the second summand $|v_2| \ll |v_1|$.
Then we integrate by parts for $y_2$, for each of the two summands in \eqref{prop:derk:smt:pf2}.
For the second summand, we further exchange $v_1$ and $v_2$, $y_1$ and $y_2$.
In the end, we conclude that
\eqref{prop:derk:smt:pf0} $= A + B$, where
\begin{multline}
A =
\oint\oint \iint_{0 \leq y_1 \leq y_2 \leq G}
\frac{g_i(y_2)g_j(y_1)\theta^{-1}}{(2\pi \im)^2 } \cdot \frac{1}{(v_1 - v_2)^2 (v_1 + y_1)(v_2 + y_2)}
\\
\times
\left(\frac{v_1}{v_1 + y_1} \cdot \frac{v_1 - \hat{\alpha}}{v_1 - \hat{\alpha} - \hat{M}}\right)^{k_j + 1} \left(\frac{v_2}{v_2 + y_2} \cdot \frac{v_2 - \hat{\alpha}}{v_2 - \hat{\alpha} - \hat{M}}\right)^{k_i + 1} dy_1 dy_2 dv_1 dv_2
\\
+
\oint\oint \iint_{0 \leq y_1 \leq y_2 \leq G}
\frac{g_i(y_1)g_j(y_2)\theta^{-1}}{(2\pi \im)^2 } \cdot \frac{1}{(v_1 - v_2)^2 (v_1 + y_1)(v_2 + y_2)}
\\
\times
\left(\frac{v_1}{v_1 + y_1} \cdot \frac{v_1 - \hat{\alpha}}{v_1 - \hat{\alpha} - \hat{M}}\right)^{k_i + 1} \left(\frac{v_2}{v_2 + y_2} \cdot \frac{v_2 - \hat{\alpha}}{v_2 - \hat{\alpha} - \hat{M}}\right)^{k_j + 1} dy_1 dy_2 dv_1 dv_2 ,
\end{multline}
and
\begin{multline}
B =
\oint\oint \int_0^G \frac{g_i(y)g_j(y)\theta^{-1}}{(2\pi \im)^2 (k_j + 1)} \cdot \frac{1}{(v_1 - v_2)^2(v_2 + y)}
\\
\times
\left(\frac{v_1}{v_1 + y} \cdot \frac{v_1 - \hat{\alpha}}{v_1 - \hat{\alpha} - \hat{M}}\right)^{k_j + 1} \left(\frac{v_2}{v_2 + y} \cdot \frac{v_2 - \hat{\alpha}}{v_2 - \hat{\alpha} - \hat{M}}\right)^{k_i + 1} dy  dv_1 dv_2
\\
-
\oint\oint \int_0^G \frac{g_i(y)g_j(y)\theta^{-1}}{(2\pi \im)^2 (k_j + 1)} \cdot \frac{1}{(v_1 - v_2)^2(v_1 + y)}
\\
\times
\left(\frac{v_1}{v_1 + y} \cdot \frac{v_1 - \hat{\alpha}}{v_1 - \hat{\alpha} - \hat{M}}\right)^{k_i + 1} \left(\frac{v_2}{v_2 + y} \cdot \frac{v_2 - \hat{\alpha}}{v_2 - \hat{\alpha} - \hat{M}}\right)^{k_j + 1} dy  dv_1 dv_2 ,
\end{multline}
where the contours in both $A$ and $B$ are nested and $|v_1| \ll |v_2|$.
Note that $A$ equals the first and second summands in \eqref{eq:prop:derk:smp:cov}.

By symmetry, if we interchange $i$ and $j$ in $A$ and $B$ to get $A'$ and $B'$, we would have that \eqref{prop:derk:smt:pf0} $ = A' + B'$.
Note that $A = A'$, and
\begin{multline}
B' =
-
\oint\oint \int_0^G \frac{g_i(y)g_j(y)\theta^{-1}}{(2\pi \im)^2 (k_i + 1)} \cdot \frac{1}{(v_1 - v_2)^2(v_1 + y)}
\\
\times
\left(\frac{v_1}{v_1 + y} \cdot \frac{v_1 - \hat{\alpha}}{v_1 - \hat{\alpha} - \hat{M}}\right)^{k_j + 1} \left(\frac{v_2}{v_2 + y} \cdot \frac{v_2 - \hat{\alpha}}{v_2 - \hat{\alpha} - \hat{M}}\right)^{k_i + 1} dy  dv_1 dv_2
\\
+
\oint\oint \int_0^G \frac{g_i(y)g_j(y)\theta^{-1}}{(2\pi \im)^2 (k_i + 1)} \cdot \frac{1}{(v_1 - v_2)^2(v_2 + y)}
\\
\times
\left(\frac{v_1}{v_1 + y} \cdot \frac{v_1 - \hat{\alpha}}{v_1 - \hat{\alpha} - \hat{M}}\right)^{k_i + 1} \left(\frac{v_2}{v_2 + y} \cdot \frac{v_2 - \hat{\alpha}}{v_2 - \hat{\alpha} - \hat{M}}\right)^{k_j + 1} dy  dv_1 dv_2 ,
\end{multline}
where the contours are also nested and $|v_1| \ll |v_2|$.
Then we have
\begin{multline}  \label{eq:last:exp}
B = B' =
\frac{(k_j + 1)B + (k_i + 1)B'}{k_i + k_j + 2}
\\
=
\oint\oint \int_0^G \frac{g_i(y)g_j(y)\theta^{-1}}{(2\pi \im)^2 (k_i + k_j + 2)} \cdot \frac{1}{(v_1 - v_2)(v_1 + y)(v_2 + y)}
\\
\times
\left(
\left(\frac{v_1}{v_1 + y} \cdot \frac{v_1 - \hat{\alpha}}{v_1 - \hat{\alpha} - \hat{M}}\right)^{k_j + 1} \left(\frac{v_2}{v_2 + y} \cdot \frac{v_2 - \hat{\alpha}}{v_2 - \hat{\alpha} - \hat{M}}\right)^{k_i + 1}
\right.
\\
+
\left.
\left(\frac{v_1}{v_1 + y} \cdot \frac{v_1 - \hat{\alpha}}{v_1 - \hat{\alpha} - \hat{M}}\right)^{k_i + 1} \left(\frac{v_2}{v_2 + y} \cdot \frac{v_2 - \hat{\alpha}}{v_2 - \hat{\alpha} - \hat{M}}\right)^{k_j + 1}
\right)
dy  dv_1 dv_2 ,
\end{multline}
where the contours are nested and $|v_1| \ll |v_2|$.
By applying Theorem \ref{thm:dr} to \eqref{eq:last:exp}, we get the last summand in \eqref{eq:prop:derk:smp:cov}.
\end{proof}

Using this proposition we can bound the covariance uniformly:
\begin{cor}   \label{cor:l2bd}
There is a constant $C(\hat{\alpha}, \hat{M}, k, G)$, such that for any $g_1, g_2 \in C^{\infty}([0, G])$, we have
\begin{equation}
\mathbb{E}\left( \mathfrak{Z}_{g_1, k} \mathfrak{Z}_{g_2, k} \right) \leq C(\hat{\alpha}, \hat{M}, k, G) \|g_1\|_{L^{2}} \|g_2\|_{L^{2}} .
\end{equation}
\end{cor}
\begin{proof}
By Proposition \ref{prop:derk:smt}, since the random variables $\mathfrak{Z}_{g_1, k}$ and $\mathfrak{Z}_{g_2, k}$ are centered Gaussian, $\mathbb{E}\left( \mathfrak{Z}_{g_1, k} \mathfrak{Z}_{g_2, k} \right)$ is just the covariance given by (\ref{eq:prop:derk:smp:cov}).
We fix the contours in (\ref{eq:prop:derk:smp:cov}) to enclose line segment $[- G, 0]$ but not $\hat{\alpha} + \hat{M}$; then it is bounded by
\begin{multline}
C \left( \int_{0}^G \int_0^G  \left| g_1(y_1) g_2(y_2) \right| dy_1 dy_2
+
\int_0^G \left| g_1(y) g_2(y) \right|  dy
\right)
\\
\leq C \left( \left\| g_1 \right\|_{L^1} \left\| g_2 \right\|_{L^1} + \left\| g_1 \right\|_{L^2}\left\| g_2 \right\|_{L^2} \right)
\leq 2C \left\| g_1 \right\|_{L^2}\left\| g_2 \right\|_{L^2},
\end{multline}
for some constant $C$ independent of $g_1, g_2$.
Setting $C(\hat{\alpha}, \hat{M}, k, G) = 2C$ finishes the proof.
\end{proof}

Now we show that $\mathfrak{Z}_{g, k}$ can be defined for any $G \in \mathbb{R}_{>0}$ and $g\in L^2([0, G])$.

\begin{proof}[Proof of Lemma \ref{lemma:derk:defn}]
Since smooth functions are dense in $L^2([0, G])$, there is a sequence $h_1, h_2, \cdots \in C^{\infty}([0, G])$ which converges to $g$ in $L^2([0, G])$.
We further take a sequence $\lambda_1, \lambda_2, \cdots \in C^{\infty}([0, G])$, where each $0 \leq \lambda_n \leq 1$, $\lambda_n(G) = 1$, and each $\| \lambda_n \|_{L^2} < 2^{-n} \| h_n \|_{L^{\infty}}^{-1}$.
Set $g_n = (1 - \lambda_n)h_n$, then each $g_n \in C^{\infty}([0, G])$ satisfies $g_n(G) = 0$, and $\| g_n - g \|_{L^2} \leq \| h_n - g \|_{L^2} + \|h_n\|_{L^{\infty}} \|\lambda_n \|_{L^{2}}$.
Then the sequence $g_1, g_2, \cdots$ converges to $g$ in $L^2$.

Passing to a subsequence if necessary, we can assume that for any $n$ we have
$\| g_{n} - g_{n+1} \|_{L^2} < 2^{-n}$.
By Corollary \ref{cor:l2bd}, we have
\begin{multline}
\mathbb{E}\left(  \left| \mathfrak{Z}_{g_n, k} - \mathfrak{Z}_{g_{n+1}, k} \right| \right)
\leq
\mathbb{E}\left( \left| \mathfrak{Z}_{g_n, k} - \mathfrak{Z}_{g_{n+1}, k} \right|^2 \right)^{\frac{1}{2}}
=
\mathbb{E}\left( \left| \mathfrak{Z}_{g_n - g_{n+1}, k} \right|^2 \right)^{\frac{1}{2}}
\\
\leq
C(\hat{\alpha}, \hat{M}, k, G)^{\frac{1}{2}} \| g_{n} - g_{n+1} \|_{L^{2}}
\leq 2^{-n} C(\hat{\alpha}, \hat{M}, k, G)^{\frac{1}{2}}  .
\end{multline}
Then by the dominated convergence theorem, the limit
\begin{equation}
\lim_{m\rightarrow \infty} \mathfrak{Z}_{g_m, k}
=
\lim_{m \rightarrow \infty} \sum_{n=0}^{m - 1} \left( \mathfrak{Z}_{g_{n+1}, k} - \mathfrak{Z}_{g_n, k}\right), \mathrm{\; where\;} \mathfrak{Z}_{g_0, k} = 0,
\end{equation}
exists almost surely.
Denote it as $\mathfrak{Z}_{g, k}$.

For the uniqueness of $\mathfrak{Z}_{g, k}$, if there is another such sequence $\tilde{g}_1, \tilde{g}_2, \cdots$, then we have
\begin{multline}
\mathbb{E} \left( \left| \mathfrak{Z}_{g, k} - \mathfrak{Z}_{\tilde{g}_n, k} \right| \right)
\leq
\mathbb{E} \left( \left| \mathfrak{Z}_{g, k} - \mathfrak{Z}_{{g}_n, k} \right| \right)
+
\mathbb{E} \left( \left| \mathfrak{Z}_{{g}_n, k} - \mathfrak{Z}_{\tilde{g}_n, k} \right| \right)
\\
\leq
\mathbb{E} \left( \left| \mathfrak{Z}_{g, k} - \mathfrak{Z}_{{g}_n, k} \right| \right)
+
C(\hat{\alpha}, \hat{M}, k, G)^{\frac{1}{2}} \| g_n - \tilde{g}_n \|_{L^{2}},
\end{multline}
which goes to $0$ as $n \rightarrow \infty$.
Then $\mathfrak{Z}_{\tilde{g}_n, k}$ also converges almost surely to $\mathfrak{Z}_{g, k}$.
\end{proof}

With this we can extend Proposition \ref{prop:derk:smt} to functions in $L^2([0, 1])$.
\begin{prop}  \label{prop:derk:smte}
Let $k_1, \cdots, k_h$ be integers, $G \in \mathbb{R}_{>0}$, and
$g_1, \cdots, g_h \in L^2([0, G])$.
Then the joint distribution of the vector
$\left( \mathfrak{Z}_{g_i, k_i} \right)_{i=1}^h$
is Gaussian, and the covariance between the $i$th and $j$th component is given by the same expression
(\ref{eq:prop:derk:smp:cov}).
\end{prop}

\begin{proof}
For each $1 \leq i \leq h$, take a sequence $g_{1,i}, g_{2,i}, \cdots$ such that each is in $C^{\infty}([0, G])$, and $g_{n, i}(G) = 0$ for each positive integer $n$, and $g_{n, i}$ converges to $g_i$ in $L^2([0, G])$.
The joint distribution of $\left( \mathfrak{Z}_{g_i, k_i} \right)_{i=1}^h$
is the limit
\begin{equation}
\lim_{n \rightarrow \infty}
\left( \mathfrak{Z}_{g_{n, i}, k_i} \right)_{i=1}^h ,
\end{equation}
in the sense that this vector almost surely converges.
By Proposition \ref{prop:derk:smt}, for each $n$, $\left( \mathfrak{Z}_{g_{n, i}, k_i} \right)_{i=1}^h$ is jointly Gaussian, then so is $\left( \mathfrak{Z}_{g_i, k_i} \right)_{i=1}^h$;
and the covariances are given by taking the $n \rightarrow \infty$ limit of the corresponding covariances.
\end{proof}

Finally we finish the proof of Theorem \ref{thm:gff:sm}.

\begin{proof}[Proof of Theorem \ref{thm:gff:sm}]
Integrating by parts in the $u$--direction, we obtain
\begin{multline}
\int_0^G\int_0^1 u^{k_i} g_i(y) \mathcal{W}_L(u, \hat{N}_i) du dy
=
L\int_0^G \frac{g_i(y)}{k_i + 1} \left( \min\{ N_i, M \} - \min\{ N_i - 1, M \} \right) dy
\\
- L\int_0^G \frac{g_i(y)}{k_i + 1} \left( \sum_{j=1}^{\min\{ N_i, M \}} \left(x_j^{N_i} \right)^{k_i + 1}
 - \sum_{j=1}^{\min\{ N_i-1, M \}} \left(x_j^{N_i} \right)^{k_i + 1}
 \right) dy.
\end{multline}
Thus, (\ref{eq:prop:gff:sm1}) equals
\begin{equation}
\left( - 
\int_0^G \frac{g_i(y)}{k_i + 1}
\left( \mathfrak{P}_{k_i+1}(x^{N_i}) - \mathfrak{P}_{k_i+1}(x^{N_i-1}) - \mathbb{E}\left(\mathfrak{P}_{k_i+1}(x^{N_i}) - \mathfrak{P}_{k_i+1}(x^{N_i-1})\right) \right)
 dy
\right)_{i=1}^h,
\end{equation}
which, by Theorem \ref{thm:clt2}, is asymptotically Gaussian, and the covariances are given by (\ref{eq:prop:derk:smp:cov}).

On the other hand, by Proposition \ref{prop:derk:smte}, the joint distribution of $\left( \mathfrak{Z}_{g_i, k_i} \right)_{i=1}^{h}$ is also Gaussian, with covariances also given by (\ref{eq:prop:derk:smp:cov}).

For \eqref{eq:prop:gff:sm2}, it suffices to show that for each $1 \leq i \leq h'$ the random variable
\begin{equation}  \label{eq:thm:gff:sm:pf1}
\int_0^G\int_0^1 \Big[ u^{k_i'} \left(\frac{d}{dy}\tilde{g}_{i}(y)\right) \left( \mathcal{H}_L(u, y) - \mathbb{E} \left( \mathcal{H}_L(u, y) \right) \right)
+
u^{k_i'} \tilde{g}_i(y) \left( \mathcal{W}_L(u, y) - \mathbb{E} \left( \mathcal{W}_L(u, y) \right) \right) \Big]
du dy
\end{equation}
weakly converges to $0$ as $L \rightarrow \infty$.
Here we can also assume that $\tilde{g}_i(G) = 0$ for each $1 \leq i \leq h'$, since adding a constant to each $\tilde{g}_i$ does not influence \eqref{eq:thm:gff:sm:pf1}.

Take $L_0$ be the largest integer strictly less than $GL$.
Note that $\mathcal{H}_L(u, y) - \mathbb{E} \left( \mathcal{H}_L(u, y) \right)$ is a piecewise constant function for $y$:
specifically, fixing $u$, it is constant for $y \in \left[\frac{n-1}{L}, \frac{n}{L} \right)$ for any positive integer $n$.
By integrating in the $y$--direction we have
\begin{multline}
\int_0^G \int_0^1 u^{k_i'} \left(\frac{d}{dy}\tilde{g}_{i}(y)\right) \left( \mathcal{H}_L(u, y) - \mathbb{E} \left( \mathcal{H}_L(u, y) \right) \right) du dy
\\
=
- \int_0^1 u^{k_i'} L^{-1}\sum_{n = 1}^{L_0} \tilde{g}_i\left( \frac{n}{L} \right) \left[ \mathcal{W}_L\left(u, \frac{n}{L}\right) - \mathbb{E} \left( \mathcal{W}_L\left(u, \frac{n}{L}\right) \right) \right] du .
\end{multline}
This implies that the absolute value of (\ref{eq:thm:gff:sm:pf1}) is bounded by
\begin{multline}   \label{eq:thm:gff:sm:pf2}
\sup_{a, b \in [0, G], |a - b| \leq L^{-1}} |\tilde{g}_i(a) - \tilde{g}_i(b)| \int_0^G \left| \int_0^1 u^{k_i'}  \mathcal{W}_L(u, y) - \mathbb{E} \left( \mathcal{W}_L(u, y) \right) du \right| dy
\\
\leq
\left\| \frac{d}{dy}\tilde{g}_{i} \right\|_{L^{\infty}}
\int_0^G \left|L^{-1} \int_0^1 u^{k_i'} \mathcal{W}_L(u, y) - \mathbb{E} \left( \mathcal{W}_L(u, y) \right) du \right| dy  .
\end{multline}
By Theorem \ref{thm:gff:la}, as $L \rightarrow \infty$, the integrand of the outer integral weakly converges to zero, and this is uniform in $y$ according to Remark \ref{rem:add3}.
Then \eqref{eq:thm:gff:sm:pf2}, and
\eqref{eq:thm:gff:sm:pf1}, weakly converge to $0$ as $L \rightarrow \infty$.
\end{proof}

\appendix
\section*{Appendices}
\addcontentsline{toc}{section}{Appendices}
\renewcommand{\thesubsection}{\Alph{subsection}}

\numberwithin{equation}{subsection}
\numberwithin{theorem}{subsection}
\subsection{Dimension reduction identities}  \label{sec:dr}

In this appendix we discuss integral identities, which are widely used in proofs in the main text.
A special case ($m=1$) of the following result was communicated to the authors by Alexei Borodin, and we present our own proof here.

For any positive integer $n$, let $\sigma_n$ denote the cycle $(12\cdots n)$, and let $S^{cyc}(n)$ denote the $n$--element subgroup of the symmetric group spanned by $\sigma_n$.
\begin{theorem} \label{thm:dr}
Let $n\geq 2$, and $f_1, \cdots, f_n : \mathbb{C} \rightarrow \mathbb{C}$ be meromorphic with
possible poles at $\{\mathfrak{p}_1, \cdots, \mathfrak{p}_m \}$.
Then we have the identity
\begin{equation} \label{eq:thm:dr}
\sum_{\sigma \in S^{cyc}(n)} \frac{1}{(2\pi\im)^n} \oint \cdots \oint \frac{f_{\sigma(1)}(u_{1}) \cdots f_{\sigma(n)}(u_{n}) }{(u_2 - u_1) \cdots (u_n - u_{n-1})} du_1\cdots du_n
= \frac{1}{2\pi\im} \oint f_1(u)\cdots f_n(u) du ,
\end{equation}
where the contours in both sides are positively oriented, enclosing $\{\mathfrak{p}_1, \cdots, \mathfrak{p}_m \}$,
and for the left hand side we require $|u_1| \ll \cdots \ll |u_n|$.
\end{theorem}

\begin{proof}
Let $\mathfrak{C}_1, \cdots, \mathfrak{C}_{2n-1}$
be closed paths around $\{\mathfrak{p}_1, \cdots, \mathfrak{p}_m \}$,
and each $\mathfrak{C}_i$ is inside $\mathfrak{C}_{i+1}$,
$1\leq i \leq 2n-2$.
Also, to simplify notations,
set $f_{n+t} = f_t$ and $u_{n+t} = u_t$ for any $1 \leq t \leq n-1$.
Then the left hand side of (\ref{eq:thm:dr}) can be written as
\begin{equation}
\sum_{t=0}^{n-1}
\frac{1}{(2\pi \im)^n}
\oint_{\mathfrak{C}_{1}} \cdots \oint_{\mathfrak{C}_{n}} \frac{f_{1+t}(u_1) \cdots f_{n+t}(u_n) }{(u_2 - u_1) \cdots (u_n- u_{n-1})}
du_n \cdots du_1  .
\end{equation}

When $n=2$,
we have
\begin{multline}  \label{eq:thm:dr:pf1}
\frac{1}{(2\pi \im)^2}
\oint_{\mathfrak{C}_{1}} \oint_{\mathfrak{C}_{2}} \frac{f_1(u_1) f_2(u_2)}{u_2 - u_1} du_2 du_1
+
\frac{1}{(2\pi \im)^2}
\oint_{\mathfrak{C}_{1}} \oint_{\mathfrak{C}_{2}} \frac{f_2(u_1) f_1(u_2)}{u_2 - u_1} du_2 du_1
\\
=
\frac{1}{(2\pi \im)^2}
\oint_{\mathfrak{C}_{1}} \oint_{\mathfrak{C}_{2}} \frac{f_1(u_1) f_2(u_2)}{u_2 - u_1} du_2 du_1
+
\frac{1}{(2\pi \im)^2}
\oint_{\mathfrak{C}_{3}} \oint_{\mathfrak{C}_{2}} \frac{f_1(u_1) f_2(u_2)}{u_1 - u_2} du_2 du_1
\\
=
\frac{1}{(2\pi \im)^2}
\oint_{\mathfrak{C}_{3} - \mathfrak{C}_{1}} \oint_{\mathfrak{C}_{2}} \frac{f_1(u_1) f_2(u_2)}{u_1 - u_2} du_2 du_1
 ,
\end{multline}
where $\oint_{\mathfrak{C}_{3} - \mathfrak{C}_{1}}$ is a notation for the difference of integrals over $\mathfrak{C}_3$ and $\mathfrak{C}_1$.
Further, (\ref{eq:thm:dr:pf1}) equals to
\begin{equation}
\frac{1}{2\pi \im}
\oint_{\mathfrak{C}_{2}} f_1(u) f_2(u) du ,
\end{equation}
since as a function of $u_1$,
$\frac{f_1(u_1) f_2(u_2)}{u_1 - u_2} $ has a single pole at $u_2$
between $\mathfrak{C}_{3}$ and $\mathfrak{C}_{1}$;
and the residue at this pole equals
$f_1(u_2)f_2(u_2)$.
This proves the case of $n=2$.

When $n\geq 3$, we argue by induction and assume that Theorem \ref{thm:dr} is true for $n-1$.
For any $1 \leq t \leq n-1$,
we have that
\begin{multline}
\frac{1}{(2\pi \im)^n}
\oint_{\mathfrak{C}_{1}} \cdots \oint_{\mathfrak{C}_{n}} \frac{f_{1+t}(u_1) \cdots f_{n+t}(u_n) }{(u_2 - u_1) \cdots (u_n- u_{n-1})}
du_n \cdots du_1 \\
= \frac{1}{(2\pi \im)^n}
\oint_{\mathfrak{C}_{1+t}} \cdots \oint_{\mathfrak{C}_{n+t}} \frac{f_{1}(u_1) \cdots f_{n}(u_n) }{(u_{2+t} - u_{1+t}) \cdots (u_{n+t}- u_{n-1+t})}
du_{n+t} \cdots du_{1+t} \\
= \frac{1}{(2\pi \im)^n}
\oint_{\mathfrak{C}_{n+1}} \cdots \oint_{\mathfrak{C}_{n+t}} \oint_{\mathfrak{C}_{t+1}} \cdots \oint_{\mathfrak{C}_{n}}  \frac{f_{1}(u_1) \cdots f_{n}(u_n) }{(u_{2+t} - u_{1+t}) \cdots (u_{n+t}- u_{n-1+t})}
du_{n} \cdots du_{1}  .
\end{multline}

Now we can move the contours of $u_1, \cdots, u_t$ from $\mathfrak{C}_{n+1}, \cdots, \mathfrak{C}_{n+t}$
to $\mathfrak{C}_{1}, \cdots, \mathfrak{C}_{t}$,
respectively.
We move the contours one by one starting from $u_1$,
and each move is across $\mathfrak{C}_{t+1}, \cdots, \mathfrak{C}_{n}$.
For $u_1 (= u_{n+1})$, the only pole between $\mathfrak{C}_{n+1}$ and $\mathfrak{C}_{1}$
is $u_n$;
for any $u_i$, $1 < i \leq t$,
there is no pole between $\mathfrak{C}_{n+i}$ and $\mathfrak{C}_{i}$.
Thus we have that
\begin{multline}
\frac{1}{(2\pi \im)^n}
\oint_{\mathfrak{C}_{n+1}} \cdots \oint_{\mathfrak{C}_{n+t}} \oint_{\mathfrak{C}_{t+1}} \cdots \oint_{\mathfrak{C}_{n}}  \frac{f_{1}(u_1) \cdots f_{n}(u_n) }{(u_{2+t} - u_{1+t}) \cdots (u_{n+t}- u_{n-1+t})}
du_{n} \cdots du_{1}  \\
=
\frac{1}{(2\pi \im)^n}
\oint_{\mathfrak{C}_{1}} \cdots \oint_{\mathfrak{C}_{n}}  \frac{f_{1}(u_1) \cdots f_{n}(u_n) }{(u_{2+t} - u_{1+t}) \cdots (u_{n+t}- u_{n-1+t})}
du_{n} \cdots du_{1} \\
+
\frac{1}{(2\pi \im)^{n-1}}
\oint_{\mathfrak{C}_{1}} \cdots \oint_{\mathfrak{C}_{n-1}} \frac{f_{1+t}(u_1) \cdots f_{n}(u_{n-t}) f_{n+1}(u_{n-t}) \cdots f_{n+t}(u_{n-1}) }{(u_2 - u_1) \cdots (u_{n-1}- u_{n-2})}
du_{n-1} \cdots du_1  .
\end{multline}
Notice that (taking into account that $u_{n+t} = u_t$)
\begin{equation}
\sum_{t=0}^{n-1} \frac{f_{1}(u_1) \cdots f_{n}(u_n) }{(u_{2+t} - u_{1+t}) \cdots (u_{n+t}- u_{n-1+t})}
= 0  ,
\end{equation}
and by the induction assumption (applied to $f_2, \cdots, f_{n-1}, f_n f_1$), we have that
\begin{multline}
\sum_{t=0}^{n-1}
\frac{1}{(2\pi \im)^n}
\oint_{\mathfrak{C}_{1}} \cdots \oint_{\mathfrak{C}_{n}} \frac{f_{1+t}(u_1) \cdots f_{n+t}(u_n) }{(u_2 - u_1) \cdots (u_n- u_{n-1})}
du_n \cdots du_1  \\
= \sum_{t=1}^{n-1}
\frac{1}{(2\pi \im)^{n-1}}
\oint_{\mathfrak{C}_{1}} \cdots \oint_{\mathfrak{C}_{n-1}} \frac{f_{1+t}(u_1) \cdots f_{n}(u_{n-t}) f_{n+1}(u_{n-t}) \cdots f_{n+t}(u_{n-1}) }{(u_2 - u_1) \cdots (u_{n-1}- u_{n-2})}
du_{n-1} \cdots du_1  \\
= \frac{1}{2\pi\im} \oint f_1(u)\cdots f_n(u) du  .
\end{multline}
\end{proof}

\begin{cor} \label{cor:dr}
Let $s$ be a positive integer.
Let $f$, $g_1, \cdots, g_s$ be meromorphic functions with
possible poles at $\{\mathfrak{p}_1, \cdots, \mathfrak{p}_m \}$.
Then for $n\geq 2$,
\begin{multline} \label{eq:dr:cor}
\frac{1}{(2\pi \im)^n} \oint \cdots \oint \frac{1}{(v_2-v_1)\cdots (v_n-v_{n-1})} \prod_{i=1}^n f(v_i) dv_i
\prod_{i=1}^s \left(\sum_{j=1}^n g_i(v_j)\right)\\
= \frac{n^{s-1}}{2\pi \im}\oint f(v)^n \prod_{i=1}^s g_i(v) dv ,
\end{multline}
where the contours in both sides are around all of $\{\mathfrak{p}_1, \cdots, \mathfrak{p}_m \}$,
and for the left hand side we require $|u_1| \ll \cdots \ll |u_n|$.
\end{cor}

\begin{proof}
Take disjoint sets $U_1, \cdots, U_n$, with $\bigcup_{i=1}^n U_i = \{1, \cdots, s\}$ (some of which might be empty).
In Theorem \ref{thm:dr} we let $f_{i} = f\prod_{j\in U_{i} } g_j$ for each $1 \leq i \leq n$, and get
\begin{multline}
\sum_{\sigma \in S^{cyc}(n)} \frac{1}{(2\pi \im)^n} \oint \cdots \oint \frac{1}{(v_2-v_1)\cdots (v_n-v_{n-1})} \prod_{i=1}^n \left( f(v_i) \prod_{j\in U_{\sigma(i)}} g_j(v_i) dv_i \right) \\
= \frac{1}{2\pi \im}\oint f(v)^n \prod_{i=1}^s g_i(v) dv .
\end{multline}
Summing over all $n^s$ partitions $U_1, \cdots, U_n$ of $\{1, \cdots, s\}$ into $n$ disjoint sets, we obtain (\ref{eq:dr:cor}).
\end{proof}

\bibliographystyle{halpha}
\bibliography{bibliography}

\end{document}